\documentclass[a4paper,leqno,12pt]{amsart}
\usepackage[all]{xy}
\usepackage{xspace}
\usepackage{bm}
\usepackage{amsmath}
\usepackage{amstext}
\usepackage{amsfonts}
\usepackage[mathscr]{euscript}
\usepackage{amscd}
\usepackage{latexsym}
\usepackage{amssymb}
\usepackage{mathrsfs}
\usepackage{cancel}
\usepackage{enumerate}
\usepackage{color}
\theoremstyle{plain}
\swapnumbers
    \newtheorem{thm}{Theorem}[section]
    \newtheorem{prop}[thm]{Proposition}
    \newtheorem{lemma}[thm]{Lemma}
    \newtheorem{keylemma}[thm]{Key Lemma}
    \newtheorem{corollary}[thm]{Corollary}
    
    \newtheorem{subsec}[thm]{}
    
    \newtheorem*{thma}{Theorem A}
    \newtheorem*{thmb}{Theorem B}
    \newtheorem*{thmc}{Theorem C}
    \newtheorem*{thmd}{Theorem D}
\theoremstyle{definition}
    \newtheorem{defn}[thm]{Definition}
    \newtheorem{example}[thm]{Example}
    
    \newtheorem{notation}[thm]{Notation}

            \newtheorem{caveat}[thm]{Caveat}
        \newtheorem{remark}[thm]{Remark}

        \newtheorem{assume}[thm]{Assumption}
        \newtheorem{assumes}[thm]{Assumptions}
    \newtheorem{ack}[thm]{Acknowledgements}
\newcommand{\supsect}[2]
{\vspace*{-5mm}\quad\\\begin{center}\textbf{{#1}}\vsm.~~~~\textbf{{#2}}\end{center}}
\newenvironment{myeq}[1][]
{\stepcounter{thm}\begin{equation}\tag{\thethm}{#1}}
{\end{equation}}

\newcommand{\mydiagram}[2][]
{\stepcounter{thm}\begin{equation}
     \tag{\thethm}{#1}\vcenter{\xymatrix{#2}}\end{equation}}
\newcommand{\mysdiag}[2][]
{\stepcounter{thm}\begin{equation}
     \tag{\thethm}{#1}\vcenter{\xymatrix@R=25pt@C=-25pt{#2}}\end{equation}}
\newcommand{\mytdiag}[2][]
{\stepcounter{thm}\begin{equation}
     \tag{\thethm}{#1}\vcenter{\xymatrix@R=25pt@C=20pt{#2}}\end{equation}}
\newcommand{\myudiag}[2][]
{\stepcounter{thm}\begin{equation}
     \tag{\thethm}{#1}\vcenter{\xymatrix@R=27pt@C=12pt{#2}}\end{equation}}
\newcommand{\myvdiag}[2][]
{\stepcounter{thm}\begin{equation}
     \tag{\thethm}{#1}\vcenter{\xymatrix@R=20pt@C=14pt{#2}}\end{equation}}
\newcommand{\mywdiag}[2][]
{\stepcounter{thm}\begin{equation}
     \tag{\thethm}{#1}\vcenter{\xymatrix@R=22pt@C=32pt{#2}}\end{equation}}
\newcommand{\myxdiag}[2][]
{\stepcounter{thm}\begin{equation}
    \tag{\thethm}{#1}\vcenter{\xymatrix@R=30pt@C=25
      pt{#2}}\end{equation}}
\newcommand{\myzdiag}[2][]
{\stepcounter{thm}\begin{equation}
     \tag{\thethm}{#1}\vcenter{\xymatrix@R=10pt@C=25pt{#2}}\end{equation}}
\newenvironment{mysubsection}[2][]
{\begin{subsec}\begin{upshape}\begin{bfseries}{#2.}
\end{bfseries}{#1}}
{\end{upshape}\end{subsec}}
\newcommand{\sect}{\setcounter{thm}{0}\section}
\newcommand{\wh}{\ -- \ }
\newcommand{\wwh}{-- \ }
\newcommand{\w}[2][ ]{\ \ensuremath{#2}{#1}\ }
\newcommand{\ww}[1]{\ \ensuremath{#1}}
\newcommand{\www}[2][ ]{\ensuremath{#2}{#1}\ }
\newcommand{\wb}[2][ ]{\ (\ensuremath{#2}){#1}\ }
\newcommand{\wwb}[1]{\ (\ensuremath{#1})-}
\newcommand{\wwbu}[1]{\ {\upshape(}\ensuremath{#1}{\upshape)}-}
\newcommand{\wref}[2][ ]{\ \eqref{#2}{#1}\ }

\newcommand{\hsp}{\hspace*{7 mm}}
\newcommand{\hs}{\hspace*{4 mm}}
\newcommand{\hsn}{\hspace*{1 mm}}
\newcommand{\hsm}{\hspace*{2 mm}}
\newcommand{\vsn}{\vspace{2 mm}}

\newcommand{\vsm}{\vspace{4 mm}}
\newcommand{\hra}{\hookrightarrow}

\newcommand{\xra}[1]{\xrightarrow{#1}}
\newcommand{\toto}{\substack{{\to}\\ \to}}
\newcommand{\lra}[1]{\langle{#1}\rangle}
\newcommand{\llrra}[1]{\langle\langle{#1}\rangle\rangle}
\newcommand{\llrr}[2]{\langle\langle{#1}\rangle\rangle\sb{#2}}
\newcommand{\epic}{\twoheadrightarrow}
\newcommand{\monic}{\rightarrowtail}
\newcommand{\xepic}[1]{\stackrel{#1}{\epic}}
\newcommand{\xmonic}[1]{\stackrel{#1}{\monic}}
\newcommand{\up}[1]{\sp{(#1)}}
\newcommand{\bp}[1]{\sp{[#1]}}
\newcommand{\lp}[1]{\sb{[#1]}}
\newcommand{\lo}[1]{\sb{(#1)}}
\newcommand{\rest}[1]{\lvert\sb{#1}}
\newcommand{\oS}[1]{\boxtimes S\sp{#1}}
\newcommand{\odisc}[1]{\boxtimes e\sp{#1}}
\newcommand{\Cof}{\operatorname{Cof}}
\newcommand{\Cok}{\operatorname{Cok}}
\newcommand{\Cone}{\operatorname{Cone}}

\newcommand{\CW}{\operatorname{CW}}
\newcommand{\CWA}[1]{\CW\hspace{-1.5mm}\sb{A}{#1}}
\newcommand{\colim}{\operatorname{colim}}
\newcommand{\colimit}[1]
{\raisebox{-1.7ex}{$\stackrel{\textstyle\colim}{\scriptstyle{#1}}$}}

\newcommand{\ev}{\operatorname{ev}}
\newcommand{\bev}{\mathbf{ev}}

\newcommand{\gr}{\operatorname{gr}}
\newcommand{\ho}{\operatorname{ho}}

\newcommand{\hocolim}{\operatorname{hocolim}}

\newcommand{\Hom}{\operatorname{Hom}}

\newcommand{\Ker}{\operatorname{Ker}}
\newcommand{\Id}{\operatorname{Id}}

\newcommand{\inc}{\operatorname{inc}}
\newcommand{\Moore}{\operatorname{M}}
\newcommand{\Obj}{\operatorname{Obj}\,}
\newcommand{\op}{\sp{\operatorname{op}}}
\newcommand{\Path}{\operatorname{Path}}

\newcommand{\proj}{\operatorname{proj}}

\newcommand{\sk}[1]{\operatorname{sk}\sb{#1}}
\newcommand{\SL}[1]{\operatorname{SL}\sb{#1}(\ZZ)}
\newcommand{\csk}[1]{\operatorname{csk}\sp{#1}}
\newcommand{\Sq}[1]{\operatorname{Sq}\sp{#1}}

\newcommand{\Tot}{\operatorname{Tot}}
\newcommand{\map}{\operatorname{map}}

\newcommand{\mapa}{\map\sb{\ast}}

\newcommand{\A}{\mathcal{A}}

\newcommand{\B}{\mathcal{B}}
\newcommand{\C}{\mathcal{C}}
\newcommand{\hC}{\widehat{\C}}

\newcommand{\E}{\mathcal{E}}
\newcommand{\F}{\mathcal{F}}

\newcommand{\G}{\mathcal{G}}

\newcommand{\cL}{\mathcal{L}}
\newcommand{\dgL}{\operatorname{dg}\!\cL}
\newcommand{\M}{{\mathcal M}}

\newcommand{\PP}{\mathcal{P}}
\newcommand{\Pn}[1]{\PP\sp{#1}}
\newcommand{\hP}[1]{\widehat{\PP}\sp{#1}}
\newcommand{\U}{\mathcal{U}}
\newcommand{\hU}{\widehat{\U}}
\newcommand{\ppp}{\hspace*{0.3mm}\sp{\prime}\hspace{-0.3mm}}
\newcommand{\pppb}{\hspace*{0.3mm}\sp{\prime}\hspace{0.8mm}}

\newcommand{\pppp}{\hspace*{0.5mm}\sp{\prime}\hspace{-0.7mm}}
\newcommand{\ppppp}{\hspace*{0.5mm}\sp{\prime}\hspace{-0.3mm}}
\newcommand{\ak}[1]{a\sb{#1}}

\newcommand{\dz}[1]{d\sb{0}\sp{#1}}

\newcommand{\wpar}[1]{\widehat{\partial}\sb{#1}}

\newcommand{\Eu}[1]{E\bp{#1}}
\newcommand{\Euk}[2]{E\sb{#1}\bp{#2}}
\newcommand{\Eup}[1]{\ppp E\bp{#1}}
\newcommand{\Eul}[2]{E\bp{#1}\sb{#2}}

\newcommand{\hEu}[1]{\widehat{E}\bp{#1}}

\newcommand{\hEuk}[2]{\widehat{E}\sb{#1}\bp{#2}}

\newcommand{\tEuk}[2]{\widetilde{E}\sb{#2}\bp{#1}}
\newcommand{\euk}[1]{e\sb{#1}}

\newcommand{\hek}[1]{\widehat{e}\sb{#1}}
\newcommand{\Fk}[1]{F\sb{#1}}
\newcommand{\Fkp}[1]{\pppp\Fk{#1}}

\newcommand{\tF}{\widetilde{F}}

\newcommand{\tG}{\widetilde{G}}

\newcommand{\Htr}{{\mathcal H}}
\newcommand{\Htl}[1]{\Htr\bp{#1}}

\newcommand{\Huk}[2]{H\sb{#1}\bp{#2}}
\newcommand{\Hupk}[2]{\ppp H\sb{#1}\bp{#2}}
\newcommand{\huk}[1]{h\sb{#1}}

\newcommand{\hHuk}[2]{\widehat{H}\sb{#1}\bp{#2}}

\newcommand{\tHuk}[2]{\widetilde{H}\sb{#2}\bp{#1}}

\newcommand{\vpm}{\varphi\sp{M}}
\newcommand{\psm}{\psi\sp{M}}
\newcommand{\ovp}{\overline{\varphi}}

\newcommand{\ophl}[1]{\ovp\sb{#1}}

\newcommand{\prn}[1]{\iota\bp{#1}}

\newcommand{\orh}[1]{\overline{\rho}\sb{#1}}

\newcommand{\vare}{\varepsilon}
\newcommand{\varu}[1]{\vare\up{#1}}
\newcommand{\varep}{\ppp\vare}

\newcommand{\svn}[1]{\overline{\vare\sb{[{#1}]}}}
\newcommand{\bv}{{\bm \vare}}
\newcommand{\bve}[1]{\bv\bp{#1}}

\newcommand{\bvp}[1]{\pppp\bv\bp{#1}}

\newcommand{\vth}{\vartheta}

\newcommand{\oCe}{\overline{Ce}}
\newcommand{\oee}{\overline{e}}
\newcommand{\we}{\widehat{e}}

\newcommand{\en}[1]{e\bp{#1}}
\newcommand{\eni}[2]{e\bp{#1}\lo{#2}}
\newcommand{\enk}[2]{e\bp{#1}\sb{#2}}

\newcommand{\oen}[1]{\oee\sb{#1}}

\newcommand{\oon}[2]{\oee\sb{#1}\bp{#2}}
\newcommand{\hon}[2]{\we\sb{#1}\bp{#2}}
\newcommand{\Con}[2]{\oCe\sp{#1}\sb{#2}}

\newcommand{\es}{e\sb{\#}}

\newcommand{\hen}[1]{\we\bp{#1}}
\newcommand{\hes}{\we\sb{\#}}
\newcommand{\wf}{\widehat{f}}
\newcommand{\wg}{\widehat{g}}
\newcommand{\wj}{\widehat{\mbox{\textit{\j}}}}

\newcommand{\orr}{\overline{r}}
\newcommand{\oCr}{\overline{Cr}}
\newcommand{\wrr}{\widehat{r}}
\newcommand{\rn}[1]{r\bp{#1}}
\newcommand{\rnk}[2]{\rn{#1}\sb{#2}}
\newcommand{\rni}[2]{\rn{#1}\lo{#2}}
\newcommand{\oar}[1]{\orr\,\sb{#1}}

\newcommand{\oor}[2]{\orr\sb{#1}\bp{#2}}
\newcommand{\hor}[2]{\wrr\sb{#1}\bp{#2}}
\newcommand{\Cor}[2]{\oCr\sp{#1}\sb{#2}}

\newcommand{\rs}{r\sb{\#}}

\newcommand{\hrn}[1]{\wrr\bp{#1}}

\newcommand{\snk}[2]{\overline{s}\bp{#1}\sb{#2}}

\newcommand{\Abgp}{\mbox{\sf Abgp}}

\newcommand{\Ch}{\mbox{\sf Ch}}
\newcommand{\ChC}{\Ch\sb{\C}}

\newcommand{\cS}{{\EuScript S}}
\newcommand{\Sa}{\cS\sb{\ast}}

\newcommand{\Gp}{\mbox{\sf Gp}}
\newcommand{\Set}{\mbox{\sf Set}}
\newcommand{\Seta}{\Set\sb{\ast}}
\newcommand{\Top}{\mbox{\sf Top}}
\newcommand{\Topa}{\Top\sb{\ast}}

\newcommand{\Vect}{\mbox{\sf Vect}}
\newcommand{\VF}{\Vect\sb{\Fp}}
\newcommand{\Del}{\mathbf{\Delta}}
\newcommand{\Deln}[1]{\Del\sp{#1}}
\newcommand{\Delnk}[2]{\Deln{#1}\lo{#2}}
\newcommand{\Dop}{\Delta\op}
\newcommand{\res}{\operatorname{res}}
\newcommand{\Dres}{\Delta\sb{\res}}
\newcommand{\Dp}{\Delta\sb{+}}
\newcommand{\Dresp}{\Delta\sb{\res,+}}
\newcommand{\Du}{\Del\sp{\bullet}}

\newcommand{\Val}[1]{\operatorname{Val}({#1})}
\newcommand{\Vals}[1]{\operatorname{Vals}[{#1}]}

\newcommand{\cW}{{\EuScript W}}
\newcommand{\cWp}{\pppb\!\cW}
\newcommand{\cWpp}{\ppp\!\cWp}
\newcommand{\cWpi}[1]{\ppppp\cW\up{#1}}

\newcommand{\ccWp}{\sp{\prime}\cW}
\newcommand{\cuW}[1]{\cW\up{#1}}

\newcommand{\Phip}[1]{\widehat{\Phi}\up{#1}}
\newcommand{\FF}{\mathbb F}
\newcommand{\Fp}{\FF\sb{p}}
\newcommand{\Ft}{\FF\sb{2}}

\newcommand{\LL}[1]{{\mathbb L}\lra{#1}}
\newcommand{\NN}{\mathbb N}
\newcommand{\QQ}{\mathbb Q}
\newcommand{\RR}{\mathbb R}
\newcommand{\ZZ}{\mathbb Z}
\newcommand{\bA}{\mathbf{A}}
\newcommand{\bB}{\mathbf{B}}

\newcommand{\bC}{\mathbf{C}}
\newcommand{\bD}{\mathbf{D}}
\newcommand{\bE}{\mathbf{E}}
\newcommand{\be}[1]{\mathbf{e}\sp{#1}}
\newcommand{\bK}{\mathbf{K}}
\newcommand{\KP}[2]{\bK({#1},{#2})}
\newcommand{\Kp}[1]{\KP{\Fp}{#1}}

\newcommand{\KZ}[1]{\KP{\ZZ}{#1}}

\newcommand{\bP}{\mathbf{P}}
\newcommand{\bS}[1]{\mathbf{S}\sp{#1}}
\newcommand{\bT}{\mathbf{T}}
\newcommand{\bU}{\mathbf{U}}
\newcommand{\bV}{\mathbf{V}}
\newcommand{\bW}{\mathbf{W}\!}
\newcommand{\bX}{\mathbf{X}}
\newcommand{\Xun}[2]{\bX\sb{#1}\up{#2}}
\newcommand{\Xnd}[1]{\Xun{\bullet}{#1}}
\newcommand{\bY}{\mathbf{Y}}
\newcommand{\bYp}{\ppp\bY}

\newcommand{\Yi}[1]{\bY\up{#1}}
\newcommand{\bZ}{\mathbf{Z}}
\newcommand{\ZI}{\Path(\bZ)}

\newcommand{\cpd}[1]{c\sb{+}({#1})\sb{\bullet}}
\newcommand{\cd}[1]{c({#1})\sb{\bullet}}
\newcommand{\Ad}{A\sb{\bullet}}

\newcommand{\oA}[1]{\overline{A}\sb{#1}}
\newcommand{\Bd}{B\sb{\bullet}}
\newcommand{\oB}{\overline{\bB}}

\newcommand{\Cd}{C\sb{\bullet}}

\newcommand{\fu}[1]{f\up{#1}}
\newcommand{\vfu}[1]{\varphi\up{#1}}

\newcommand{\Gnk}[2]{G\bp{#1}\sb{#2}}
\newcommand{\Gn}[1]{\Gnk{#1}{}}
\newcommand{\Gank}[2]{\Gamma\bp{#1}\sb{#2}}
\newcommand{\Gan}[1]{\Gank{#1}{}}
\newcommand{\tGank}[2]{\widehat{\Gamma}\bp{#1}\sb{#2}}
\newcommand{\ttGank}[2]{\widetilde{\Gamma}\bp{#1}\sb{#2}}

\newcommand{\Gd}{G\sb{\bullet}}

\newcommand{\oG}[1]{\overline{G}\sb{#1}}

\newcommand{\Td}{\bT\sb{\bullet}}
\newcommand{\oU}[1]{\overline{U}\sb{#1}}
\newcommand{\Ud}{\bU\sb{\bullet}}

\newcommand{\Vd}{V\sb{\bullet}}
\newcommand{\Vpn}[1]{\ppp V\sb{#1}}
\newcommand{\Vdp}{\Vpn{\bullet}}
\newcommand{\Vud}[1]{\Vd\up{#1}}
\newcommand{\oV}[1]{\overline{V}\sb{#1}}
\newcommand{\oVp}[1]{\ppp\oV{#1}}
\newcommand{\ouV}[2]{\oV{#1}\up{#2}}
\newcommand{\Wd}{\bW\sb{\bullet}}
\newcommand{\Wn}[2]{\bW\sb{#1}\bp{#2}}

\newcommand{\W}[1]{\Wn{\bullet}{#1}}

\newcommand{\Wpn}[2]{\ppp\bW\sb{#1}\bp{#2}}
\newcommand{\Wp}[1]{\Wpn{\bullet}{#1}}

\newcommand{\Wni}[3]{\bW\sb{#1}\sp{[{#2}]({#3})}}
\newcommand{\Wi}[2]{\Wni{\bullet}{#1}{#2}}

\newcommand{\tW}{\widetilde{\bW}}
\newcommand{\tWn}[2]{\tW\quad\hspace*{-4mm}\sb{#1}\bp{#2}}

\newcommand{\tWd}[1]{\tWn{\bullet}{#1}}

\newcommand{\oW}[1]{\overline{\bW}\sb{#1}}
\newcommand{\oWp}[1]{\ppp\oW{#1}}

\newcommand{\hW}[1]{\widehat{\bW}\sb{#1}}
\newcommand{\hWd}[1]{\widehat{\bW}\sb{\bullet}\bp{#1}}
\newcommand{\hWdn}[2]{\widehat{\bW}\sb{#2}\bp{#1}}
\newcommand{\hWp}[1]{\ppp\hW{#1}}
\newcommand{\vW}[1]{\widehat{\bW}\sb{#1}}
\newcommand{\vWn}[2]{\vW{#1}\bp{#2}}
\newcommand{\vWu}[1]{\vWn{\bullet}{#1}}
\newcommand{\unWn}[1]{\underline{\bW}\;\sb{#1}}
\newcommand{\unW}{\unWn{\bullet}}
\newcommand{\unX}{\underline{\bX}}
\newcommand{\Xu}{\bX\sp{\bullet}}
\newcommand{\Xn}[2]{\bX\sp{#1}\lp{#2}}

\newcommand{\vXu}[1]{\widehat{\bX}\sp{\bullet}\lp{#1}}
\newcommand{\X}[1]{\Xn{\bullet}{#1}}
\newcommand{\Yd}{\bY\sb{\bullet}}
\newcommand{\Zd}{\bZ\sb{\bullet}}

\newcommand{\Alg}[1]{{#1}\text{-}{\EuScript Alg}}
\newcommand{\Pa}[1][ ]{$\Pi$-algebra{#1}}
\newcommand{\RPa}[1][ ]{$\Pi\sb{R}$-algebra{#1}}
\newcommand{\PAal}[1][ ]{$\PA$-algebra{#1}}

\newcommand{\pis}{\pi\sb{\ast}}
\newcommand{\piA}{\pis\sp{\A}}
\newcommand{\PA}{\Pi\sb{\A}}
\newcommand{\PAlg}{\Alg{\PA}}

\newcommand{\His}{H\sb{\ast}}

\newcommand{\ua}[1]{\mbox{\b{a}}\sb{#1}}
\newcommand{\ub}[1]{\mbox{\b{b}}\sb{#1}}
\newcommand{\uc}[1]{\mbox{\b{c}}\sb{#1}}
\newcommand{\ue}[1]{\mbox{\b{e}}\sb{#1}}
\newcommand{\uw}[1]{\mbox{\b{w}}\sb{#1}}
\newcommand{\ow}[1]{\mbox{\'{w}}\sb{#1}}
\newcommand{\oow}[1]{\mbox{\H{w}}\sb{#1}}
\newcommand{\hw}[1]{\mbox{\^{w}}\sb{#1}}
\newcommand{\hhw}[1]{\mbox{\u{w}}\sb{#1}}
\newcommand{\ux}[1]{\mbox{\b{x}}\sb{#1}}
\newcommand{\ox}[1]{\mbox{\'{x}}\sb{#1}}
\newcommand{\oox}[1]{\mbox{\H{x}}\sb{#1}}
\newcommand{\uy}[1]{\mbox{\b{y}}\sb{#1}}
\newcommand{\oy}[1]{\mbox{\'{y}}\sb{#1}}
\newcommand{\ooy}[1]{\mbox{\H{y}}\sb{#1}}
\newcommand{\uz}[1]{\mbox{\b{z}}\sb{#1}}
\newcommand{\oz}[1]{\mbox{\'{z}}\sb{#1}}
\newcommand{\ooz}[1]{\mbox{\H{z}}\sb{#1}}
\newcommand{\lin}[1]{\{{#1}\}}

\newcommand{\bbn}{[\mathbf{n}]}
\newcommand{\bbk}{[\mathbf{k}]}
\newcommand{\oCB}{\overline{C\bB}}

\newcommand{\osB}[1]{\overline{\Sigma\sp{#1}\bB}}
\newcommand{\oCsB}[1]{\overline{C\Sigma\sp{#1}\bB}}

\newcommand{\CsoW}[2]{C\Sigma\sp{#1}\overline{\bW\sb{#2}}} 
\newcommand{\osW}[2]{\overline{\Sigma\sp{#1}\bW\sb{#2}}}

\newcommand{\oCsW}[2]{\overline{C\Sigma\sp{#1}\bW\sb{#2}}}
\newcommand{\osWp}[2]{\overline{\Sigma\sp{#1}\mbox{$\ppp\bW$}\sb{#2}}}
\newcommand{\oCsWp}[2]{\overline{C\Sigma\sp{#1}\mbox{$\ppp\bW$}\sb{#2}}}

\newcommand{\Xd}{\bX\sb{\bullet}}
\newcommand{\oX}[1]{\overline{\bX}\sb{#1}}
\newcommand{\osX}[2]{\overline{\Sigma\sp{#1}\bX\sb{#2}}}

\newcommand{\hX}[1]{\widehat{\bX}\sb{#1}}
\newcommand{\hXd}{\hX{\bullet}}

\newcommand{\ooq}{\overline{q}}

\newcommand{\oq}[1]{\ooq\sp{#1}}

\newcommand{\oi}[1]{\overline{\iota}\sp{#1}}

\newcommand{\As}{A\sb{\ast}}

\newcommand{\Bs}{B\sb{\ast}}

\newcommand{\bCs}{\bC\sb{\ast}}
\newcommand{\cM}[1]{C\sp{\Moore}\sb{#1}}
\newcommand{\cMs}{\cM{\ast}}
\newcommand{\Ds}{D\sb{\ast}}
\newcommand{\Dsn}[1]{\bD\sb{\ast}\sp{[#1]}}
\newcommand{\Dsnp}[1]{\ppp\bD\sb{\ast}\sp{[#1]}}
\newcommand{\bDs}{\bD\sb{\ast}}
\newcommand{\uDs}[1]{\bD\sb{\ast}\up{#1}}
\newcommand{\sDd}[1]{\Sigma\bD\sb{\bullet}\bp{#1}}
\newcommand{\Es}{E\sb{\ast}}
\newcommand{\sEd}[1]{\Sigma\bE\sp{\bullet}\lp{#1}}
\newcommand{\bEs}{\bE\sb{\ast}}
\newcommand{\uEs}[1]{\ppp\bEs\up{#1}}
\newcommand{\bPs}{\bP\sb{\ast}}
\newcommand{\bGs}{\mathbf{G}\sb{\ast}}
\newcommand{\uPs}[1]{\bPs\up{#1}}
\newcommand{\Fn}[1]{F\bp{#1}}
\newcommand{\tFn}[1]{\tF\bp{#1}}
\newcommand{\Tn}[1]{T\bp{#1}}
\newcommand{\cZ}[1]{Z\sp{\Moore}\sb{#1}}

\newcommand{\df}[1]{\partial\sb{#1}}
\newcommand{\wdf}[1]{\widehat{\partial}\sb{#1}}
\newcommand{\diff}[2]{\df{#1}\sp{#2}}
\newcommand{\wdiff}[2]{\wdf{#1}\sp{#2}}
\newcommand{\difu}[2]{\df{#1}\up{#2}}
\newcommand{\dif}[1]{\df{#1}\sp{\Moore}}
\newcommand{\od}{\overline{\partial}}
\newcommand{\ud}{\overline{d}}

\newcommand{\odz}[1]{\od\sp{#1}\sb{0}}
\newcommand{\udz}[1]{\ud\sb{0}\sp{#1}}

\newcommand{\szero}[1]{\sigma\sp{\ast}(#1)}

\newcommand{\ofn}[2]{\overline{m}^{[#1]}_{#2}}

\begin{document}
%
%
\title{Higher homotopy invariants for spaces and maps}
%
%
\author[D.~Blanc]{David Blanc}
\address{Department of Mathematics\\ University of Haifa\\ 3498838 Haifa\\ Israel}
\email{blanc@math.haifa.ac.il}
\author [M.W.~Johnson]{Mark W.~Johnson}
\address{Department of Mathematics\\ Penn State Altoona\\
                  Altoona, PA 16601\\ USA}
\email{mwj3@psu.edu}
\author[J.M.~Turner]{James M.~Turner}
\address{Department of Mathematics\\ Calvin University\\
         Grand Rapids, MI 49546\\ USA}
\email{jturner@calvin.edu}
\date{\today}

\begin{abstract}
For a pointed topological space $\bX$, we use an inductive construction of a
simplicial resolution of $\bX$ by wedges of spheres to construct a ``higher homotopy
structure'' for $\bX$ (in terms of chain complexes of spaces). This structure is
then used to define a collection of higher homotopy invariants which suffice to
recover $\bX$ up to weak equivalence. It can also be used to distinguish between
different maps \w{f:\bX\to\bY} which induce the same morphism
\w[.]{f\sb{\ast}:\pis\bX\to\pis\bY}
\end{abstract}

\maketitle

\setcounter{section}{0}

%
%
\section*{Introduction}
\label{cint}

We describe two sequences of higher order operations constituting complete invariants
for the homotopy type of a topological space or map, respectively.

Higher homotopy and cohomology operations, such as Massey products and Toda brackets,
are among the earliest known examples of homotopy invariants which are not primary.
They have played an important computational role in algebraic topology (see, e.g.,
\cite{AdHI,TodC}). However, no truly satisfactory theory of general higher
homotopy operations has been proposed so far, despite several attempts
(see, e.g., \cite{SpanS,SpanH}). Here we follow the point of view taken in
\cite{BMarkH,BJTurnHH}, where more precise definitions are given.

\begin{mysubsection}{Higher homotopy operations}
\label{shho}
A higher homotopy operation is an obstruction to rectifying a homotopy
commutative diagram \w{\unX:\Gamma\to\ho\C} in some pointed model category $\C$,
where $\Gamma$ is a finite directed category with
a weakly initial object \w{v\sb{i}} and weakly final object \w[.]{v\sb{f}}
When the longest composable sequence in $\Gamma$
has length \w[,]{n+1} we have an \emph{$n$-th order} operation, with a value in
\w[.]{[\Sigma\sp{n-1}\unX(v\sb{i}),\,\unX(v\sb{f})]}
The obstructions are constructed by induction on initial (or terminal) subdiagrams
$I$ of $\Gamma$ of increasing length: if the $k$-th order obstruction vanishes,
we choose a rectification for the appropriate subdiagram, which allows us to
\emph{define} the \wwb{k+1}st order obstruction. The various choices made along
the way contribute to the \emph{indeterminacy} of the operation: we say that a
\wwb{k+1}st order operation \emph{vanishes} if the obstruction does so for
some such choice. See \cite{BMarkH} and \cite[\S 3]{BJTurnHH} for more details.

In general there is more than one way to define obstructions for a given
rectification problem. The point of view espoused here is that any two constructions
of higher order operations which yield the same answer at each stage are considered
to be equivalent.

In this paper we consider higher \emph{homotopy}  operations in the narrower sense,
where $\C$ is some model for \w[,]{\Topa} and all spaces \w{\unX(v)} (except perhaps
\w[)]{\unX(v\sb{f})} are wedges of spheres. The values of such an operation thus
indeed lie in \w[,]{\pis\unX(v\sb{f})} and in fact the whole diagram $\unX$ can be
described as a collection of elements in \w[,]{\pis\unX(v\sb{f})} which vanish
under the action of certain primary homotopy operations, together with a system of
(higher) relations among such primary operations.
\end{mysubsection}

\begin{mysubsection}{Linear higher homotopy operations}
\label{slhho}
We shall not need any more of the general theory, but we briefly sketch the linear
case, also known as a \emph{long Toda bracket} (see \S \ref{slont} below). This is
not quite the version we need here, but is simpler to describe, and best conveys the
basic ideas we use.

Start with a ``chain complex in \w{\ho\C}'' \wh that
is, a finite sequence of maps
\begin{myeq}\label{eqhochain}
\bX\sb{n}~\xra{\df{n}}~\bX\sb{n-1}~\to\dotsc \to\bX\sb{1}~\xra{\df{1}}~\bX\sb{0}
\end{myeq}
\noindent in a pointed simplicial model category $\C$
with \w{\df{k-1}\circ\df{k}\sim 0} for every \w[.]{n \geq k > 1}
\emph{Rectifying} this homotopy-commutative diagram means replacing each space and map
\w{\dif{k}:\bX\sb{k}\to\bX\sb{k-1}} by a weakly equivalent
\w[,]{\ppp\df{k}:\ppp\bX\sb{k}\to\ppp\bX\sb{k-1}} with
\w{\ppp\df{k-1}\circ\ppp\df{k}} now actually zero.

In line with the general approach of \S \ref{shho}, we use a double induction to
try to rectify \wref[:]{eqhochain}
in the outer (ascending) induction on \w[,]{n\geq 2} we use the vanishing of the
\wwb{n-1}st order Toda bracket to rectify \wref{eqhochain}
through dimension $n$.  In the inner (descending) induction on
\w[,]{1\leq k\leq n} we calculate the next Toda bracket, corresponding to
the final segment of length \w[.]{k+1}

The simplest case is \w[,]{n=2} where there is no obstruction (and thus
no descending induction): by changing \w{\df{1}:\bX\sb{1}\to\bX\sb{0}}
into a fibration \w[,]{\difu{1}{1}:\Xun{1}{1}\to\Xun{0}{1}:=\bX\sb{0}}
a standard model category argument shows that we can then choose
\w{\difu{2}{1}:\Xun{2}{1}:=\bX\sb{2}\to\Xun{0}{1}} so that
\w{\difu{1}{1}\circ\difu{2}{1}=0} (see \cite[Lemma 5.11]{BJTurnR}).

In the $n$-th stage of the outer induction, we assume not only that we have
rectified \wref{eqhochain} through dimension \w[,]{n-1} but also that we have
made it into a fibrant \wwb{n-1}truncated chain complex \w{\Xnd{n-1}}
in $\C$ (in the injective model category structure). This means that if we write
\w[,]{Z\sb{n}\Xnd{n-1}:=\Ker(\df{n})} and use the rectification to
factor \w{\Xnd{n-1}} through dimension \w{n-1} as
\myvdiag[\label{eqmodch}]{
\Xun{n-1}{n-1}~\ar@{->>}[r]\sp(0.5){\wpar{n-1}} \ar@/^{2.5pc}/[rr]\sp{\df{n-1}}&
Z\sb{n-2}\Xnd{n-1}~\ar@{^{(}->}[r]\sp(0.6){v\sb{n-2}}~&
\Xun{n-2}{n-1}~\ar@{->>}[r]\sp(0.45){\wpar{n-2}} \ar@/^{2.5pc}/[rr]\sp{\df{n-2}}&
Z\sb{n-3}\Xnd{n-1}~\ar@{^{(}->}[r]^-{v\sb{n-3}}~&
{\Xun{n-3}{n-1}}\ar@{->>}[r]^-{\wpar{n-3}} & \cdots \Xun{0}{n-1},
}
\noindent then we require each \w{\wpar{k}} to be a fibration, so that
\mydiagram[\label{eqstfibseq}]{
Z\sb{k}\Xnd{n-1} \ar@{^{(}->}[rr]\sp{v\sb{k}} &&
\Xun{k}{n-1} \ar@{->>}[rr]\sp{\wpar{k}} && Z\sb{k-1}\Xnd{n-1}
}
\noindent is a (strict) fibration sequence \wb[.]{1\leq k<n}

Now we choose a nullhomotopy \w[;]{F\sb{n}:\df{n-1}\circ\df{n}\sim 0}
if we could lift it to a nullhomotopy
\w[,]{\widehat{F}\sb{n}:\wpar{n-1}\circ\df{n}\sim 0} we
would be done (by the case \w[).]{n=2} However, in any case we see that
\w{\wpar{n-2}\circ\Fk{n}} is a self nullhomotopy
\w[,]{0=\wpar{n-2}\circ\df{n-1}\circ\df{n}\sim 0} so it
induces a map \w[.]{\ak{n-1}:\Sigma\Xun{n}{n-1}\to\Xun{n-3}{n-1}} This is in fact
a value of the ordinary Toda bracket
\w[.]{\lra{\df{n-2},\,\df{n-1},\,\df{n}}} If \w{\ak{n-1}} is nullhomotopic, we
choose a nullhomotopy \w[,]{\Fk{n-1}:\ak{n-1}\sim 0} and again
see that \w{\wpar{n-3}\circ\Fk{n-1}} is a self nullhomotopy
so it induces
a map \w[,]{\ak{n-2}:\Sigma\sp{2}\Xun{n}{n-1}\to\Xun{n-4}{n-1}} which is a value
of the tertiary Toda bracket \w[.]{\lra{\df{n-3},\df{n-2},\df{n-1},\df{n}}}
If \w{\ak{n-1}} is not nullhomotopic for any choice of \w[,]{\Fk{n}} we cannot
proceed any further, and must backtrack to choose a different rectification
of a shorter final segment of \wref[.]{eqhochain}

As long as the intermediate Toda brackets vanish, we can proceed, until we end up
with the last obstruction, which is the \emph{\wwb{n-1}st order Toda bracket}
\begin{myeq}\label{eqlongtoda}
\lra{\df{1},\dotsc,\df{n-2},\df{n-1},\df{n}}~
\subseteq~[\Sigma\sp{n-2}\Xun{n}{n-1},\ \Xun{0}{n-1}]~.
\end{myeq}
\noindent A more precise description of the process is given in \S \ref{slont} below.

In this paper we elaborate on the idea, first enunciated in \cite{BlaAI}
(see also \cite{BJTurnHA}), that there is a \emph{complete} set of invariants
for weak homotopy types of spaces consisting of higher homotopy operations.
The main improvements on previous results are:

\begin{enumerate}
\renewcommand{\labelenumi}{(\alph{enumi})~}
\item Using higher order operations which are \emph{linear} \wh in the sense of requiring
  a single choice in each simplicial dimension \wh rather than the more complicated
simplicial operations of \cite{BlaAI,BJTurnHA};
\item Making precise the relation between the vanishing of the \wwb{n-1}st
order operations and our ability to define the $n$-th order operation.
\item Explaining how the higher operations based on different algebraic resolutions
are related.
\item Constructing a similar set of invariants for maps.
\end{enumerate}
\end{mysubsection}

\begin{mysubsection}{Main results}\label{smainresults}
We can now describe the most significant results of this paper. For simplicity we
state them here for our main motivating example \wh the usual homotopy groups
\w{\pis\bY} of a pointed connected topological space, with their \Pa structure
coming from the action of the primary homotopy operations on them \wh although
in fact we prove them in a more general model category setting.

We start with two technical facts which play a central role in the proofs\vsn:

\begin{thma}
Any resolution \w{\Vd} of the \Pa \w{\pis\bY} can be realized by
an augmented simplicial space \w[,]{\Wd\to\bY} with each
\w{\bW\sb{n}} a wedge of spheres, obtained as the limit of a sequence of
$n$-truncated approximations \w[.]{\lra{\W{n}}\sb{n\in\NN}}
\end{thma}
\noindent See Theorem \ref{tres} below\vsn .

We call the system of successive approximations
\w{\cW=\lra{\W{n}}\sb{n\in\NN}} a \emph{sequential realization}
of \w{\Vd} for $\bY$ (see \S \ref{dsrar}). We then prove\vsn:

\begin{thmb}
Any two sequential realizations $\cW$ and \w{\cWp} of two CW resolutions
\w{\Vd} and \w{\Vdp} for the same space $\bY$ are connected by a zigzag of
split weak equivalences.
\end{thmb}
\noindent See Theorem \ref{tzigzag} below\vsn.

This allows us to compare the system of higher operations \w{\llrra{\bY}}
associated to different sequential realizations for $\bY$, and then show that
we can use any one such $\cW$ to determine their vanishing\vsn:

\begin{thmc}
Given an abstract isomorphism of \Pa[s] \w[,]{\vth:\pis\bY\to\pis\bZ}
the associated system of higher homotopy operations vanishes coherently for
\emph{some} sequential realization $\cW$ for $\bY$ if and only if
it does so for \emph{every} sequential realization if and only
if $\bY$ and $\bZ$ are weakly equivalent.
\end{thmc}
\noindent See Theorem \ref{tvanish} below\vsn.

By extending the ideas sketched above, one can use any sequential realization
for $\bY$ to define a system of higher homotopy operations associated to any two
maps \w{\fu{0},\fu{1}:\bY\to\bZ} which induce the same map in \w[,]{\pis} and
show:

\begin{thmd}
If $\bY$ and $\bZ$ are CW complexes, the system of higher operations associated
to \w{\fu{0},\fu{1}:\bY\to\bZ} as above vanishes if and
only if \w{\fu{0}} and \w{\fu{1}} are homotopic.
\end{thmd}
\noindent See Theorem \ref{tvanishmap} below\vsn.

To illustrate our methods, in Section \ref{cfii} we define a filtration index invariant
for mod $p$ cohomology classes, dual to the Adams filtration on homotopy groups,
and show how it may be interpreted in terms of certain higher homotopy operations
using a reverse Adams spectral sequence.
\end{mysubsection}

\begin{mysubsection}{Main techniques}\label{smaintechniques}
As explained above, the ``deconstruction'' of a space $\bY$ (or map) into its
constituent higher order structure is carried out inductively, using a sequence
of  (finite) approximations to a simplicial resolution of $\bY$.

However, simplicial techniques tend to be rather complicated, and the main
technical tool we shall be using is a sort of ``Dold-Kan correspondence for spaces'',
which allows us to do the heavy work in the inductive step
using \emph{chain complexes} of spaces (see Section \ref{cback}.B below).
As one might expect, the passage from simplicial objects to chain complexes is
straightforward, using Moore chains (see \S \ref{dmco}). The reverse direction is
functorial, and thus can be thought of as a formal black box (in which we lose
the ability to describe the resulting simplicial object explicitly).

Nevertheless, the first step in the reverse passage, in which we simply
replace a chain complex by the corresponding restricted simplicial object
(with higher faces zero and no degeneracies, and thus no change in the individual spaces)
is completely explicit, and contains precisely the information needed to fully
describe our higher homotopy operations.
\end{mysubsection}

\begin{notation}\label{snac}
Let $\Delta$ denote the category of non-empty finite ordered sets and order-preserving
maps (see \cite[\S 2]{MayS}), and \w{\Dres} the subcategory with the same
objects, with only monic maps. Similarly, $\Dp$ denotes the category of all
finite ordered sets (and order-preserving maps), and \w{\Dresp} the corresponding
subcategory of monic maps. A \emph{simplicial object} \w{\Gd}
in a category $\C$ is a functor \w[,]{\Dop\to\C} a \emph{restricted}
simplicial object is a functor \w[,]{\Dres\op\to\C} while
an \emph{augmented simplicial object} is a functor \w[,]{\Dp\op\to\C}
and a \emph{restricted augmented}
simplicial object is a functor \w[.]{\Dresp\op\to\C}
We write \w{G\sb{n}} for the value of \w{\Gd} at \w[.]{\bbn=(0<1<\dotsc<n)}
There is a natural embedding
\w[,]{\cd{-}:\C\to\C\sp{\Dop}} with \w{\cd{A}} the constant simplicial object
and similarly \w{\cpd{A}} for the constant augmented simplicial object.
The inclusion of categories \w{\sigma:\Delta \to \Dp} induces a functor
\w{\szero{-}:\C\sp{\Dp\op} \to \C\sp{\Dop}} forgetting the augmentation.

The category of compactly generated Hausdorff spaces (see \cite{SteCC}
and \cite[\S 7.10.1]{PHirM}), called simply \textit{topological spaces},
will be denoted by \w[,]{\Top} that of pointed topological spaces by \w[,]{\Topa}
and that of pointed connected topological spaces by \w[.]{\Top\sb{0}}

The category of simplicial sets will be denoted by \w[,]{\cS=\Set\sp{\Dop}} that of
pointed simplicial sets by \w[,]{\Sa=\Seta\sp{\Dop}} and that of simplicial groups
by \w{\G=\Gp\sp{\Dop}} (see \cite[I, \S 3]{GJarS}).

For maps \w[,]{f:A\to X} \w[,]{g:B\to X} and \w{h:A\to Y} in any (co)complete
category $\C$, we denote by \w{f\bot g:A\amalg B\to X} and
\w[,]{f\top h:A\to X\times Y} respectively, the induced maps from the coproduct
and into the product, and by \w{\inc\sb{A}:A\to A\amalg B} the inclusion.
\end{notation}

\begin{caveat}\label{cave}
The general results (though not the examples nor the application in Section \ref{cfii})
are for the most part Eckmann-Hilton dual to those of \cite{BSenH} and \cite{BBSenH}.
Nevertheless, we feel that they deserve a separate treatment, since
\begin{enumerate}[(a)]
\item this duality is not formal, and the differences need to be spelled out
carefully;
\item a great deal of work is needed to translate these results to the dual setting,
even where the approach is the same; and
\item the invariants here apply to arbitrary weak homotopy types of (connected)
spaces, rather than just to $R$-types of $R$-good spaces for \w{R=\Fp} or $\QQ$.
Therefore, they can potentially be extended to topologically enriched categories.
\end{enumerate}
\end{caveat}

\begin{ack}
We wish to thank the referee for his or her detailed and pertinent comments.
The first author was supported in part by Israel Science
Foundation Grant No.~770/16 and the third author by National Science Foundation
grant DMS-1207746.
\end{ack}

%
%
\sect{Background}
\label{cback}

We first set up the framework in which our theory works, and recall some basic
facts and constructions about simplicial objects and algebraic theories.

\begin{assume}\label{sass}
Throughout this paper we work in a cellular pointed simplicial model
category $\C$ (see \cite[\S 9.1, 11.1, \& 12.1.1]{PHirM}) with
functorial factorizations (see \cite[\S 1.1.1]{HovM}), and assume all objects
in $\C$ are fibrant (so $\C$ is right proper, by \cite[13.1.3]{PHirM}).
The main examples we have in mind are \w{\C=\Topa} or $\G$ (see \cite[II, \S 3]{QuiH}
and  \cite[\S 11.1.9, 13.1.11]{PHirM}), but in \S \ref{srex}
below we also consider the category \w{\dgL} of differential graded Lie algebras
over $\QQ$.

In such a category $\C$ we define the standard \emph{cone} \w{C\bX} and
\emph{suspension} \w{\Sigma\bX} of a (cofibrant) object $\bX$ by the pushouts of
\w{\ast\leftarrow\bX\hra\bX\otimes\Deln{1}} and \w[,]{\ast\leftarrow\bX\hra C\bX}
respectively, with \w{q: C\bX \to \Sigma\bX} the induced map.
\end{assume}

\begin{defn}\label{dconcat}
Given (cofibrant) $\bX$ and $\bY$ in such a model category $\C$, and maps
\w{G: C\bX\to \bY} and \w[,]{\gamma:\Sigma\bX\to\bY}
note that the cofibration sequence \w{\bX\hra\bX\otimes\Deln{1}\to C\bX}
induces a coaction \w{\psi:C\bX\to C\bX\vee \Sigma \bX}
(see \cite[I, 3.5ff.]{QuiH}). The {\it concatenation}
\w{G \star(\gamma\circ q):C\bX\to\bY} is then defined to be
\w[.]{(G\bot\gamma)\circ\psi}
\end{defn}

\supsect{\protect{\ref{cback}}.A}{\PAal[s]}

Let \w{\bA=\Sigma\bA'} be a fixed cofibrant suspension (and thus a homotopy cogroup
object) in a pointed model category $\C$ as in \S \ref{sass}. Denote by $\A$
the full sub-category of \w{\ho\C} generated by $\bA$ under suspensions and
arbitrary coproducts (so all objects in $\A$ may be assumed cofibrant in $\C$), and by
\w{\PA} the full sub-category of $\A$ consisting of all coproducts of cardinality
\www[,]{<\kappa} for a given
limit cardinal $\kappa$ (needed in order to guarantee that \w{\PA} is small, so the functor
categories from it are well-behaved).

\begin{defn}\label{dpaal}
A \emph{\PAal} is a product-preserving functor \w{\Lambda:\PA\op\to\Seta} (where the
products in \w{\PA\op} are the coproducts of $\C$), and the category of such is
denoted by \w[.]{\PAlg} We write \w{\Lambda\lin{\bB}} for
the value of $\Lambda$ at \w[.]{\bB\in\PA} There is a forgetful functor
\w{\hU:\PAlg\to\gr\Seta} to the category of non-negatively graded
pointed sets, with \w[.]{\hU(\Lambda)\sb{k}:=\Lambda\lin{\Sigma\sp{k}\bA}}
A \emph{free} \PAal is one in the image of the left adjoint of $\hU$,
denoted by \w[.]{\F:\gr\Seta\to\PAlg}

For each \w{\bY\in\C} we have a \emph{realizable} \PAal \w[,]{\piA\bY}
defined by setting \w{(\piA\bY)\lin{\bB}:=[\bB,\bY]} for each \w[.]{B\in\PA}
We say that a map \w{f:\bX\to\bY} in $\C$ is an $\bA$-\emph{equivalence} if the
induced map \w{f\sb{\#}:\piA\bX\to\piA\bY} is an isomorphism in \w{\PAlg} \wwh
or equivalently, if \w{f\sb{\#}:\map\sb{\C}(\bA,\bX)\to\map\sb{\C}(\bA,\bY)} is a
weak equivalence of pointed simplicial sets.

In particular, any \PAal of the form \w{\piA\bB} for \w{\bB\in\Obj\PA\subseteq\Obj\C}
is free, as is any coproduct of such. However, we make the additional assumption
that for any \w{\bB=\coprod\sb{i\in I}\,\Sigma\sp{n\sb{i}}\bA\in\A} and \w[,]{k\geq 0}
we have a natural isomorphism
\begin{myeq}\label{eqfreepa}
[\Sigma\sp{k}\bA,\,\bB]\sb{\C}~\cong~\colim\sb{\bB'}[\Sigma\sp{k}\bA,\,\bB']\sb{\C}~,
\end{myeq}
\noindent where the colimit is taken over all sub-coproducts
\w{\bB'=\coprod\sb{i\in I'}\,\Sigma\sp{n\sb{i}}\bA} with \w{I'\subseteq I} of
cardinality \www{<\kappa} (so that \w[).]{\bB'\in\PA} This implies that
\w[,]{\piA\bB=\coprod\sb{i\in I}\,\piA\Sigma\sp{n\sb{i}}\bA}
(as a coproduct in \w[),]{\PAlg} so \w{\piA\bB} is free
for all \w{\bB\in\A} (see \S \ref{egmapalg} below for a specific example).

\end{defn}

We can use the fact that a \PAal $\Lambda$ preserves products
in \w{\A\op} to define \w{\Lambda\lin{\bB}} for any \w[.]{\bB\in\A} The Yoneda Lemma
then implies:

\begin{lemma}\label{lfreeta}
If $\Lambda$ is any \PAal and \w[,]{\bB\in\A} there is a natural isomorphism
\w[.]{\Hom\sb{\PAlg}(\piA\bB,\,\Lambda)\cong\Lambda\lin{\bB}}
\end{lemma}

This suggests the notation
\begin{myeq}\label{eqvalfreepa}
\Lambda\lin{V}~:=~\Hom\sb{\PAlg}(V,\,\Lambda)
\end{myeq}
\noindent for any \PAal[s] $\Lambda$ and $V$ with $V$ \emph{free}.

Moreover, we have:

\begin{prop}[see \protect{\cite[\S 6]{BPescF}}]\label{psimptal}
For $\bA$ as above, the category \w{\PAlg\sp{\Dop}} of simplicial
\PAal[s] has a model category structure, in which the weak equivalences
and fibrations are those of the underlying graded simplicial sets.
\end{prop}

\begin{example}\label{egmapalg}
When \w[,]{\C=\Topa} \w[,]{\bA=\bS{1}} and \w[,]{\kappa=\omega} we see
that \w{\PA} is the full sub-simplicial category of \w{\Topa} whose
objects are finite wedges of spheres. In this case a \PAal is just
a \Pa[,] in the sense of \cite[\S 4]{StoV}, with \w{\piA\bY=\pis\Omega\bY}
(equipped with an action of the primary homotopy operations on it),
and an $\bA$-equivalence is just a weak equivalence of (base point components of)
topological spaces, in the usual sense.  In this case
our assumption \wref{eqfreepa} holds by compactness of \w{\bS{n}} and
\w{\bS{n}\times[0,1]} for all \w[.]{n\geq 1}

We could also let \w{\bA=\bS{n}} for some \w[,]{n>1} or use localized spheres
\w{\bA=\bS{n}\sb{R}} in the category $\C$ of $R$-local pointed spaces, for $R$ a subring
of $\QQ$, or algebraic models thereof (such as differential graded Lie algebras, in the
rational case \wh see \S \ref{srex} below). In the latter two cases we refer to either
(equivalent) notion of a \PAal as a \emph{\RPa[.]}
\end{example}

\begin{assumes}\label{afree}
We shall henceforth assume that, in addition to \wref[,]{eqfreepa} the
category \w{\PAlg} (associated to the given $\bA$ in $\C$ as in \S \ref{sass})
satisfies the following requirements:
\begin{enumerate}[(a)]
\item\label{freemodule} By definition, any free \PAal \w{\Lambda\cong\F(X\sb{\ast})}
  is (non-canonically) isomorphic to a coproduct
  \w[,]{\Lambda=\coprod\sb{n=0}\sp{\infty}\,\Lambda\sb{n}} where
  \w{\Lambda\sb{n}} is isomorphic to \w{\F(X\sb{n})} for \w{X\sb{n}\in\gr\Seta}
  concentrated in degree $n$. We assume that
  \w{L\sb{n}:=\Lambda\sb{n}\lin{\Sigma\sp{n}\bA}} is a free $R$-module for some
  principal ideal domain $R$, or possibly a free group when \w[,]{n=0} and choices
  of generators for \w{L\sb{n}} are in bijection with choices of
  generators for \w[.]{\Lambda\sb{n}}
\item\label{freesplit} If \w{j:\Lambda'\hra\Lambda} is a map of free \PAal[s]
  which has a retraction \w[,]{r:\Lambda\epic\Lambda'} then $\Lambda$ splits
  (non-canonically) as a coproduct \w{\Lambda'\amalg\Lambda''}
  with \w{\Lambda''} also free.
\item\label{splitbit} For any $\bU$ and $\bV$ in $\A$, the inclusions
  \w{i\sb{U}:\bU\to\bU\amalg\bV} and \w{i\sb{V}:\bV\to\bU\amalg\bV} and their
  retractions \w{\rho\sb{U}:\bU\amalg\bV\to\bU} and \w{\rho\sb{V}:\bU\amalg\bV\to\bV}
induce a natural decomposition of groups
\begin{myeq}\label{eqgpdecomp}
[\bA',\bU\amalg\bV]~=~[\bA',\,\bU]~\times~[\bA',\,\bV]~\times~C\sb{\bA'}(\bU,\bV)
\end{myeq}
\noindent for any \w{\bA':=\Sigma\sp{n}\bA} \wb[,]{n\geq 0} with the
\emph{cross-term} \w{C\sb{\bA'}(\bU,\bV)}
(the kernel of \w[)]{(\rho\sb{U})\sb{\#}\top(\rho\sb{V})\sb{\#}}
represented by maps \w{f:\bA'\to\bU\amalg\bV} with
\w[.]{\rho\sb{U}\circ f=\ast=\rho\sb{V}\circ f}

\end{enumerate}
\end{assumes}

We now show:

\begin{lemma}\label{lbasis}
  If \w{R\subseteq\QQ} and \w{\bW=\bigvee\sb{i=1}\sp{N}\,\bS{n}\sb{R}} for \w[,]{n\geq 2}
  any basis \w{\B=\{\kappa\sb{1},\dotsc,\kappa\sb{N}\}} for
\w{\pi\sb{n}\bW\cong\bigoplus\sb{i=1}\sp{N}\,R} is a generating set for
the free \RPa \w[.]{\pis\bW}
\end{lemma}

\begin{proof}
We first show that if \w{R\subseteq\QQ} and
\w{\bW=\bigvee\sb{i=1}\sp{N}\,\bS{n}\sb{R}} for \w[,]{n\geq 2} then any basis
\w{\B=\{\kappa\sb{1},\dotsc,\kappa\sb{N}\}} for
\w{\pi\sb{n}\bW\cong\bigoplus\sb{i=1}\sp{N}\,R} is a generating set for
the free \RPa \w[:]{\pis\bW}

Let \w{\E=\{\lambda\sb{1},\dotsc,\lambda\sb{N}\}} be the basis for \w{\pi\sb{n}\bW}
corresponding to the standard generators for \w{\pis\bW} (associated to
the given coproduct decomposition of $\bW$), and \w{M\in\SL{N}}
the change of basis matrix with respect to $\B$. The corresponding map
\w{\vpm:\bW\to\bW} induces an automorphism of \w[,]{H\sb{n}(\bW;R)}
so it is a self-homotopy equivalence by the $R$-local Hurewicz and Whitehead Theorems
(cf.\ \cite[II, 1.2]{HMRoiL}), with homotopy inverse \w[,]{\psm:\bW\to\bW} with
\w{\vpm\sb{\ast}(\lambda\sb{i})=\kappa\sb{i}} for \w[.]{1\leq i\leq N}

By Hilton's Theorem (see \cite{HilS} or \cite[XI, Theorem 6.7]{GWhE}),
for any \w{t\geq n} and \w[,]{\alpha\in\pi\sb{t}\bW} we may write
\w{\beta=\psm\sb{\ast}(\alpha)} uniquely in the form
\begin{myeq}\label{eqrelateone}
\beta~=~\sum\sb{\ell}\,\eta\sb{\ell}\sp{\#}
\omega\sb{\ell}(\lambda\sb{1},\dotsc,\lambda\sb{N})
\end{myeq}
\noindent where \w{\omega\sb{\ell}(\lambda\sb{1},\dotsc,\lambda\sb{N})} is some
\ww{k\sb{\ell}}-fold iterated Whitehead product in a chosen Hall basis
in the free Whitehead-Lie algebra on elements of $\E$, and
\w[.]{\eta\sb{\ell}\in\pi\sb{t}\bS{k\sb{\ell}(n-1)+1}\sb{R}}
Therefore,
\begin{myeq}\label{eqrelatetwo}
\alpha~=~\vpm\sb{\ast}(\beta)~=~\sum\sb{\ell}\,\eta\sb{\ell}\sp{\#}
\omega\sb{\ell}(\kappa\sb{1},\dotsc,\kappa\sb{N})
\end{myeq}
\noindent so $\B$ generates \w[.]{\pis\bW}

Conversely, the result of applying any primary operation $\phi$ to the set $\B$
can be written in the form \wref[,]{eqrelatetwo} so \w{\beta:=\psm\sb{\ast}(\alpha)}
has the form \wref{eqrelateone} with respect to $\E$, and this vanishes if and only $\phi$
was trivial. Thus $\B$ generates \w{\pis\bW} \emph{freely}.
\end{proof}

\begin{prop}\label{passume}
  The assumptions of \S \ref{afree} hold for the motivating example of
  \w{\bA=\bS{n}} \wb{n \geq 1} in
\w[,]{\C=\Topa} with \w[,]{R=\ZZ} as well for \w{\bA=\bS{n}\sb{R}} an $R$-local
sphere in \w{\C=\Top\sb{R}} (the $R$-local model category of pointed spaces), where
$R$ is any sub-ring of $\QQ$.
\end{prop}

\begin{proof}
The statement of \S \ref{afree}\eqref{freemodule} follows from Lemma \ref{lbasis}
(with the non-finitely generated case following from \wref[,]{eqfreepa} which follows
in turn from the compactness of \w{\bS{n}} and \w[).]{\bS{n} \times [0,1]}

If \w{j:\Lambda'\hra\Lambda} is a map of free \Pa[s] with a retraction
\w[,]{r:\Lambda\epic\Lambda'} we may prove \S \ref{afree}\eqref{freesplit}
by induction on the degree: by our convention we use the loop space grading,
so the fundamental group is in degree $0$ and thus \w{\Lambda\sb{0}} (the sub-\Pa
generated by all elements in \w[)]{\Lambda\lin{\bS{1}}} is just a
free group, as is \w[.]{\Lambda'\sb{0}} One can show that if we set
\w[,]{\Lambda''\sb{0}:=\Ker(r\sb{0}:\Lambda\sb{0}\to\Lambda'\sb{0})} which is
also a free group (and thus a free \Pa[),] then
\w[.]{\Lambda\sb{0}\cong\Lambda'\sb{0}\amalg\Lambda''\sb{0}}

If we assume by induction that
$$
\Lambda\sb{<n}~:=~\coprod\sb{k=0}\sp{n-1}\Lambda\sb{k}~\cong~
\coprod\sb{k=0}\sp{n-1}\Lambda'\sb{k}\amalg\coprod\sb{k=0}\sp{n-1}\Lambda''\sb{k}~,
$$
\noindent we have a map of free \Pa[s]
\w{j\sb{0}:\Lambda'\sb{n}\hra\Lambda\sb{n}} with retraction
\w[,]{r\sb{0}:\Lambda\sb{n}\epic\Lambda'\sb{n}} inducing a split inclusion of
free $R$-modules \w[,]{L'\sb{n}\hra L\sb{n}} and thus a decomposition
\w[.]{L\sb{n}\cong L'\sb{n}\oplus L''\sb{n}} Since $R$ is a PID, \w{L''\sb{n}} is
also a free $R$-module.  This allows us to complete a basis for \w{L'\sb{n}} to
one for \w[,]{L\sb{n}} yielding a corresponding decomposition of free \Pa[s]
\w{\Lambda\sb{n}\cong\Lambda'\sb{n}\amalg\Lambda''\sb{n}} by Lemma \ref{lbasis}.

Finally, \S \ref{afree}(\ref{splitbit}) holds for any
suspension \w{\bA=\Sigma\bA'} in \w[,]{\Topa} by the Hilton-Milnor Theorem
(see \cite{MilnF}). Note that it also holds for any small $\bA$ in a stable model category
(see \cite[\S 7.2]{HovM}), since all cross-terms then vanish.
\end{proof}

\begin{remark}\label{rrightbous}
In fact, the $\bA$-equivalences as defined in \S \ref{dpaal} are the weak
equivalences in the right Bousfield localization of $\C$ with respect to $\bA$
(see \cite[\S 5.1]{PHirM}).
In particular, the natural map \w{\CWA{\bY}\to\bY} is an
$\bA$-equivalence, where the cellularization \w{\CWA{\bY}} serves as a functorial
cofibrant replacement for $\bY$ (see \cite[\S2 A]{DroC}).

Two maps \w{f,g:\bX\to\bY} in $\C$ are \emph{$\bA$-equivalent} if and only if
they are related by a zigzag of $\bA$-equivalences. In particular, if all
objects in $\C$ are fibrant (which will be the case in the examples of interest to us),
this implies that the induced maps \w{\widehat{f},\widehat{g}:\CWA{\bX}\to\CWA{\bY}}
are homotopic. We write \w{[\bX,\bY]\sb{\bA}} for the set of $\bA$-equivalence
classes of maps (i.e., the set of maps in the homotopy category \w{\ho\C} for this
model structure).

In the motivating example, where \w{\C=\Top\sb{0}} (see \S \ref{snac}) and
\w[,]{\bA=\bS{1}} if $\bX$ and $\bY$ are CW-complexes we see that $\bA$-equivalences
are actually homotopy equivalences, and two maps \w{f,g:\bX\to\bY} are $\bA$-equivalent
if and only if they are homotopic, so \w[.]{[\bX,\bY]\sb{\bA}=[\bX,\bY]}
\end{remark}

\newpage

%
\supsect{\protect{\ref{cback}}.B}{Chain complexes}

For $\C$ any pointed category, an augmented \emph{chain complex}
in $\C$ is a diagram \w{\As} of the form
\mydiagram[\label{eqchaincx}]{
\dotsc A\sb{n+1}\ar[r]\sp{\df{n+1}} &
A\sb{n} \ar[r]\sp{\df{n}} &
A\sb{n-1} \ar[r]\sp{\df{n-1}} &
A\sb{n-2} & \dotsc A\sb{0} \ar[r]\sp{\df{0}} & A\sb{-1}
}
\noindent with \w{\df{n}\circ\df{n+1}=0} for all
\w[.]{n\geq 0} We denote the category of such chain complexes by \w[,]{\ChC}
that of $n$-truncated chain complexes by \w[,]{\ChC\sp{\leq n}} and that of
bounded-below chain complexes with \w{A\sb{i}=\ast} for \w{-1\leq i<n} by
\w[,]{\ChC\sp{\geq n}} with the obvious truncation functors
\w{\sk{n}:\ChC\to\ChC\sp{\leq n}} (the usual skeleton, or restriction) and
\w[.]{\csk{n}:\ChC\to\ChC\sp{\geq n}}

\begin{defn}\label{dspherecc}
For any object $\oG{}$ in a pointed category $\C$, let \w{\oG{}\oS{n}}
be the chain complex in \w{\ChC} having $\oG{}$ in dimension $n$
(and $\ast$ elsewhere). Similarly, \w{\oG{}\odisc{n}} has $\oG{}$ in dimensions
$n$ and \w[,]{n-1} with the identity between them as boundary (and $\ast$ elsewhere).
We write \w{\iota\sb{n}:\oG{}\oS{n-1}\hra \oG{}\odisc{n}} for the inclusion.
\end{defn}

\begin{mysubsection}{Model categories of chain complexes}
\label{smodchaincx}
When $\C$ is a pointed model category as in \S \ref{sass}, we
will consider \emph{projective} model category structures
on \w{\ChC} and \w[,]{\ChC\sp{\leq n}} in which the weak equivalences and
fibrations are both defined levelwise, so all objects will be fibrant.
For \w[,]{\ChC\sp{\leq n}} the cofibrant objects are the
\emph{strongly cofibrant} $n$-chain complexes \w[,]{\As} where for each \w{k\leq n}
the natural map \w{\Cok(\df{k+1})\to A\sb{k-1}} is a cofibration
(with \w[).]{A\sb{n+1}:=\ast} See \cite[\S 11.6]{PHirM}.

There is a dual \emph{injective} model category structure
on \w{\ChC} and \w[,]{\ChC\sp{\leq n}} in which the weak equivalences and
cofibrations are defined levelwise, and the fibrant objects are described in
\wref[.]{eqstfibseq}
\end{mysubsection}

\begin{mysubsection}{Attaching cells to chain complexes}
\label{sattach}
The usual way to construct a chain complex \w{\As} in \w{\ChC} is by means
of \emph{attaching maps} \w{\od:\oA{n}\oS{n-1}\to\sk{n-1}\As} in
\w[.]{\ChC\sp{\leq n-1}}
The next skeleton \w{\sk{n}\As} is then
the pushout
\mydiagram[\label{skeletonl}]{
\ar @{} [drr]|(0.71){\framebox{\scriptsize{PO}}}
\oA{n}\oS{n-1} \ar[d]\sb{\iota\sb{n}} \ar[rr]\sp{\od} &&
\sk{n-1}\As \ar[d] \\
\oA{n}\odisc{n} \ar[rr] && \sk{n}\As~,
}
\noindent (see \S \ref{dspherecc}), with $\od$ in degree \w{n-1} equal to
\w[.]{\partial\sb{n}:A\sb{n}\to A\sb{n-1}}

When $\C$ is a model category, in order to make this process homotopy meaningful we
generally use a (strongly) cofibrant replacement of the source \w{\oA{n}\oS{n-1}} of the
attaching map \w[.]{\od}
\end{mysubsection}

\supsect{\protect{\ref{cback}}.C}{Augmented simplicial objects}

We now collect some standard facts and constructions related to
augmented simplicial objects in a category $\C$:

\begin{defn}\label{dmco}
In a pointed and complete category $\C$, the $n$-th \emph{Moore chains} object
of a restricted augmented simplicial object \w{\Gd\in\C\sp{\Dresp\op}} is defined to be:
\begin{myeq}\label{eqmoor}
\cM{n}\Gd~:=~\cap\sb{i=1}\sp{n}\Ker\{d\sb{i}:G\sb{n}\to G\sb{n-1}\}~,
\end{myeq}
\noindent that is, the limit of the diagram
\mydiagram{
  G\sb{n} \ar@/^{1.1pc}/[rr]\sp{d\sb{1}}\sb{\vdots}
  \ar@/_{1.1pc}/[rr]\sb{d\sb{n}} &&
  G\sb{n-1} && \ast \ar[ll]
}
with differential
$$
\dif{n}:=d\sb{0}\rest{\cM{n}\Gd}:\cM{n}\Gd\to\cM{n-1}\Gd~.
$$
The $n$-th \emph{Moore cycles} object is \w{\cZ{n}\Gd:=\Ker(\dif{n})} (the analogous
limit including \w[).]{d\sb{0}} Write \w{w\sb{n}:\cM{n}\Gd\hra G\sb{n-1}} and
\w{v\sb{n}:\cZ{n}\Gd\hra\cM{n}\Gd} for the inclusions.

We use the same notation for unrestricted or unaugmented \w[,]{\Gd}
although the reader should note that for non-trivial augmented \w[,]{\Gd}
\w{\cZ{0}(\Gd)} differs from \w[.]{\cZ{0}(\sigma\sp{\ast}(\Gd))=G\sb{0}}
\end{defn}

\begin{defn}\label{dlmo}
For a (possibly \wwb{n-1}truncated) simplicial object \w{\Gd\in\C\sp{\Dop}} in a
cocomplete category $\C$, the $n$-th \emph{latching object} for \w{\Gd} is the colimit
\begin{myeq}\label{eqlatch}
L\sb{n}\Gd~:=~\colimit{\theta\op:\bbk\to\bbn}\,G\sb{k}~,
\end{myeq}
\noindent where $\theta$ ranges over the surjective maps \w{\bbn\to\bbk} in
$\Del$ (for \w[).]{k < n} There is a natural map \w{\sigma\sb{n}:L\sb{n}\Gd\to G\sb{n}}
induced by the indexing maps $\theta$ of the colimit for any $n$-truncated simplicial
object, and any iterated degeneracy map \w{s\sb{I}=\theta\sb{\ast}:G\sb{k}\to G\sb{n}}
factors as
\begin{myeq}\label{equnivdeg}
s\sb{I}~=~\sigma\sb{n}\circ \inc\sb{\theta}~,
\end{myeq}
\noindent where \w{\inc\sb{\theta}:G\sb{k}\to L\sb{n}\Gd} is the structure
map for the copy of \w{G\sb{k}} indexed by $\theta$.

Note that the inclusion \w{\Dres\hra\Delta} induces a forgetful functor
\w[,]{\U:\C\sp{\Dop}\to\C\sp{\Dres\op}} and its left adjoint
\w{\cL:\C\sp{\Dres\op}\to\C\sp{\Dop}} is given by
\w[,]{(\cL\Gd)\sb{n}=G\sb{n}\amalg L\sb{n}\Gd} with degeneracies given by
\wref{equnivdeg} and face maps coming from the simplicial identities.
It follows that any augmentation of
\w{\Gd} also serves as an augmentation of \w{\cL\Gd} and vice versa, so this
remains an adjunction for the augmented categories.

Dually, the $n$-th \emph{matching object} for \w{\Gd\in\C\sp{\Dop}} is defined to be
\begin{myeq}\label{eqmatch}
M\sb{n}\Gd~:=~\lim\sb{\phi\op:\bbn\to\bbk}\,G\sb{k}~,
\end{myeq}
\noindent where $\phi$ ranges over the injective maps \w{\bbk\to\bbn} in
$\Delta$. As above, there is a natural map
\w{\zeta\sb{n}:G\sb{n}\to M\sb{n}\Gd} induced by the structure maps of the limit
for any $n$-truncated restricted simplicial object,
and every (iterated) face map factors through it (see \cite[X,\S 4.5]{BKanH}).

For an augmented \w[,]{\Gd\in\C\sp{\Dp\op}} matching objects are defined similarly, but
now \w[,]{M\sb{0}\Gd=G\sb{-1}} and \w{M\sb{1}\Gd} is the pullback of
\w[,]{G\sb{0} \to G\sb{-1} \leftarrow G\sb{0}} rather
than a product.
\end{defn}

\begin{remark}\label{rmodsso}
When $\C$ is a model category, we shall use the Reedy model structure
of \cite[\S 15.3]{PHirM}, which differs from the projective structure,
on \w[,]{\C\sp{\Dop}} \w[,]{\C\sp{\Dres\op}}
\w[,]{\C\sp{\Dp\op}} and \w[.]{\C\sp{\Dresp\op}}
Note that the constant augmented object \w{\cpd{A}} for a fibrant object \w{A \in \C}
is Reedy fibrant in \w{\C\sp{\Dp\op}} but \w{\cd{A}} is not Reedy fibrant in
\w{\C\sp{\Dop}} (see \S \ref{raug} below).
\end{remark}

\begin{mysubsection}{Comparing chain complexes and simplicial objects}
\label{schaincx}
If $\C$ is a pointed category, the Moore chain functor
\w{\cMs:\C\sp{\Dresp\op}\to\ChC} just described has a left
adjoint (and right inverse) \w{\E:\ChC\to\C\sp{\Dresp\op}} with
\w[,]{(\E\A\sb{\ast})\sb{n}=A\sb{n}} \w[,]{d\sb{0}\sp{n}=\dif{n}} and
\w{d\sb{i}\sp{n}=0} for \w[.]{i\geq 1} This also holds for \w{\ChC\sp{\leq n}}
if we truncate \w[.]{\C\sp{\Dresp\op}} Moreover:
\end{mysubsection}

\begin{lemma}\label{lfibwe}
For \w{\C=\Topa} or \w[,]{\G=\Gp\sp{\Dop}} the functor
\w{\cMs:\C\sp{\Dresp\op}\to\Ch\sb{\C}} preserves fibrancy and weak equivalences
among fibrant objects with respect to the Reedy model structure of \S \ref{rmodsso}
in \w{\C\sp{\Dresp\op}} and the injective model structure of \S \ref{smodchaincx} on
\w[.]{\Ch\sb{\C}}
\end{lemma}

\begin{proof}
See \cite[Proposition 5.7]{DKStB} and \cite[Lemma 2.7]{StoV}.
\end{proof}

We recall the following augmented dual of \cite[X, Proposition 6.3(ii)]{BKanH}:

\begin{lemma}\label{lmoore}
Let \w{\Xd\in\C\sp{\Dp\op}} be a Reedy fibrant augmented simplicial object
over a model category $\C$, and $\oB$ a cofibrant homotopy
cogroup object in $\C$. Then for any Moore chain \w{\beta\in\cM{n}[\oB,\Xd]}
for the augmented simplicial group \w[\upshape :]{[\oB,\Xd]}
\begin{enumerate}[(a)]
\item $\beta$ can be realized by a map \w[.]{b:\oB\to\cM{n}\Xd}
\item If $\beta$ is a Moore \emph{cycle}, in \w[,]{\cZ{n}[\oB,\Xd]} we can choose a
  nullhomotopy for \w[,]{\dif{n}\circ b} \w[.]{H:C\oB\to\cM{n-1}\Xd}
\end{enumerate}
\end{lemma}

\begin{proof}
Since \w{\Xd} is Reedy fibrant (see \cite[Ch.\ 15]{PHirM}), the augmented simplicial
space \w{\Ud=\mapa(\oB,\Xd)\in \Sa\sp{\Dp\op}} is Reedy fibrant, so by
\cite[Lemma 2.7]{StoV}, for every \w{j > 0} the inclusion
  \w{\iota:\cM{n}\Ud\hra\bU\sb{n}} induces an isomorphism
  \w[.]{\iota\sb{\ast}:\pi\sb{j}\cM{n}\Ud\to\cM{n}\pi\sb{j}\Ud}
Since \w{\cM{n}} is a limit, \w[.]{\cM{n}\Ud=\mapa(\oB,\cM{n}\Xd)}
Since $\oB$ is a homotopy cogroup object, \w{\pi\sb{0}\Ud} is still a group,
so the above holds for \w{j=0} too.

Note that in both the augmented and non-augmented case \w[,]{\cM{0}\Ud=\bU\sb{0}}
so the result also holds in dimensions \w[.]{n=0,-1}
\end{proof}

By analogy with the mapping cone for chain complexes
(see \cite[\S 1.5]{WeibHA}) we have the following notion, which will play a key
technical role in what follows:

\begin{defn}\label{dchaincof}
For any map \w{f:\Ad\to\Bd} in \w{\C\sp{\Dresp\op}}
we define the restricted augmented simplicial object \w{\Cd=\Cone(f)}
by setting \w{C\sb{n}:= B\sb{n}\amalg A\sb{n-1}} (where \w[),]{A_{-2}=\ast} with
\[
d\sb{i}\sp{C\sb{n}}~:=~
\begin{cases}
  \inc\sb{B\sb{n-1}} \circ (d\sb{0}\sp{B\sb{n}} \bot\, f\sb{n-1}) & \text{if}~i=0\\
 d\sb{i}\sp{B\sb{n}} \amalg\, d\sb{i-1}\sp{A\sb{n-1}} & \text{if}~i\geq 1~,
\end{cases}
\]
\noindent in the notation of \S \ref{snac}, and a natural inclusion of restricted
augmented simplicial objects \w{\ell:\Bd\hra\Cone(f)} which is the identity
in degree \w[.]{-1}

For the required face identity, we may verify that
$$
(d\sb{0}\circ d\sb{j})\rest{A\sb{n-1}}~=
(d\sb{0}\sp{C\sb{n-1}})\rest{A\sb{n-2}}\circ d\sb{j-1}\sp{A\sb{n-1}}=
f\sb{n-2}\circ d\sb{j-1}\sp{A\sb{n-1}}~=
d\sb{j-1}\sp{B\sb{n-1}}\circ f\sb{n-1}~=(d\sb{j-1}\circ d\sb{0})\rest{A\sb{n-1}}
$$
\noindent for all \w[,]{0<j} while
$$
(d\sb{i}\circ d\sb{j})\rest{A\sb{n-1}}~=~
d\sb{i-1}\sp{A\sb{n-2}}\circ d\sb{j-1}\sp{A\sb{n-1}}~=~
d\sb{j-2}\sp{A\sb{n-2}}\circ d\sb{i-1}\sp{A\sb{n-1}}~=~
(d\sb{j-1}\circ d\sb{i})\rest{A\sb{n-1}}
$$
\noindent for all \w[.]{1\leq i<j}
\end{defn}

\begin{example}\label{xcone}
  Suppose that \w{\Ad} is concentrated in one dimension, for example,
  \w{\Ad=\E(\oG{n}\oS{n-1})} and \w{\Bd} is \wwb{n-1}truncated.
  Then in dimensions \w[,]{k<n} the inclusion \w{\ell\sb{k}} is an isomorphism,
  \w{B\sb{k} \cong B\sb{k} \amalg \ast} (since $\C$ is pointed), with the
  face maps defined through these isomorphisms.
  In dimension \w[,]{k=n} we have \w{\Cone(f)\sb{n}=\ast \amalg \oG{n} \cong \oG{n}}
  with \w{d\sb{0} \cong f\sb{n-1}} and all higher face maps zero.
\end{example}

\begin{defn}\label{dscwo}
  An unaugmented simplicial object \w{\Gd\in\C\sp{\Dop}} over a pointed category $\C$
  is called a \emph{CW object} if it is equipped with a \emph{CW basis}
\w{(\oG{n})\sb{n=0}\sp{\infty}} in $\C$ such that
\w[,]{G\sb{n}=\oG{n}\amalg L\sb{n}\Gd} and \w{d\sb{i}\rest{\oG{n}}=0}
for \w[.]{1\leq i\leq n} By the simplicial identities the
restriction of the $0$-th face map \w{d\sb{0}\rest{\oG{n}}:\oG{n}\to G\sb{n-1}}
factors as the composite
\mydiagram[\label{eqattach}]{
\oG{n} \ar[r]^(0.4){\odz{G\sb{n}}} &
\cZ{n-1}\Gd~ \ar@{^{(}->}[r]\sp{v\sb{n-1}} &
\cM{n-1}\Gd~ \ar@{^{(}->}[r]^(0.6){w\sb{n-1}} & G\sb{n-1}
}
\noindent (in the notation of \S \ref{dmco}, with
\w[),]{v\sb{n-1}\circ\odz{G\sb{n}}=(\dif{n})\rest{\oG{n}}} and we call
\w{\odz{G\sb{n}}} the \emph{$n$-th attaching map} for \w[.]{\Gd}
\end{defn}

The following observation essentially follows from Example \ref{xcone}
and the construction of $\cL$.

\begin{lemma}\label{lscwo}
Any CW object \w{\Gd} over $\C$ with CW basis \w{(\oG{n})\sb{n=0}\sp{\infty}}
can be constructed inductively as follows, starting with \w{\sk{0}\Gd:=\cd{\oG{0}}}
(see \S \ref{snac}): given the \wwbu{n-1}truncated simplicial object
\w[,]{\sk{n-1}\Gd} the attaching map \w{\odz{G\sb{n}}:\oG{n}\to\cZ{n-1}(\sk{n-1}\Gd)}
is equivalent to a chain map \w{f:\oG{n}\oS{n-1}\to\cMs(\sk{n-1}\Gd)}
(see \S \ref{dspherecc}) and so to an adjoint restricted simplicial map
\w{\widetilde{f}:\E(\oG{n}\oS{n-1})\to\U\sk{n-1}\Gd} (see \S \ref{schaincx});
we define \w{\sk{n}\Gd} to be the pushout in $n$-truncated simplicial objects
\mydiagram[\label{eqnskeleton}]{
\ar@{}[drr]|(0.71){\framebox{\scriptsize{PO}}}
\cL\U\sk{n-1}\Gd \ar[d]\sb{\cL\ell} \ar[rr]\sp{\vartheta} &&
\sk{n-1}\Gd \ar[d] \\
\cL\Cone(\widetilde{f}) \ar[rr] && \sk{n}\Gd~,
}
\noindent where \w{\vartheta:\cL\U\to\Id} is the counit for the adjunction
of \S \ref{dlmo}, and $\ell$ is as in \S \ref{dchaincof} (see \wref[).]{skeletonl}
\end{lemma}

This yields an explicit description of \w[,]{G\sb{n}=\oG{n}\amalg L\sb{n}\Gd}
since by induction we see that the $n$-th latching object of \w{\Gd} is given by:
\begin{myeq}\label{eqslatch}
L\sb{n}\Gd~:=~
\coprod\sb{0\leq k\leq n-1}~~\coprod\sb{0\leq i\sb{1}<\dotsc<i\sb{n-k-1}\leq n-1}~
\oG{k}~,
\end{myeq}
\noindent where the iterated degeneracy map
\w[,]{s\sb{i\sb{n-k-1}}\dotsc s\sb{i\sb{2}}s\sb{i\sb{1}}} restricted to the
basis \w[,]{\oG{k}} is the inclusion into the copy of \w{\oG{k}} indexed by
$k$ (in the first coproduct) and \w{(i\sb{1},\dotsc,i\sb{n-k-1})} (in the second).

We note for future reference the following useful fact (which we shall not need here):

\begin{lemma}\label{lfreesimp}
Every free simplicial \PAal \w{\Vd} has a CW basis \w[.]{\{\oV{n}\}\sb{n=0}\sp{\infty}}
\end{lemma}

\begin{proof}
This follows from \S \ref{afree}\eqref{freesplit} by induction on the
simplicial dimension \w[,]{n\geq 0} since
the simplicial identity \w{d\sb{i}s\sb{i}=\Id} shows that \w{V\sb{n-1}} splits
off \w{V\sb{n}} in various ways, so \w{L\sb{n}\Vd} does, too,
as in \wref[.]{eqslatch}
\end{proof}

\begin{defn}\label{dcwres}
A \emph{CW-resolution} of a \PAal \w{\Lambda\in\PAlg} is a cofibrant replacement
\w{\vare:\Gd\xra{\simeq}\cd{\Lambda}} (in the model category of simplicial \PAal[s]
from \ref{psimptal}), which is also a CW object with CW basis
\w{(\oG{n})\sb{n=0}\sp{\infty}} consisting of free \PAal[s.]
\end{defn}

\begin{assumes}\label{setmc}
In order to formulate our results most efficiently, in addition to the assumptions
of \S \ref{sass} and \S \ref{afree} we henceforth also require:
\begin{enumerate}[(1)]
\item The category \w{\C\sp{\Dop}} of simplicial objects over $\C$ has a
resolution model category structure (see \cite{JardB} and
compare \cite{DKStE}) with respect to $\bA$.
\item\label{realfun} There is a \emph{realization functor}
  \w[,]{\|-\|:\C\sp{\Dop}\to\C} equipped with initial augmentation
  \w[,]{\eta:\Wd\to\|\Wd\|} such that for any
augmented simplicial object \w{\vare:\Wd\to\bY} over $\C$ where the
associated augmented simplicial \PAal \w{\vare\sb{\#}:\piA \Wd\to\piA\bY} is
acyclic (that is, \w{\vare\sb{\#}:\piA \Wd\to\cd{\piA\bY}} is a weak equivalence, as in
\S \ref{psimptal}), the natural map \w{\|\Wd\|\to\bY} induces an isomorphism in
\w[.]{\PAlg}

This would typically be defined as a coend, as for the usual geometric realization
(but see \cite[4.10]{BJTurnR}).
\end{enumerate}
\end{assumes}

These assumptions hold in our motivating example of \S \ref{egmapalg}:

\begin{example}\label{egfree}
Let \w{\C=\Top\sb{0}} (see \S \ref{snac}) and \w{\bA=\bS{n}} for some \w[.]{n\geq 1}
In this case, \w{\|\Wd\|} is the geometric realization, and
condition \eqref{realfun} follows from the collapse of the Bousfield-Friedlander spectral
sequence under the given hypotheses (see \cite[Theorem B.5]{BFrieH}).
However, $R$-local spaces in \w{\Top\sb{0}} also satisfy
these assumptions, as do differential graded (Lie) algebras over $\QQ$
(see \cite{QuiR}), with \w{\|-\|} a suitable homotopy colimit \wh and more
generally, for other \ww{E\sp{2}}-model categories in the sense of
\cite[\S 4.8]{BJTurnR}.
\end{example}

\begin{remark}\label{rcw}
If we set \w{Z\sb{-1}\Gd:=\Lambda} and\w[,]{\odz{G\sb{0}}:=\vare}
any CW object \w{\Gd} for which each \w{\oG{n}} is a free \PAal
and each attaching map \w{\odz{G\sb{n}}} surjects onto \w{Z\sb{n-1}\Gd}
\wb{n\geq 0} is a CW-resolution of $\Lambda$. We can then make \w{\Gd}
into an \emph{augmented} simplicial CW object by setting \w{G\sb{-1}:=\Lambda}
with \w{\vare\sb{0}:G\sb{0} \to \Lambda} as the augmentation.
\end{remark}

%
%
\sect{Realizing simplicial \PAal resolutions}
\label{crsar}

The main technical tool needed in this paper is an explicit version,
and generalization, of \cite[Theorem 3.16]{BlaCW}, which states that any algebraic
resolution \w{\Vd} of a realizable \Pa $\Lambda$ may be realized by a simplicial
space \w[.]{\Wd} This \w{\Wd} must be of a particular form, which we now describe.
Throughout this section we assume that \w{\bA\in\C} is as in \S \ref{setmc},
and \w{\PA} as in \S \ref{cback}.A.

Our goal here is to show how to realize a CW (algebraic) resolution \w{\Vd}
of a realizable \PAal \w[,]{\Lambda=\piA\bY}  with CW basis
\w[,]{\{\oV{n}\}\sb{n=0}\sp{\infty}} by an augmented simplicial object
\w{\Wd \to \bY} in $\C$.  We would like to mimic the CW construction of
\w{\Vd} by exhibiting \w{\Wd} as a homotopy colimit of a sequence of maps
\begin{myeq}\label{eqftower}
\W{0}~\xra{\prn{1}}~\W{1}~\xra{\prn{2}}~\W{2}~\to~\dotsc~
\W{n-1}~\xra{\prn{n}}~\W{n}~\to~\dotsc~,
\end{myeq}
\noindent where \w{\W{n}} realizes \w{\Vd} through simplicial dimension $n$.

In the induction step, we pass from \w{\Xd=\W{n-1}} to \w{\W{n}} by attaching an
object \w{\oB} realizing \w{\oV{n}} in simplicial dimension $n$, as for \w[.]{\Vd}
By Lemma \ref{lscwo}, it is enough to find an attaching map \w{f:\oB\oS{n-1}\to\cMs{\Xd}}
in \w[.]{\ChC}  Unfortunately, there are obstructions to doing so
in general (see \cite{BJTurnR,BJTurnHA}), hence we must:
\begin{enumerate}[(1)]
\item replace \w{\oB\oS{n-1}} with a (strongly) cofibrant object \w[;]{\bDs}
\item realize the algebraic attaching map $f$ of Lemma \ref{lscwo} by a map
  \w{F:\bDs \to \cMs(\Xd)} in \w[;]{\ChC} and
\item modify the result of Lemma \ref{lscwo} to obtain a Reedy cofibration
  \w[,]{\Xd \to \Xd[F]} with Reedy fibrant target
  (see \wref[),]{eqdoublereplace} playing the role of \w{\prn{n}} above.
\end{enumerate}
This section will treat each of these steps separately.

\supsect{\protect{\ref{crsar}}.A}{Strongly cofibrant chain complexes}

\label{smakesc}
Recall from \S \ref{smodchaincx} that weak equivalences in \w{\ChC\sp{\leq n}}
are defined entrywise and that an $n$-chain complex \w{\bDs}  is
strongly cofibrant precisely when the structure map \w{\Cof(\df{k+1})\to \bD\sb{k-1}}
(out of ``$k$-chains modulo boundaries'') is a cofibration for each $k$.
Thus if \w{\bDs\in\ChC\sp{\leq n-1}} is a strongly cofibrant approximation
to \w[,]{\oB\oS{n-1}} \w{\bD\sb{k}} must be contractible for \w[,]{k\neq n-1} since
then \w[.]{(\oB\oS{n-1})\sb{k}=\ast}

As explained in \cite[\S 15.2]{PHirM}, it is natural to construct \w{\bDs}
by a descending induction on \w[,]{0\leq k\leq n-1} starting with
\w{\bD\sb{n-1}=\oB} (assumed cofibrant by \S \ref{cback}.A). Since \w[,]{\bD\sb{n}=\ast}
also \w[,]{\Cof(\df{n})=\oB} therefore by construction \w{\bD\sb{n-2}} must be a
cone on \w{\oB} in the sense of \cite[I, \S 2]{QuiH}.

We could of course choose \w{\bD\sb{n-2}} to be the standard cone \w{C\oB}
of \S \ref{sass}, but we will require more general (strongly) cofibrant objects,
in order to replace certain maps by fibrations (see Lemma \ref{tresext}).
Therefore, for each \w{0\leq k\leq n-1} we merely require that there be
given a (strict) cofibration sequence in $\C$:
\mydiagram[\label{eqconesus}]{
\osB{k}~\quad \ar@{^{(}->}[r]\sp{\oi{k}} & ~\oCsB{k}~
\ar@{->>}[r]\sp{\oq{k}} & ~\osB{k+1}~
}
\noindent with \w{\oCsB{k}\simeq\ast} for \w{0\leq k\leq n-1} (with the convention
that \w[).]{\oCsB{-1}:=\oB} Thus \w{\osB{k}} is indeed a model for the
suspension of \w[,]{\osB{k-1}} and, as a consequence, the $k$-th suspension of
\w[,]{\oB} in the sense of \cite[\textit{loc.\ cit.}]{QuiH} (see \wref{eqmodconesusp}
below).

If we let \w[,]{\bD\sb{k}:=\oCsB{n-k-2}} the differential
\w{\diff{k}{\bD}:\bD\sb{k}\to \bD\sb{k-1}} is defined to be the composite of
\mydiagram[\label{eqdif}]{
\oCsB{n-k-2}~\ar@{->>}[rr]\sp{\oq{n-k-2}} &&
\osB{n-k-1}~\ar@{^{(}->}[rr]\sp{\oi{n-k-1}} && \oCsB{n-k-1}~,
}
\noindent even for \w[.]{k=0} Moreover, \w{\osB{n-k}=\Cof(\diff{k}{\bD})} (a strict
cofiber, since \w{\oq{n-k-2}} is epic), so the cofibration
\w{\oi{n-k}} shows that \w{\bDs} is indeed strongly cofibrant.

The first three stages of the process are depicted in the following
commutative diagram:
\mydiagram[\label{eqmark}]{
&&&& \ast \ar@{^{(}->}[d]\sp{\partial\sb{n}\sp{\bD}} \ar[rr] &&
\ast \ar[d]\sp{\partial\sb{n}} \\
&&&& \oB \ar[dll] \ar@{^{(}->}[d]\sb{\oi{0}=}\sp{\partial\sb{n-1}\sp{\bD}}
\ar[rr]\sp{\Id} &&
\oB \ar[d]\sp{\partial\sb{n-1}} \\
&& \ast \ar@{^{(}->}[d] &&
\oCB \ar[dll]\sb{\oq{0}} \ar[d]\sp{\partial\sb{n-2}\sp{\bD}} \ar[rr]\sp{\sim} &&
\ast \ar[d]\sp{\partial\sb{n-2}} \\
&& \osB{} \ar[dll] \ar@{^{(}->}[rr]\sp{\oi{1}} &&
\oCsB{} \ar[dll]\sb{\oq{1}} \ar[d]\sp{\partial\sb{n-3}\sp{\bD}}
\ar[rr]\sp{\sim} && \ast \ar[d]\sp{\partial\sb{n-3}} \\
\ast \ar@{^{(}->}[rr] && \osB{2} \ar@{^{(}->}[rr]\sp{\oi{2}} &&
\oCsB{2}  \ar[rr]\sp{\sim} && \ast
}

The parallelograms on the left are (homotopy) pushouts,
and the triangles are used to define the differentials, with cofibrations and
weak equivalences as indicated.

When we use standard cones and suspensions throughout, we obtain the
\emph{standard} cofibrant replacement for \w[,]{\oB\oS{n-1}} which we denote
by \w[.]{\Dsn{n}(\oB)}

\supsect{\protect{\ref{crsar}}.B}{Realizing attaching maps}

Assume given a CW resolution \w{\Vd} of \w{\Lambda=\piA\bY} in \w[,]{\PAlg\sp{\Dop}}
with CW basis \w[,]{\{\oV{n}\}\sb{n=0}\sp{\infty}} and a Reedy fibrant
\wwb{n-1}truncated augmented simplicial object \w{\Xd} in $\C$,
realizing \w{\Vd} through simplicial dimension \w[,]{n-1} with \w[.]{\bX\sb{-1}=\bY}
In addition, assume we have (cofibrant) $\oB$ realizing \w[,]{\oV{n}}
and we would like to construct a map \w{\oB\oS{n-1}\to\cMs{\Xd}} in \w{\ChC\sp{\leq n-1}}
realizing the (algebraic) chain map \w[,]{f:\oV{n}\oS{n-1}\to\cMs(\sk{n-1}\Vd)}
in order to apply Lemma \ref{lscwo}.  As noted above, we must first replace
\w{\oB\oS{n-1}} by a (strongly) cofibrant \w{\bDs} to produce
\w[,]{F:\bDs\to\cMs{\Xd}} using the following

\begin{prop}\label{pobst}
Given a CW resolution \w[,]{\Vd} a Reedy fibrant \wwbu{n-1}truncated augmented
simplicial object \w[,]{\Xd} an object \w{\oB\in\C} realizing \w[,]{\oV{n}}
and a strongly cofibrant \w{\bDs} as above,
the algebraic attaching map \w{f:\oV{n}\oS{n-1}\to\cMs(\sk{n-1}\Vd)} can be
realized by a chain map \w[.]{F:\bDs \to\cMs{\Xd}}
\end{prop}

\begin{proof}
We construct \w{\Fk{k}} by a downward induction on coskeleta (see \S 1.B), for
\w[.]{-1\leq k\leq n-1}

To start the induction we must choose \w[.]{\Fk{n-1}:\bD\sb{n-1}=\oB\to\cM{n-1}\Xd}
Since \w{\piA\bX\sb{k}\cong V\sb{k}} for all \w{0\leq k\leq n-1} by assumption,
the algebraic attaching map \w{\odz{V\sb{n}}:\oV{n}\to V\sb{n-1}} can be thought of
as a homotopy class
\begin{myeq}\label{eqmoorecyc}
\begin{split}
\alpha~\in&~[\oB,\,\bX\sb{n-1}]
~=~\piA\bX\sb{n-1}\lin{\oB}~\cong~V\sb{n-1}\lin{\oB}\\
~\cong&~\Hom\sb{\PAlg}(\piA\oB,\,V\sb{n-1})
~\cong~\Hom\sb{\PAlg}(\oV{n},\,V\sb{n-1})~,
\end{split}
\end{myeq}
\noindent where the next to last isomorphism follows from Lemma \ref{lfreeta}.

Since by Definition \ref{dscwo} \w{\odz{V\sb{n}}:\oV{n}\to V\sb{n-1}} actually
lands in \w[,]{\cM{n-1}\Vd} this $\alpha$ is a Moore chain in
\w[,]{\pi\sb{0}\map\sb{\C}(\oB,\,\Xd)} so by Lemma \ref{lmoore}(a),
\w{\alpha} can be represented by a map \w[.]{\Fk{n-1}:\oB\to\cM{n-1}\Xd}
By \wref[,]{eqattach} \w{\odz{V\sb{n}}} lands in the \wwb{n-1}Moore cycles, so
by Lemma \ref{lmoore}(b) the \wwb{n-2}Moore chain
\w{\ak{n-2}:=\diff{n-1}{C}\circ\Fk{n-1}} has a nullhomotopy
\w[,]{\Fk{n-2}:\diff{n-1}{C}\circ\Fk{n-1}\sim 0} and thus a map
\w[\vsn.]{\Fk{n-2}:\bD\sb{n-2}=\oCB{} \to \cM{n-2}\Xd}

In the $k$-th stage of the induction, we assume given \w{F:\csk{k}\bDs\to\csk{k}\cMs{\Xd}}
for \w{\Xd} and \w{\bDs} as above, with \w[.]{0 \leq k \leq n-2}
We shall show that we can always extend $F$ to the \wwbu{k-1}coskeleta by
modifying \w[.]{F\sb{k}}

Note that we can decompose \w{\diff{k}{C}} as \w[,]{v\sb{k-1}\circ\wdiff{k}{C}}
and already \w[.]{\wdiff{k}{C}\circ v\sb{k}=0}  As a consequence,
\w[,]{0=\wdiff{k}{C}\circ\diff{k+1}{C}\circ\Fk{k+1}=\wdiff{k}{C}\circ\Fk{k}
  \circ\diff{k+1}{D}=\wdiff{k}{C}\circ\Fk{k}\circ\oi{n-k-2}\circ\oq{n-k-3}}
and \w{\oq{n-k-3}} is epic, so we see \w[.]{\wdiff{k}{C}\circ\Fk{k}\circ\oi{n-k-2}=0}
Thus, the pushout property in \eqref{eqmark} implies there is a unique
\w{\ak{k-1}:\osB{n-k-1}\to\cZ{k-1}\Xd} in
\mydiagram[\label{eqgrid}]{
  \bD\sb{k+1}=\oCsB{n-k-3} \ar[dd]_{\diff{k+1}{\bD}} \ar[dr]^{\oq{n-k-3}}
  \ar[rrr]^{F\sb{k+1}} & & &
  \cM{k+1}{\Xd} \ar[dl]^{\wdiff{k+1}{C}} \ar[dd]^{\diff{k+1}{C}}\\
& \osB{n-k-2} \ar[dl]^{\oi{n-k-2}} \ar[r]^{a\sb{k}}& \cZ{k}{\Xd} \ar[dr]^{v\sb{k}}\\
  \bD\sb{k}=\oCsB{n-k-2} \ar[dd]_{\diff{k}{\bD}} \ar[dr]^{\oq{n-k-2}} \ar[rrr]^{F\sb{k}} & & &
  \cM{k}{\Xd}\ar[dl]^{\wdiff{k}{C}}
    \ar[dd]^{\diff{k}{C}} \\
& \osB{n-k-1} \ar[dl]^{\oi{n-k-1}} \ar[r]^{a\sb{k-1}} &\cZ{k-1}{\Xd} \ar[dr]^{v\sb{k-1}}\\
\bD\sb{k-1}=\oCsB{n-k-1} \ar@{-->}[rrr]^{F\sb{k-1}}  & & & \cM{k-1}{\Xd}
}
\noindent satisfying
\begin{myeq}\label{eqikak}
v\sb{k-1}\circ\ak{k-1}\circ\oq{n-k-2}~=~\diff{k}{C}\circ\Fk{k}~.
\end{myeq}
\noindent  Note that \w{a\sb{k}} in \wref{eqgrid} is constructed similarly,
satisfying \wref{eqikak} for $k$ rather than \w[,]{k-1}
and that \w{F\sb{k}} makes the upper square commute precisely when it is a
nullhomotopy for \w[,]{v\sb{k}\circ\ak{k}} that is
\begin{myeq}\label{eqak}
v\sb{k}\circ\ak{k}~=~\Fk{k}\circ\oi{n-k-2}~.
\end{myeq}
\noindent Similarly, \w{a:=v\sb{k-1}\circ\ak{k-1}} is nullhomotopic
if and only if \w{\Fk{k-1}} extends the chain map to dimension \w[.]{k-1}
Hence it remains to show that there is a choice of nullhomotopy \w{\Fk{k}} such
that the induced map $a$ will also be nullhomotopic.

Recall from \cite[\S 2]{SpanS} that choices of (homotopy classes of) nullhomotopies
for the map \w{v\sb{k} \circ a\sb{k}:\osB{n-k-2} \to \cM{k}{\Xd}} are in one-to-one
correspondence with homotopy classes
\w[,]{[\eta] \in [\Sigma\osB{n-k-2},\cM{k}{\Xd}]} where $\eta$ acts on \w{\Fk{k}}
by concatenation to yield \w{\Fk{k} \star (\eta\circ \oq{n-k-2})}
(see \S \ref{dconcat}). Furthermore, replacing \w{\Fk{k}} by
\w{\Fk{k}':=\Fk{k} \star (\eta\circ \oq{n-k-2})} changes
\w{[a]} to \w[.]{[a']:=[a] + [\diff{k}{C} \circ \eta]=[a] + \diff{k}{\bV} [\eta]}

Since \w[,]{0=\wdiff{k-1}{C}\circ v\sb{k-1} \circ a\sb{k-1}} it follows that
\w{0=[\diff{k-1}{C}]\circ (-[a])=\diff{k-1}{\bV}(-[a])} in
$$
\piA\cM{k-2}\Xd\lin{\osB{n-k-1}}~=~\cM{k-2}\piA\Xd\lin{\osB{n-k-1}}~=~
\cM{k-2}\Vd\lin{\osB{n-k-1}}
$$
\noindent using Lemma \ref{lmoore}, so \w[.]{-[a] \in \cZ{k-1}\Vd\lin{\osB{n-k-1}}}
By acyclicity of \w[,]{\Vd} there is a class
$$
  [\eta] \in \cM{k}\Vd\lin{\osB{n-k-1}}~=~\piA\cM{k}\Xd\lin{\osB{n-k-1}}~=~
  [\Sigma\osB{n-k-2},\cM{k}\Xd]
$$
\noindent with \w[.]{-[a]= \diff{k}{\bV}[\eta]}  Therefore, replacing \w{\Fk{k}} by
\w{\Fk{k}' = \Fk{k} \star (\eta\circ \oq{n-k-2})}
yields a nullhomotopic \w[,]{a'} and thus allows us to extend $F$ to
dimension \w[.]{k-1}
\end{proof}

\begin{mysubsection}{Long Toda brackets}
\label{slont}
Proposition \ref{pobst} suggests the following quick (if somewhat ad hoc)
definition of long Toda brackets as the last in a bigraded collection of obstructions
for rectifying certain diagrams:

Assume given an \wwb{n+1}homotopy chain complex
\begin{myeq}\label{eqhochan}
\bY\sb{n}~\xra{d\sb{n}}~\bY\sb{n-1}~\xra{d\sb{n-1}}~\bY\sb{n-2}~\to~\dotsc~
\to\bY\sb{0}~\xra{d\sb{0}}~\bY\sb{-1}
\end{myeq}
\noindent in \w{\ho\C} \wwh so \w{d\sb{k-1}\circ d\sb{k}\sim 0} for
\w[.]{1\leq k\leq n} Assume further by induction that we have rectified the final
$n$-segment and replaced it by a diagram:
\begin{myeq}\label{eqnmochain}
\bC\sb{n-1}~\xra{\df{n-1}}~\bC\sb{n-2}~\to\dotsc \to\bC\sb{0}~\xra{\df{0}}~
\bC\sb{-1}~,
\end{myeq}
\noindent which is (strongly) fibrant (in the injective model structure on
\w[).]{\ChC\sp{\leq n-1}}
This means that each \w{\bC\sb{k}\simeq\bY\sb{k}} and \w{\df{k-1}\circ\df{k}=0}
for  \w[,]{1\leq k<n} and that we have a map
\w{\widehat{\partial}:\bY\sb{n}\to\bC\sb{n-1}} such that
\w[.]{\df{n-1}\circ\widehat{\partial}\sim 0} This defines a
``chain map up to homotopy'' \w{\Phi:\bY\sb{n}\oS{n-1}\to\bCs} between two
\wwb{n-1}truncated (augmented) chain complexes over $\C$.

Using the standard strongly cofibrant replacement \w{\Dsn{n}(\bY\sb{n})}
for \w{\bY\sb{n}\oS{n-1}} (see \S \ref{crsar}.A), we can try to realize $\Phi$
by a strict map of chain complexes \w[,]{F:\bDs\to\bCs} constructed by a
downward induction on \w[.]{-1\leq k\leq n-1} The successive obstructions
to doing so are the maps \w{\ak{k}:\Sigma\sp{n-k-1}\bY\sb{n}\to\cZ{k}\bCs}
of the proof of Proposition \ref{pobst}.

As we saw in that proof, a partial chain map
\w{(\Fk{i}:\bD\sb{i}\to\bC\sb{i})\sb{i=k+1}\sp{n}}
can be extended to dimension $k$ if and only if \w[.]{v\sb{k}\circ\ak{k}\sim 0}
Thus we think of the homotopy classes of \w{v\sb{k}\circ\ak{k}} \wb{k\geq0} as the
\emph{intermediate obstructions} to obtaining the \emph{value}
\w{[v\sb{-1}\circ\ak{-1}]\in[\Sigma\sp{n-1}\bY\sb{n},\,\bC\sb{-1}]}
of the \emph{$n$-th order Toda bracket} \w[.]{\lra{d\sb{0},\dotsc,d\sb{n}}}
(In fact, \w{v\sb{-1}} is the identity in this last case.)

See \cite{BBGondH,BMarkH,BJTurnHH,BJTurnC,BBSenT} for more conceptual alternative
definitions of higher Toda brackets.
\end{mysubsection}

\supsect{\protect{\ref{crsar}}.C}{Passage to simplicial objects}

Having constructed a realization \w{F:\bDs\to\cMs{\Xd}=\cMs{\U\Xd}} (see \S \ref{dmco})
of the $n$-th algebraic attaching
map for \w[,]{\Vd} as described in \S \ref{crsar}.B, we wish to complete the passage
from \w{\Xd=\W{n-1}} to a new augmented simplicial object \w{\Xd[F]=\W{n}} as
in \wref[,]{eqftower} in such a way that \w{\Xd[F]} will still be Reedy fibrant
and cofibrant, and the inclusion \w{j:\Xd\hra\Xd[F]} will be a Reedy cofibration
(two properties which will be needed for future applications).

For this purpose, let \w{\tF:\E\bDs\to\U\Xd} be the adjoint of $F$ (see
\S \ref{schaincx}), with \w{\ell:\U\Xd \to\Cone(\tF)} the natural inclusion
into the cone (see \S \ref{dchaincof}). Note that $\ell$ is an acyclic cofibration in
simplicial dimensions \www[,]{\leq n-1} so the same is true of \w[.]{\cL(\ell)}

We add on the degeneracies to obtain \w[,]{\hXd[F]} defined to be the following pushout
in the category of augmented simplicial objects (all having the given object $\bY$
in degree $-1$):
\mydiagram[\label{zhatzero}]{
\cL\U\Xd \ar[d]_{\cL(\ell)} \ar[r]^{\theta} & \Xd \ar[d]^{\wj} \\
\cL \Cone(\tF)\ar[r] & \hXd[F]
}
\noindent where $\theta$ is the counit of the adjunction
(compare \wref[).]{eqnskeleton} Again, $\wj$ is an acyclic cofibration in
dimensions \www[.]{<n}

Choose a Reedy fibrant replacement \w{p':\hXd[F]\xra{\simeq} \Xd'[F]} by
factoring \w{\hXd[F]\to\ast} as an acyclic cofibration followed by a fibration
in the model category \w{\C\sp{\Dp\op}} of
\S \ref{rmodsso}.  Since $\bY$ is fibrant in $\C$, we can choose
\w{\Xd'[F]} to still have $\bY$ in degree $-1$, because no compatibility
is required in that lowest degree.

Finally, factor the composite \w{\Xd\xra{\widehat j}\hXd[F]\xra{p'} \Xd'[F]}
as a cofibration followed by an acyclic fibration to obtain
\w[,]{\Xd \xra{j} \Xd[F] \xepic{p} \Xd'[F]} where \w[,]{\Xd[F]} our candidate for
\w{\W{n}} in \wref[,]{eqftower} is now Reedy fibrant and cofibrant,
since \w{\Xd} is Reedy cofibrant by assumption, and the map
\w{\prn{n}:\W{n-1}\to\W{n}} is the Reedy cofibration $j$ (which is the identity
in degree $-1$):
\mydiagram[\label{eqdoublereplace}]{
  \W{n-1}\ar[d]\sb{\prn{n}} \ar@{=}[r] &\Xd \ar@{^{(}->}[d]_{j} \ar[rr]^{\wj} &&
  \hXd[F] \ar@{^{(}->}[d]^{\sim}_{p'} \ar[dr] \\
\W{n} \ar@{=}[r] & \Xd[F] \ar@{->>}[rr]^{\sim}\sb{p} && \Xd'[F] \ar@{->>}[r] & \ast
}

\begin{lemma}\label{lconef}
The objects \w{\hXd[F]} and \w{\Xd'[F]} constructed as above are Reedy cofibrant
\end{lemma}

\begin{proof}
By Definition \ref{dchaincof} we have the following explicit description
of \w[:]{\Cone(\tF)}
\begin{myeq}\label{eqdopb}
\Cone(\tF)\sb{k}~=~\bX\sb{k}\amalg\oCsB{n-k-1}
\end{myeq}
\noindent for \w[,]{0\leq k\leq n} where the new $0$-th face map is
\w{d\sb{0}\rest{\oCsB{n-k-1}}=\Fk{k-1}} (landing in
\w[)]{\cM{k-1}\Xd\subseteq\bX\sb{k-1}\subseteq\Cone(\tF)\sb{k-1}}
and the new first face map is
\w{d\sb{1}\rest{\oCsB{n-k-1}}=\partial\sp{\bDs}\sb{k-1}=\oi{n-k}\circ\oq{n-k}}
(landing in
\w[).]{\oCsB{n-k}\subseteq\cM{k-1}\Xd\subseteq\bX\sb{k-1}\subseteq\Cone(\tF)\sb{k-1}}
All other face maps \w{d\sb{j}} for \w{j\geq 2} restrict to $0$ on
\w[.]{\oCsB{n-k-1}}

If we use \wref{eqslatch} to define \w{L\sb{n}\Gd} (where
\w[,]{\oG{k}:=\oCsB{n-k-1})} we see that \w{\hXd[F]} may be described explicitly by
\begin{myeq}\label{eqdopc}
\begin{split}
  \hXd[F]\sb{k}~=&~\bX\sb{k}\amalg\oCsB{n-k-1}\amalg L\sb{k}\Gd\\
  =&~ \bX\sb{k}~\amalg~
\coprod\sb{0< k\leq r}~\coprod\sb{0\leq i\sb{1}<\dotsc<i\sb{k}\leq r-1}~\oG{r-k}~,
\end{split}
\end{myeq}
\noindent as in the proof of Lemma \ref{lscwo}.  Moreover,
\w{L\sb{k}\hXd[F]} splits naturally as the coproduct of \w{L\sb{k}\Xd} and
\w[,]{L\sb{k}\Gd} and the map \w{\sigma\sb{k}:L\sb{k}\hXd[F]\to\hXd[F]\sb{k}}
of \S \ref{dlmo} is the coproduct of \w{\sigma\sb{k}:L\sb{k}\Xd\to\bX\sb{k}}
(which is a cofibration in $\C$, since \w{\Xd} is Reedy cofibrant) and the
inclusion \w{L\sb{k}\Gd\hra L\sb{k}\Gd\amalg\oG{k}} (which is also a cofibration
since each \w[,]{\oG{i}} and thus \w{L\sb{k}\Gd} and \w[,]{\oG{k}} are
cofibrant in $\C$).
\end{proof}

\begin{lemma}\label{lextnat}
The construction of \w{\Zd \to \Zd[F]} is natural in the sense that whenever the diagram
\mydiagram{
\Ds \ar[r]^-{F} \ar[d]_{H} & \cMs(\Yd) \ar[d]^{\cMs(h)} \\
\Es \ar[r]^-{G} & \cMs(\Zd)
}
\noindent commutes in \w[,]{\ChC} there is an induced commutative diagram
\mydiagram{
\Yd \ar[r]^-{j_F} \ar[d]_{h} & \Yd[F] \ar[d] \\
\Zd \ar[r]^-{j_G} & \Zd[G] ~.
}
\noindent in \w[.]{\C\sp{\Dp\op}}
\end{lemma}

\begin{proof}
  Take the adjoint of the original square and extend to cones by the naturality
  in Definition \ref{dchaincof} to produce a commutative diagram
\mydiagram{
  \E(\Ds) \ar[d]_{\E(H)} \ar[r]^{\tF} & \U\Yd \ar[d]^{\U(h)} \ar[r]^-{\ell_Y} &
  \Cone(\tF) \ar[d] \\
\E(\Es) \ar[r]^{\tG} & \U\Zd \ar[r]^-{\ell_Z} & \Cone(\tG) ~.
}
\noindent The right square, together with the naturality square for the unit
of adjunction $\theta$, combine to produce a map of pushouts sitting in a
commutative square of simplicial objects
\mydiagram{
\Yd \ar[d]_{h} \ar[r] & \widehat{\Yd}[F] \ar[d] \\
 \Zd  \ar[r] & \widehat{\Zd}[G]
}
\noindent in \w[.]{\C\sp{\Dp\op}} Now recall that by Assumption \ref{sass} we
have functorial factorizations in $\C$, and thus in \w{\C\sp{\Dp\op}}
with respect to the Reedy model category (see the constructions in \cite[\S 15.3]{PHirM}).
\end{proof}

\supsect{\protect{\ref{crsar}}.D}{Sequential realizations of algebraic resolutions}

We may now summarize the procedure described above in the following

\begin{defn}\label{dsrar}
Assume given a CW-resolution \w{\Vd} of a realizable \PAal \w[,]{\Lambda=\piA\bY}
with CW basis \w[.]{\{\oV{n}\}\sb{n=0}\sp{\infty}}
A \emph{sequential realization of \w{\Vd} for $\bY$} is a tower
\begin{myeq}\label{eqtower}
\W{0}~\xra{\prn{1}}~\W{1}~\xra{\prn{2}}~\W{2}~\to~\dotsc~
\W{n-1}~\xra{\prn{n}}~\W{n}~\to~\dotsc~
\end{myeq}
\noindent (see (2.15)) of Reedy cofibrations between Reedy fibrant and cofibrant augmented
simplicial objects (in \w[)]{\C\sp{\Dp\op}} together with objects \w{\oW{n}}
realizing the given CW basis \PAal \w[,]{\oV{n}} such that:
\begin{enumerate}[(i)]
\item The augmented simplicial object \w{\W{n}} realizes
  \w{\Vd\to\Lambda} through simplicial dimension $n$ \wh that is, the $n$-truncation
  of the augmented simplicial \PAal \w{\piA\W{n}\to\piA\bY} is isomorphic to
  the $n$-truncation of \w[,]{\Vd\to\Lambda} with \w{\Wn{-1}{n}=\bY} and
  \w[.]{\prn{n}\sb{-1}=\Id\sb{\bY}}
\item \label{extendseqreal} Each \w{\W{n}=\W{n-1}[\Fn{n}]} (as in \S \ref{crsar}.C) where
  \w{\Fn{n}:\Dsn{n}\to\cMs(\W{n-1})} realizes the attaching map \w{\odz{V\sb{n}}}
  as in \S \ref{crsar}.B.
\item \label{approxseqreal} We have an acyclic cofibration
  \w{\Tn{n}:\Dsn{n}(\oW{n})\to\Dsn{n}} of chain complexes in the projective model
  category structure, where \w{\Dsn{n}(\oW{n})} is the standard strongly cofibrant
  replacement for \w[,]{\oW{n}\oS{n-1}} as in \S \ref{crsar}.A.
\end{enumerate}

A finite tower as in \wref{eqtower} ending at \w{\W{N}} will be called
an $N$-\emph{stage sequential realization of \w{\Vd} for $\bY$}.
\end{defn}

\begin{lemma}\label{rcwreal}
The colimit  \w{\Wd}of \wref{eqtower} (with \w[)]{\bW\sb{-1}=\bY}
realizes \w{\Vd} in the sense that
\w{\piA\Wd\to\piA\bY} is isomorphic to \w[.]{\Vd\to\Lambda}
\end{lemma}

\begin{proof}
  We can deduce from \eqref{extendseqreal} that
  \w{\bW\sb{k}} is the homotopy colimit (over \w[)]{n\geq k} of the objects
  \w[,]{\Wn{k}{n}} so it realizes \w[.]{V\sb{k}}
\end{proof}

\begin{remark}
  Condition \eqref{approxseqreal} of Definition \ref{dsrar} implies that for each
  \w{0\leq k\leq n-1} we have a commutative diagram of horizontal cofibration sequences
\mydiagram[\label{eqmodconesusp}]{
\Sigma\sp{k}\oW{n}~\quad \ar[d]_(0.4){\simeq}^(0.4){\sigma\sp{k}}
\ar@{^{(}->}[r]\sp{\iota\sp{k}} &
~C\Sigma\sp{k}\oW{n}~\ar@{^{(}->}[d]^(0.4){\tau\sp{k}}_(0.4){\simeq}
\ar@{->>}[r]\sp{q\sp{k}} &
~\Sigma\sp{k+1}\oW{n} \ar@{^{(}->}[d]^(0.4){\sigma\sp{k+1}}_(0.4){\simeq}\\
\osW{k}{n}~\quad \ar@{^{(}->}[r]\sp{\oi{k}} & ~\oCsW{k}{n}~
\ar@{->>}[r]\sp{\oq{k}} & ~\osW{k+1}{n}
}
\noindent in $\C$, in which the vertical maps are all acyclic cofibrations.

By convention, we set \w{\osW{-1}{n}:=\ast} and \w[,]{\osW{0}{n}=\oCsW{-1}{n}:=\oW{n}}
with the identity map as
\begin{myeq}\label{eqconvent}
\oq{-1}~:~\oCsW{-1}{n}~\xra{=}~\osW{0}{n}~.
\end{myeq}
\end{remark}

We now have the following analogue of \cite[Theorem 2.33]{BSenH}:

\begin{thm}\label{tres}
For \w{\bA\in\C} as in \S \ref{setmc}, any CW-resolution \w{\Vd} of a realizable
\PAal \w{\Lambda=\piA\bY} has a sequential realization
$$
\cW=\lra{\W{n},\,\prn{n},\, \Dsn{n},\, \Fn{n},\, \Tn{n}}\sb{n=0}\sp{\infty}.
$$
\end{thm}

\begin{proof}
For each \w[,]{n\geq 0} we choose once and for all a fibrant and
cofibrant object \w{\oW{n}\in\A} realizing \w[.]{\oV{n}}
We construct a sequential realization $\cW$ by induction on $n$.

We begin the induction with  \w{\W{-1}:=\cpd{\bY}} (which is Reedy (co)fibrant
in \w[,]{\C\sp{\Dp\op}} since $\bY$ is (co)fibrant
in $\C$ \wh see Remark \ref{rmodsso}).

Note that because \w{\oV{0}} is a free \PAal[,] the \PAal augmentation
\w{\vare:\oV{0}\to\Lambda} corresponds to a unique element
\w[,]{[\svn{0}]\in\Lambda\lin{\oV{0}}} (see \wref[)]{eqvalfreepa} for which we
may choose a representative \w[,]{\bve{0}\sb{0}:\oW{0}\to\bY} by Lemma \ref{lfreeta}.
Since \w{\Dsn{0}:=\oW{0}\oS{-1}} is already
strongly cofibrant, such a choice of \w{\bve{0}\sb{0}:\oW{0}\to\bY}
defines \w{\Fn{0}:\Dsn{0}\to\cMs\W{-1}} and thus
\w{\hWd{-1}[\Fn{0}]=\cd{\oW{0}}} (augmented to $\bY$), with
\w{\W{0}=\W{-1}[\Fn{0}]} a fibrant and cofibrant replacement for this,
as in \S \ref{crsar}.C.
The general induction step (for \w[)]{n\geq 1} is described in \S \ref{crsar}.A-C. In
particular, Proposition \ref{pobst} yields \w[,]{\Fn{n}:\Dsn{n}(\oW{n}) \to \cMs\W{n-1}}
and thus \w[.]{\W{n}=\W{n-1}[\Fn{n}]}
\end{proof}

\begin{example}\label{egminusone}
  In the case \w{n=1} in the proof of Theorem \ref{tres} (covered in the general
  induction step), we choose a map \w{\oW{1}\to\oW{0}} realizing the first
attaching map \w[,]{\odz{1}:\oV{1}\to V\sb{0}=\oV{0}}
and let \w{\oCsW{0}{1}:=C\oW{1}} (the usual cone). We then
have a 1-truncated augmented simplicial object depicted by
\mydiagram[\label{eqwone}]{
\hWdn{1}{1} \ar@/_{1.5pc}/[d]\sb{\dz{0}} \ar@/^{1.5pc}/[d]\sp{d\sb{1}\sp{0}} &=&
\oW{0} \ar@/_{0.5pc}/[d]\sb{\dz{0}=d\sb{1}\sp{0}=\Id}
& \amalg & \oW{1} \ar@{->>}[dll]\sp{\udz{0}} \ar[drr]\sb{d\sb{1}\sp{0}=\iota}
& \amalg & C\oW{1} \ar@/^{0.5pc}/[d]\sp{\dz{0}=d\sb{1}\sp{0}=\Id} \\
\hWdn{1}{0} \ar[d]^(0.4){\bve{1}} \ar[u]\sp{s\sb{0}} &=&
\oW{0} \ar[d]^(0.4){\bve{0}} \ar@/_{0.5pc}/[u]\sb{=} && \amalg &&
C\oW{1} \ar@/^{0.5pc}/[u]\sp{=} \ar@{.>}[dllll]^{\Fk{-1}} \\
\hWdn{1}{-1} & =& Y ~.}

To define the augmentation \w{\bve{1}:\hWdn{1}{0}\to\bY} extending
\w[,]{\bve{0}} we use the fact that \w{\vare\circ\odz{0}=0} in \w{\PAlg}
to deduce that \w{\bve{0}\circ\udz{0}} is nullhomotopic, and any nullhomotopy
\w{\Fk{-1}} defines \w{\bve{1}} on \w[.]{C\oW{1}}  Now apply the process
of \S \ref{crsar}.C to obtain \w[.]{\W{1}}
\end{example}

%
%
\sect{Comparing sequential realizations}
\label{ccsr}

Sequential realizations, and the resulting simplicial resolutions as
constructed in Section \ref{crsar}, play a central role in our theory of higher
homotopy operations, but they depend on many particular choices.  We now explain
how any two such simplicial spaces are related by a zigzag of maps of a particularly
simple form.

We first note the following general fact about model categories, which allows us to
embed any two weakly equivalent objects as strong deformation retracts of a common target:

\begin{lemma}[\protect{\cite[Lemma 3.1]{BSenH}}]\label{lcylin}
If $X$ and $Y$ are two weakly equivalent fibrant and cofibrant objects in a
pointed simplicial model category $\C$,  there are maps as in the following
commuting diagram
\mydiagram[\label{eqmodelcat}]{
&Y\ar@/_6em/[dddd]^{\Id_Y} \ar@/_1em/[ddrrr]^(0.3){u}_(0.3){\simeq} \ar@{^{(}->}[r]^-{\inc}
    & X\amalg Y \ar[dddd]_{\phi} \ar[rrdd]\sp{k\bot u} \ar@{^{(}->}[rrrrd]\sp{k'\bot u'}
\ar@/^{4.0pc}/[rrrrrrdd]_{F} \\
X \ar[dd]^{\Id_X} \ar@{^{(}->}[rru]^(0.3){\inc}
\ar@/_1em/[rrrrd]_(0.3){k}^(0.3){\simeq} |!{[rr];[rrdd]}\hole &&
 &&&&
Z' \ar@{->>}[lld]^{p}_{\simeq}
\ar[dd]\sb{i\circ p}\sp{\simeq}\ar@{^{(}->}[rrd]\sb{e}\sp{\simeq}&&\\
&&&& \hat{Z} \ar@{^{(}->}[rrd]\sp{i}\sb{\simeq} \ar[lldd]\sp{r\top\ell}
        \ar@/_1em/[ddlll]^(0.7){\ell}_(0.7){\simeq}
\ar@/_1em/[lllld]^(0.7){\simeq}_(0.7){r} |!{[llu];[lld]}\hole &&&&
Z \ar@{->>}[lld]\sb{q}\sp{\simeq} \ar@/^{4.0pc}/[lllllldd]_{G} \\
X &&  &&&&
Z'' \ar@{->>}[lllld]\sp{r'\top\ell'} \vsm & \\
& Y & X\times Y \ar@{->>}[llu]^(0.7){\proj} \ar@{->>}[l]^-{\proj}
}
\noindent with (co)fibrations and weak equivalences as indicated, such that
\[
\phi=(\Id\sb{X}\bot r\circ u)\top(\ell\circ k\bot\Id\sb{Y})=
    (\Id\sb{X}\top \ell\circ k)\bot(r\circ u \top\Id\sb{Y}):X\amalg Y\to X\times Y
\]
\noindent factors as \w[,]{X\amalg Y \xmonic{F} Z \xepic{G} X\times Y}
where $F$ is a cofibration which is an acyclic cofibration on each summand,
and $G$ is a fibration which is an acyclic fibration onto each factor.
\end{lemma}

\begin{defn}\label{dacomp}
  Given two CW resolutions \w{\vare:\Vd\to\Lambda} and \w{\varep:\Vdp\to\Lambda}
  of the same \PAal $\Lambda$, with CW bases \w{(\oV{n})\sb{n=0}\sp{\infty}}
  and \w[,]{(\oVp{n})\sb{n=0}\sp{\infty}} respectively, an
  \emph{algebraic comparison map} \w{\Psi:\Vd\to\Vdp} between them is a system
\begin{myeq}\label{eqacomp}
\Psi~=~\lra{\varphi,\,\rho,\,(\ophl{n},\,\orh{n})\sb{n=0}\sp{\infty}}~,
\end{myeq}
\noindent where \w{\varphi:\Vd\to\Vdp} is a split monic weak equivalence of
simplicial \PAal[s] with retraction \w{\rho:\Vdp\to\Vd}
(satisfying \w[),]{\varep\circ\ophl{0}=\vare} induced by inclusions of coproduct
summands \w{\ophl{n}:\oV{n}\hra\oVp{n}} with retractions \w{\orh{n}:\oVp{n}\to\oV{n}}
for each \w[.]{n\geq 0}
\end{defn}

In this context we can sharpen Lemma \ref{lcylin} as follows:

\begin{lemma}[\protect{\cite[Lemma 3.7]{BSenH}}]\label{lcylind}
For any two free CW resolutions \w{\varu{i}:\Vud{i}\to\Lambda} of the same
\PAal $\Lambda$, with CW bases \w{\left(\ouV{n}{i}\right)\sb{n=0}\sp{\infty}}
\wb[,]{i=0,1} there is a CW resolution \w{\vare:\Vd\to\Lambda} with CW basis
\w{\left(\oU{n}\amalg\ouV{n}{0}\amalg\ouV{n}{1}\right)\sb{n=0}\sp{\infty}} and  algebraic
comparison maps \w{\Psi\up{i}:\Vud{i}\to\Vd} for \w[.]{i=0,1}
\end{lemma}

\begin{defn}\label{dcompar}
Given an algebraic comparison map
\w{\Psi=\lra{\varphi,\,\rho,\,(\ophl{n},\,\orh{n})\sb{n=0}\sp{\infty}}}
between \w{\Vd} and \w[,]{\Vdp}
as in \S \ref{dacomp} and sequential realizations $\cW$ of \w{\Vd} and
\w{\ccWp} of \w[,]{\Vdp} a \emph{comparison map \w{\Phi:\cW\to\,\ccWp} over $\Psi$}
is a system
\begin{myeq}\label{eqcorresp}
\Phi~=~\lra{\en{n},\,\rn{n},~
\hon{}{n},~
\hor{}{n},~\oon{}{n},~
\oor{}{n}\,}\sb{n=0}\sp{\infty}
\end{myeq}
\noindent consisting of:
\begin{enumerate}
\renewcommand{\labelenumi}{(\roman{enumi})~}
\item A split augmented simplicial map \w{\en{n}:\W{n}\to\Wp{n}}
  with retraction \w[,]{\rn{n}:\Wp{n}\to\W{n}} realizing \w{\varphi:\Vd\to\Vdp}
and  $\rho$, respectively,  through simplicial dimension $n$.
\item a split cofibration of chain complexes, \w{\hon{}{n}:\Dsn{n} \to \Dsnp{n}}
  with retraction \w[.]{\hor{}{n}:\Dsnp{n} \to \Dsn{n}}
\item a split cofibration \w{\oon{}{n}:\oW{n} \to \oWp{n}} of objects in $\C$,
  realizing \w{\ophl{n}} with retraction \w{\oor{}{n}:\oWp{n} \to \oW{n}}
  realizing \w[.]{\orh{n}}
\end{enumerate}

Note that \w{\oon{}{n}} induces a map of chain complexes
\w{\oon{\ast}{n}:\Dsn{n}(\oW{n}) \to \Dsnp{n}(\oWp{n})} with
retraction \w{\oor{\ast}{n}:\Dsnp{n}(\oWp{n}) \to \Dsn{n}(\oW{n})}
induced by \w[.]{\oor{}{n}} We then require that the following diagram
in chain complexes commute (in each vertical direction):
\mydiagram[\label{eqmapchaincpx}]{
\Dsn{n}(\oW{n})~ \ar@{^{(}->}[d]\sp{\oon{\ast}{n}} \ar@{>->}[rr]^{\Tn{n}} &&
\Dsn{n} \ar@{^{(}->}[d]\sp{\hon{}{n}} \ar@{->>}[rr]^{\Fn{n}} &&
\cMs\W{n-1} \ar@{^{(}->}[d]\sp{\cMs\en{n-1}} \\
\Dsnp{n}(\oWp{n})~ \ar@/^{1.1pc}/[u]\sp{\oor{\ast}{n}} \ar@{>->}[rr]_{\ppp\Tn{n}} &&
\Dsnp{n} \ar@/^{1.1pc}/[u]\sp{\hor{}{n}} \ar@{->>}[rr]_{\ppp\Fn{n}} &&
\cMs\Wp{n-1} \ar@/^{1.1pc}/[u]\sp{\cMs\rn{n-1}} ~.
}

If \w{\enk{n}{}:\W{n}\hra\Wp{n}} (hence also \w[)]{\rnk{n}{}:\Wp{n}\hra\W{n}}
is a Reedy weak equivalence, and in addition each induced map
\w[,]{\hon{k}{n}:\osW{k}{n}\hra\osWp{k}{n}} (hence each map
\w[)]{\hor{k}{n}:\osWp{k}{n}\hra\osW{k}{n}} is a weak equivalence in $\C$,
we say that $\Phi$ is a \emph{trivial} comparison map.

If we only have
\begin{myeq}\label{eqncorresp}
\Phi~=~\lra{\en{n},\,\rn{n},~
\hon{}{n},~
\hor{}{n},~\oon{}{n},~
\oor{}{n}\,}\sb{n=0}\sp{N}
\end{myeq}
\noindent as above, we say that \w{\Phi:\cW\to\,\ccWp} is an \emph{$N$-stage
comparison map} over $\Psi$. This completes our Definition.
\end{defn}

\begin{remark}\label{rtechdesccm}
  We note  for future reference that a comparison map $\Phi$
  as above yields maps fitting into commuting diagrams as follows:

\mydiagram[\label{eqmapmodconesusp}]{
\osW{k}{n} \ar@{^{(}->}[d]\sp{\oon{k}{n}}  \ar@{^{(}->}[rr]\sp{\oi{k}} &&~
\oCsW{k}{n} \ar@{^{(}->}[d]\sp{\Con{k}{n}} \ar@{->>}[rr]\sp{\oq{k}} &&
\osW{k+1}{n} \ar@{^{(}->}[d]\sp{\oon{k+1}{n}}\\
\osWp{k}{n}  \ar@/^{1.1pc}/[u]\sp{\oor{k}{n}} \ar@{^{(}->}[rr]\sp{\ppp\oi{k}} &&
\oCsWp{k}{n} \ar@/^{1.1pc}/[u]\sp{\Cor{k}{n}} \ar@{->>}[rr]\sp{\ppp\oq{k}} &&
\osWp{k+1}{n} \ar@/^{1.1pc}/[u]\sp{\oor{k+1}{n}}
}
\noindent for each \w[,]{0\leq k<n} in which both upward and downward squares
commute, as well as satisfying
\w{\Cor{k}{n}\circ\Con{k}{n}=\Id} and \w[.]{\oor{k}{n}\circ \oon{k}{n}=\Id}

Moreover, for each \w[,]{0\leq k<n} both squares in the following diagram commute:
\myudiag[\label{eqcompf}]{
\oCsW{n-k-1}{n} \ar@{^{(}->}[rr]\sp{\Fk{k}} \ar@{_{(}->}[d]\sp{\Con{n-k-1}{n}}
&& \cM{k-1}\W{n-1}  \ar@{_{(}->}[d]\sb{\cM{k-1}\en{n-1}}
&&
\oCsW{n-k-1}{n} \ar@{^{(}->}[rr]\sp{\Fk{k}}
&& \cM{k-1}\W{n-1}  \\
\oCsWp{n-k-1}{n} \ar@{^{(}->}[rr]\sp{\Fkp{k}}
&& \cM{k-1}\Wp{n-1} &&
\oCsWp{n-k-1}{n} \ar[u]\sb{\Cor{n-k-1}{n}} \ar@{^{(}->}[rr]\sp{\Fkp{k}}
&& \cM{k-1}\Wp{n-1}  \ar[u]\sp{\cM{k-1}\rn{n-1}}
}
\end{remark}

\begin{remark}\label{rcomp}
Consider the (strict) cofibration sequence:
\begin{myeq}\label{eqcofseq}
\osW{k}{n}~\xra{\oon{k}{n}}~\osWp{k}{n}~\xra{\ofn{n}{k}}~\osX{k}{n}
\end{myeq}
\noindent (which defines the right map and space).
Because of the splitting \w{\oor{k}{n}} for \w[,]{\oon{k}{n}} mapping \wref{eqcofseq}
into \w{\osWp{k}{n}} yields a Puppe exact sequence. Since
\w{[\Id-\oon{k}{n}\circ \oor{k}{n}]\in[\osWp{k}{n},\,\osWp{k}{n}]} is in
\w[,]{\Ker((\oon{k}{n})\sp{\#})} we obtain a map
\w{\snk{n}{k}:\osX{k}{n}\to \osWp{k}{n}} with
\w[.]{\snk{n}{k}\circ\ofn{n}{k}\sim\Id-\oon{k}{n}\circ\oor{k}{n}}
Thus
\mydiagram[\label{eqwkprod}]{
\osWp{k}{n} \ar[rr]_(0.42){\simeq}^(0.42){\ofn{n}{k} + \oor{k}{n}} &&
\osX{k}{n}\amalg\osW{k}{n}\ar[rr]_(0.6){\simeq}^(0.6){\snk{n}{k}\bot\oon{k}{n}} &&
\quad~~\osWp{k}{n}
}
\noindent are inverse weak equivalences for each \w[.]{0\leq k< n}
\end{remark}

\begin{defn}\label{dwesr}
We say two $n$-stage  sequential realizations $\cW$ and \w{\cWp} for \w{\bY\in\C}
are \emph{weakly equivalent} if there is a finite
zigzag of cospans of $n$-stage comparison maps connecting \w{\cWp} to $\cW$, say
\mydiagram[\label{eqnzigzag}]{
& \cuW{1} & \dotsc & \cuW{N-1} & \\
\cW=\cuW{0} \ar[ru]\sp{\Phip{1}} &&
\cuW{2} \ar[lu]\sb{\Phip{2}} & \dotsc & \cuW{N}=\cWp \ar[lu]\sb{\Phip{N}}
}
\noindent We say sequential realizations $\cW$ and \w{\cWp} for \w{\bY\in\C}
are \emph{weakly equivalent} if each of their $n$-stage approximations are such, with
respect to a given (possibly infinite) zigzag which is ``locally finite'' in the sense
that for each $n$, all but a finite number of maps in the zigzag are isomorphisms
on the $n$-truncations.
\end{defn}

The following result is used below to show that our main constructions
  are independent of choices of sequential realizations;
it is dual to \cite[Theorem 3.20]{BBSenH}. The proof is in Appendix \ref{Appa}.

\begin{thm}\label{tzigzag}
Given two $\bA$-equivalent spaces $\bY$ and \w{\bYp}  with
\w[,]{\Lambda\cong\piA\bY\cong\piA\bYp} any two sequential realizations
$\cW$ and \w{\cWp} of two CW resolutions \w{\Vd\to\Lambda}
  and \w{\Vdp\to\Lambda} for $\bY$ and \w[,]{\bYp} respectively, are weakly
      equivalent in the sense of Definition \ref{dwesr}.
\end{thm}

%
%
\sect{Higher homotopy operations}
\label{chho}

We are now in a position to define our notion of higher homotopy operations
based on sequential realizations. These are simpler than the full simplicial
operations studied in \cite{BMarkH,BJTurnHH,BBSenH}, though not strictly linear
in the sense of \cite[\S 6]{BJTurnHA} (see also \cite[\S 7]{BJTurnC}).

Our operations appear as the successive obstructions to augmenting a given simplicial
object \w[,]{\Wd} obtained as the colimit of \wref{eqtower} for some sequential
realization $\cW$, to a fixed object \w[:]{\bZ\in\C} the $n$-th operation
will be the obstruction to extending an augmentation \w{\W{n}\to\bZ} to \w[.]{\W{n+1}}
Thus in this and the following three sections we will be working with unaugmented
(possibly restricted) simplicial objects, implicitly applying \w{\szero{-}}
(see \ref{snac}) throughout.

\begin{remark}\label{raug}
An augmentation from a simplicial object \w{\Wd} to $\bX$ in $\C$ is just a map
\w{\Wd\to\cd{\bX}} in \w[.]{\C\sp{\Dop}} However, since the target is not fibrant as
an unaugmented simplicial object (see Remark \ref{rmodsso}), we choose once and for
all a fibrant replacement \w{\cd{\bX}\xra{\simeq}\Ud\epic\ast} in the Reedy model
category \w[.]{\C\sp{\Dop}} We thus can think of a \emph{homotopy augmentation}
\w{\bv:\Wd\to\Ud} as a homotopy meaningful version of an augmentation (which
need not factor through a strict augmentation \w[,]{\vare:\Wd\to\cd{\bX}} in general).

Recall now that the standard cosimplicial simplicial set \w{\Du} is given by the
standard simplicial maps between the standard simplices
\w[,]{(\Deln{k})\sb{k=0}\sp{\infty}} where \w{\eta\sp{i}:\Deln{k-1}\hra\Deln{k}}
is the inclusion of the $i$-th facet, and \w{\sigma\sp{j}:\Deln{k}\epic\Deln{k-1}}
is the $j$-th collapse map (see \cite[X, 2.2]{BKanH}). Applying the simplicial
exponentiation \w{\bX\sp{(-)}} to each \w{\Deln{k}} (for the fixed fibrant object
$\bX$) yields a simplicial object \w{\Ud:=\bX\sp{\Du}} in \w[,]{\C\sp{\Dop}}
which will serve as our canonical fibrant replacement for \w[.]{\cd{\bX}}
The Reedy weak equivalence \w{p\sp{\ast}:\cd{\bX}\to\bX\sp{\Du}}
is induced by \w[.]{p:\Du\to\ast}
\end{remark}

\begin{defn}\label{dshho}
We distinguish two levels of data needed to define our higher operations:
the \emph{basic initial data} consists of
\begin{equation*}
\begin{split}
(\star)=\lra{\bY,\bX,\vth},\hsn&\text{where $\bY$ and $\bX$ are objects in $\C$ and
    \w{\vth:\piA\bY\to\piA\bX}}\\
    & \text{is a map of \PAal[s.]}
\end{split}
\end{equation*}
\noindent while the \emph{specific initial data} consists in addition of
$$
(\star\star)=\lra{\cW,\Eu{0}},\hsn\text{with}\begin{cases}
\text{$\cW$ a sequential realization  of a CW-resolution
    \w[,]{\Vd\xra{\vare}\piA\bY}}& \\
\text{and \w{\Eu{0}:\W{0}\to\Ud} realizing}
  \w[.]{\vth\circ\vare\sb{0}:V\sb{0}\to\piA\bX}
\end{cases}
$$
\noindent Further conditions on the map \w{\Eu{0}} will depend on the specific
contexts we have in mind in Sections \ref{chhio} and \ref{chhim} below.

Given \w[,]{(\star\star)} our goal is to extend \w{\Eu{0}} inductively to
\emph{$n$-maps} \w{\Eu{n}:\W{n}\to\Ud} realizing \w{\vth\circ\vare:\Vd\to\piA\bX}
through simplicial dimension $n$. A \emph{strand} for \w{(\star\star)} is an
infinite sequence \w{\Eu{\infty}=(\Eu{0},\Eu{1},\dotsc)} of such $n$-maps satisfying
\w{\Eu{n-1}=\Eu{n}\circ\prn{n}} for each \w[.]{n \geq 1}
\end{defn}

\begin{remark}\label{ridea}
We would like to think of
\begin{myeq}\label{eqcofseqr}
\E\bDs~\xra{\tF}~\U\W{n-1}~\xra{\ell}~\Cone(\tF)
\end{myeq}
\noindent of \S \ref{crsar}.C  as a homotopy cofibration sequence of restricted
simplicial objects, with the homotopy class of \w{\Eu{n-1}\circ\tF:\E\bDs\to\U\Ud} \wwh
more precisely, of the realization of corresponding full simplicial objects \wh as our
obstruction to extending the \wwb{n-1}map \w{\Eu{n-1}} to a map
  \w{\widetilde{E}:\Cone(\tFn{n})\to\U\Ud} \wwh and so, by the pushout property of
\wref[,]{zhatzero} to a map  \w[.]{\hEu{n}:\hWd{n-1}[\Fn{n}]\to\Ud}
\end{remark}

\begin{lemma}\label{lextendfmcone}
If we can extend an \wwb{n-1}map \w{\Eu{n-1}:\W{n-1}\to\Ud} for \w{(\star\star)}
to \w[,]{\hEu{n}:\hWd{n-1}[\Fn{n}]\to\Ud} then it extends further to an $n$-map
\w[.]{\Eu{n}:\W{n}\to\Ud}
\end{lemma}

\begin{proof}
Choosing a lift \w{E'} in the diagram
\mydiagram[\label{eqsquara}]{
\hWd{n-1}[\Fn{n}] \ar[rr]^-{\hEu{n}} \ar@{>->}[d]^{\sim}_{p'} && \Ud \ar@{->>}[d] \\
{\W{n-1}}'[\Fn{n}] \ar[rr] \ar@{.>}[urr]^{E'} && \ast
}
\noindent (compare \wref[)]{eqdoublereplace} defines a map
\w[,]{E'':=E' \circ p:\W{n-1}[\Fn{n}]\to\Ud} fitting into the following
solid diagram:
\mydiagram[\label{eqsquarb}]{
  & \W{n-1} \ar[r]^{\wj} \ar@{>->}[d]^{j} \ar@{>->}[dl]_{\prn{n}}
  \ar@{.>}@/^{3.0pc}/[rrr]\sp{\Eu{n-1}}
  & \hWd{n-1}[\Fn{n}] \ar[rr]^-{\hEu{n}} \ar@{>->}[d]_{p'}^{\sim} && \Ud  \\
\W{n} \ar@{=}[r] \ar@{.>}@/_{5.0pc}/[urrrr]\sb{\Eu{n}} & \W{n-1}[\Fn{n}] \ar@{->>}[r]^{p} &
    {\W{n-1}}'[\Fn{n}] \ar[urr]^{E'} \ar@{.>}@/_{2.0pc}/[l]\sb{s} &&
}
\noindent where the composite \w{\W{n-1}\to\Ud} agrees with \w{\Eu{n-1}} only up to
homotopy. Since \w{\prn{n}} is a cofibration, by \cite[Corollary 4.20]{BJTurnR}, we can
alter \w{E''} within its homotopy class to obtain the required
\w[,]{\Eu{n}} with \w[.]{\Eu{n-1}=\Eu{n}\circ\prn{n}}
\end{proof}

\begin{mysubsection}{Extending $n$-maps}
\label{sextnmap}
We need a more concrete identification of the extension of a given
\wwb{n-1}map \w{\Eu{n-1}:\W{n-1}\to\Ud} to \w{\hEu{n}:\hWd{n-1}[\Fn{n}]\to\Ud}
in order both to exhibit the obstruction to obtaining \w{\hEu{n}} as the value
of a higher operation, and to verify that it has the necessary properties
(in particular, the ability to compare values for various sequential realizations).

From the proof of Lemma \ref{lconef}, we see that in order to construct
\w{\hEu{n}} we need entries \w{\hEuk{k}{n}:\oCsW{n-k-1}{n}\to\bX\sp{\Deln{k}}}
for \w{0\leq k\leq n} (see \wref[).]{eqgrid} In order for this \w{\widehat{E}\bp{n}}
to be a simplicial map, we must have:
\begin{myeq}\label{eqneweface}
\begin{cases}
~(\eta\sp{0})\sp{\ast}\circ\hEuk{k}{n} ~=&~
\Euk{k-1}{n-1}\circ w\sb{k-1}\circ\Fk{k-1}~,\\
~(\eta\sp{1})\sp{\ast}\circ\hEuk{k}{n}~=&~\hEuk{k-1}{n-1}\circ\diff{n-k-1}{\bD}~,\\
~(\eta\sp{i})\sp{\ast}\circ\hEuk{k}{n}~=&~0\hsm\text{for}\hsm i\geq 2~,
\end{cases}
\end{myeq}
\noindent where \w{\diff{j}{\bD}} is the differential of \w[,]{\bDs} given by
\wref[,]{eqdif} and \w{\eta\sp{i}:\Deln{k-1}\to\Deln{k}} is the $i$-th coface map
of \w[,]{\Du} given by \w[.]{\Deln{k-1}\cong\partial\sb{i}\Deln{k}\hra\Deln{k}}
\end{mysubsection}

\begin{mysubsection}{Folding polytopes}
\label{sfoldp}
For any \w[,]{\bK\in\C} we can iterate the quotient maps \w{C\bK\epic\Sigma\bK}
to obtain \w[.]{\theta\sp{n-k-1}:C\sp{n-k-1}\oW{n}\to\Sigma\sp{n-k-1}\oW{n}}
Precomposing \w{\hEuk{k}{n}:\oCsW{n-k-1}{n}\to\bX\sp{\Deln{k}}} with
this, together with \w{\tau\sp{n-k-1}} of \wref[,]{eqmodconesusp} yields
$$
C\sp{n-k}\oW{n}~\xra{C\theta\sp{n-k-1}}~
C\Sigma\sp{n-k-1}\oW{n}~\xra{\tau\sp{n-k-1}}~\oCsW{n-k-1}{n}~\xra{\hEuk{k}{n}}~
\bX\sp{\Deln{k}}~.
$$
\noindent Since \w{\Deln{i+1}\cong C\Deln{i}} for each $i$, this composite is
adjoint to a pointed map \w[.]{\tEuk{n}{k}:\oW{n}\otimes\Deln{n}\to\bX}
There are \w{n+1} such maps, corresponding to (co)simplicial dimensions
\w[.]{0\leq k\leq n}

If we denote the copy of \w{\Deln{n}} associated to \w{\hEuk{k}{n}} by
\w[,]{\Delnk{n}{k}} then the first \w{k+1} facets
\w{\partial\sb{0}\Delnk{n}{k},\dotsc,\partial\sb{k}\Delnk{n}{k}} of \w{\Delnk{n}{k}}
are associated with the corresponding facets of the \w{\Deln{k}} in
\w{\bX\sp{\Deln{k}}} under the adjunction,
the next \w{n-k-1} facets are associated to the suspension directions of
\w{C\Sigma\sp{n-k-1}\oW{n}} (so \w{\hEuk{k}{n}} maps them to the basepoint), and the
$n$-th (so last) facet corresponds to the cone direction.

For each \w[,]{n\geq 1} the $n$-th \emph{folding polytope} \w{\Pn{n}}
is then constructed from the disjoint union of the \w{n+1} $n$-simplices
\w{\Delnk{n}{0},\dotsc,\Delnk{n}{n}} by identifying the $n$-facet
\w{\partial\sb{n}\Delnk{n}{k}} of \w{\Delnk{n}{k}} with the $1$-facet
\w{\partial\sb{1}\Delnk{n}{k+1}} of \w{\Delnk{n}{k+1}} for each \w[,]{0\leq k<n}
in keeping with the second line of \eqref{eqneweface}
(see \cite[\S 4]{BBSenH}, and compare the cubical version in \cite[\S 5]{BSenH}).

The sub-simplicial complex of the boundary \w{\partial\Pn{n}} consisting of the
union of the $n$ \wwb{n-1}-facets \w{\partial\sb{0}\Delnk{n}{k}}
\w{k=1,2,\dotsc,n} will be called the \emph{edge} of \w[,]{\Pn{n}} and
denoted by \w[.]{E\Pn{n}} See Figures \ref{fig1} and \ref{fig2},
as well as Figure \ref{fig4}
below.
\end{mysubsection}

%
%
\stepcounter{thm}
\begin{figure}[htbp]
\begin{center}
\begin{picture}(400,50)(30,170)
%
%
%
%
\put(130,200){\circle*{3}}
\put(127,205){\scriptsize $1$}
\put(130,200){\line(1,0){70}}
\put(156,170){$\Delnk{1}{0}$}
\put(200,200){\circle*{3}}
\put(199,205){\scriptsize $0$}
%
%
\bezier{12}(207,195)(215,190)(223,195)
\put(210,193){\vector(-2,1){8}}
\put(220,193){\vector(2,1){8}}
\put(200,180){\scriptsize identify}
%
%
\put(230,200){\circle*{3}}
\put(227,205){\scriptsize $0$}
\put(230,200){\line(1,0){70}}
\put(260,170){$\Delnk{1}{1}$}
\put(300,200){\circle*{5}}
\put(299,205){\scriptsize $1$}
\put(321,180){{\scriptsize the edge} $E\Pn{1}$}
\put(319,184){\vector(-1,1){13}}
\end{picture}
\end{center}
\caption{The folding polytope \ $\Pn{1}$}
\label{fig1}
\end{figure}

%
%
\stepcounter{thm}
\begin{figure}[htbp]
\begin{center}
\begin{picture}(400,150)(50,-5)
%
%
%
%
\put(125,30){\circle*{3}}
\put(123,20){\scriptsize $0$}
\put(125,30){\line(-2,3){46}}
\put(125,30){\line(2,3){46}}
\put(118,-5){$\Delnk{2}{0}$}
\put(78,100){\circle*{3}}
\put(70,98){\scriptsize $2$}
\put(78,100){\line(1,0){94}}
\put(172,100){\circle*{3}}
\put(176,98){\scriptsize $1$}
%
%
\bezier{42}(155,60)(190,40)(225,60)
\put(160,57){\vector(-2,1){8}}
\put(220,57){\vector(2,1){8}}
\put(172,35){\scriptsize identify}
%
%
\put(255,30){\circle*{3}}
\put(253,20){\scriptsize $0$}
\put(255,30){\line(-2,3){46}}
\put(255,30){\line(2,3){46}}
\put(248,-5){$\Delnk{2}{1}$}
\put(208,100){\circle*{3}}
\put(200,98){\scriptsize $2$}
\linethickness{2pt}
\put(208,100){\line(1,0){94}}
\thinlines
\put(302,100){\circle*{3}}
\put(306,98){\scriptsize $1$}
%
%
\bezier{42}(285,60)(320,40)(355,60)
\put(290,57){\vector(-2,1){8}}
\put(350,57){\vector(2,1){8}}
\put(302,35){\scriptsize identify}
\bezier{12}(312,108)(320,113)(328,108)
\put(317,110){\vector(-2,-1){8}}
\put(323,110){\vector(2,-1){8}}
%
%
\put(385,30){\circle*{3}}
\put(383,20){\scriptsize $0$}
\put(385,30){\line(-2,3){46}}
\put(385,30){\line(2,3){46}}
\put(378,-5){$\Delnk{2}{2}$}
\put(338,100){\circle*{3}}
\put(330,98){\scriptsize $2$}
\linethickness{2pt}
\put(338,100){\line(1,0){94}}
\thinlines
\put(432,100){\circle*{3}}
\put(436,98){\scriptsize $1$}
%
%
\put(205,125){$\overbrace{\hspace*{81mm}}\sp{\text{the edge} \ E\Pn{2}}$}
\end{picture}
\end{center}
\caption{The folding polytope \ $\Pn{2}$}\label{fig2}
\end{figure}

We readily see by induction:

\begin{lemma}\label{lfoldpoly}
  For each \w[,]{n\geq 1} the realization of the triple
  \w{(\Pn{n},\partial\Pn{n},E\Pn{n})} is homeomorphic to
  \w[,]{(\bB\sp{n},\,\bS{n-1},\,\bB\sp{n-1}\sb{+})} where
\w{\bB\sp{n-1}\sb{+}} is the upper hemisphere of \w{\bS{n-1}} in the unit ball
\w[.]{\bB\sp{n}}
\end{lemma}

\begin{example}\label{egfoldpoly}
  We show how to use the folding polytopes in the case \w[.]{n=3}
  Given the solid $2$-homotopy for \w{\W{2}} in the following diagram, we wish
  to extend it by the dotted maps to the (given) $3$-truncated restricted
  augmented simplicial space \w[.]{\tWd{3}} For simplicity of notation,
  we assume here that we have used the standard cofibrant replacement
  \w[.]{\Dsn{3}(\oW{3})}
\mywdiag[\label{eqexttothree}]{
& \oW{3} \ar[ld]\sb{d\sb{0}=w\sb{2}\Fk{2}}
\ar[d]^(0.3){d\sb{1}=\diff{2}{\bD}}\ar@{.>}[rrrr]^(0.5){\hEuk{3}{3}}&&&&
\bX\sp{\Deln{3}} \ar@<-4.5ex>[d]\sb{(\eta\sp{0})\sp{\ast}}
\sp{(\eta\sp{1})\sp{\ast}} \ar[d] \ar@<1ex>[d]\sp{(\eta\sp{2})\sp{\ast}}
\ar@<5.5ex>[d]\sp{(\eta\sp{3})\sp{\ast}} \\
\Wn{2}{2} \ar@/^{2.0pc}/[rrrrr]\sp{\Euk{2}{2}}
\ar@<-2.5ex>[d]\sb{d\sb{0}}\sp{d\sb{1}} \ar[d] \ar@<2ex>[d]\sp{d\sb{2}} &
\amalg
\hspace*{4mm}\CsoW{0}{3} \ar[ld]_(0.4){d\sb{0}=w\sb{1}\Fk{1}}
\ar[d]^(0.3){d\sb{1}=\diff{1}{\bD}} \ar@{.>}[rrrr]\sp{\hEuk{2}{3}} &&&&
\bX\sp{\Deln{2}} \ar@<-3.5ex>[d]\sb{(\eta\sp{0})\sp{\ast}}
\sp{(\eta\sp{1})\sp{\ast}} \ar@<1ex>[d] \ar@<3ex>[d]\sp{(\eta\sp{2})\sp{\ast}} \\
\Wn{1}{2} \ar@/^{2.0pc}/[rrrrr]\sp{\Euk{1}{2}}
\ar@<-0.5ex>[d]\sb{d\sb{0}} \ar@<0.5ex>[d]\sp{d\sb{1}} &
\amalg
\hspace*{4mm}\CsoW{1}{3} \ar[ld]\sb(0.4){d\sb{0}=w\sb{0}\Fk{0}}
\ar[d]^(0.3){d\sb{1}=\diff{0}{\bD}}
\ar@{.>}[rrrr]\sp{\hEuk{1}{3}} &&&&
\bX\sp{\Deln{1}} \ar@<-0.5ex>[d]\sb{(\eta\sp{0})\sp{\ast}}
\ar@<0.5ex>[d]\sp{(\eta\sp{1})\sp{\ast}} \\
\Wn{0}{2} \ar@/^{2.0pc}/[rrrrr]\sp{\Euk{0}{2}} & \amalg
\hspace*{4mm}\CsoW{2}{3} \ar@{.>}[rrrr]\sb{\hEuk{0}{3}} &&&& \bX
}

Using the identifications of \S \ref{sfoldp}, the putative map
\w{\hEuk{0}{3}:\oCsW{2}{3}\to\bX} would be given by a map
\w{\tEuk{3}{0}:\oW{3}\times\Delnk{3}{0}\to\bX} whose restriction to the \emph{boundary}
is described in adjoint form on the left in Figure \ref{fig4}.
%
%
\stepcounter{thm}
\begin{figure}[htbp]
\begin{center}
\begin{picture}(410,250)(-9,-25)
%
%
\put(30,215){$\partial\Delnk{3}{0}$}
%
%
\put(60,60){\circle*{3}}
\put(59,50){\scriptsize $2$}
\put(23,125){\circle*{3}}
\put(16,121){\scriptsize $3$}
\put(-15,190){\circle*{3}}
\put(-17,194){\scriptsize $2$}
\put(60,190){\circle*{3}}
\put(59,194){\scriptsize $1$}
\put(135,190){\circle*{3}}
\put(135,194){\scriptsize $2$}
\put(98,125){\circle*{3}}
\put(101,121){\scriptsize $0$}
%
%
\put(-15,191){\line(1,0){150}}
\put(23,126){\line(1,0){75}}
\bezier{400}(60,60)(23,125)(-15,190)
\bezier{400}(60,60)(98,125)(135,190)
\bezier{400}(23,125)(42,157)(60,190)
\bezier{400}(98,125)(79,157)(60,190)
%
%
\put(17,165){$\ast$}
\put(57,150){$\ast$}
\put(88,165){\scriptsize $\widetilde{\hEuk{0}{2}\diff{0}{\bD}}$}
\put(57,97){$\ast$}
%
%
\bezier{40}(115,180)(120,135)(125,90)
\put(116,175){\vector(-1,4){2}}
\put(124,95){\vector(1,-4){2}}
%
%
\put(130,-20){$\partial\Delnk{3}{1}$}
%
%
\put(70,10){\circle*{3}}
\put(69,0){\scriptsize $2$}
\put(108,75){\circle*{3}}
\put(100,73){\scriptsize $3$}
\put(145,140){\circle*{3}}
\put(144,145){\scriptsize $2$}
\put(145,10){\circle*{3}}
\put(144,0){\scriptsize $1$}
\put(220,10){\circle*{3}}
\put(219,0){\scriptsize $2$}
\put(182,75){\circle*{3}}
\put(186,73){\scriptsize $0$}
%
%
\put(70,11){\line(1,0){150}}
\bezier{400}(70,10)(108,75)(145,140)
\bezier{400}(220,10)(182,75)(145,140)
\linethickness{2pt}
\put(70,11){\line(1,0){75}}
\bezier{400}(70,10)(89,42)(108,75)
\bezier{400}(108,75)(127,43)(145,10)
\thinlines
\put(145,11){\line(1,0){75}}
\bezier{400}(145,11)(163,43)(182,75)
\bezier{400}(220,11)(201,43)(182,75)
\put(108,76){\line(1,0){75}}
\bezier{400}(182,75)(164,43)(145,10)
%
%
\put(142,48){$\ast$}
\put(135,98){\scriptsize $\widetilde{\hEuk{0}{2}\diff{0}{\bD}}$}
\put(170,27){\scriptsize $\widetilde{\hEuk{1}{3}\diff{1}{\bD}}$}
\put(90,27){\scriptsize $\widetilde{\Euk{0}{2}w\sb{0}\Fk{0}}$}
%
%
\bezier{40}(180,50)(205,70)(230,90)
\put(185,54){\vector(-1,-1){7}}
\put(227,86){\vector(1,1){7}}
%
%
\put(230,215){$\partial\Delnk{3}{2}$}
%
%
\put(240,60){\circle*{3}}
\put(239,49){\scriptsize $2$}
\put(203,125){\circle*{3}}
\put(196,121){\scriptsize $3$}
\put(165,190){\circle*{3}}
\put(164,194){\scriptsize $2$}
\put(240,190){\circle*{3}}
\put(239,194){\scriptsize $1$}
\put(315,190){\circle*{3}}
\put(314,194){\scriptsize $2$}
\put(278,125){\circle*{3}}
\put(281,121){\scriptsize $0$}
%
%
\put(165,191){\line(1,0){150}}
\put(203,126){\line(1,0){75}}
\linethickness{2pt}
\put(165,191){\line(1,0){75}}
\linethickness{1.5pt}
\bezier{400}(203,125)(184,158)(165,190)
\bezier{400}(203,125)(221,158)(240,191)
\thinlines
\bezier{400}(240,60)(203,125)(165,190)
\bezier{400}(240,60)(278,125)(315,190)
\bezier{400}(203,125)(222,157)(240,190)
\bezier{400}(278,125)(259,157)(240,190)
%
%
\put(235,143){$\ast$}
\put(265,165){\scriptsize $\widetilde{\hEuk{2}{3}\diff{2}{\bD}}$}
\put(185,165){\scriptsize $\widetilde{\Euk{1}{2}w\sb{1}\Fk{1}}$}
\put(228,102){\scriptsize $\widetilde{\hEuk{1}{3}\diff{1}{\bD}}$}
%
%
\bezier{40}(300,180)(310,135)(320,90)
\put(301,175){\vector(-1,4){2}}
\put(319,95){\vector(1,-4){2}}
%
%
\put(330,-20){$\partial\Delnk{3}{3}$}
%
%
\put(260,10){\circle*{3}}
\put(258,0){\scriptsize $2$}
\put(298,75){\circle*{3}}
\put(291,72){\scriptsize $0$}
\put(335,140){\circle*{3}}
\put(334,145){\scriptsize $2$}
\put(335,10){\circle*{3}}
\put(334,0){\scriptsize $1$}
\put(410,10){\circle*{3}}
\put(409,0){\scriptsize $2$}
\put(372,75){\circle*{3}}
\put(376,72){\scriptsize $3$}
%
%
\put(260,11){\line(1,0){150}}
\put(298,76){\line(1,0){75}}
\linethickness{2pt}
\put(335,11){\line(1,0){75}}
\linethickness{1.5pt}
\bezier{400}(335,11)(353,43)(372,75)
\bezier{400}(410,11)(391,43)(372,75)
\thinlines
\bezier{400}(260,10)(298,75)(335,140)
\bezier{400}(410,10)(372,75)(335,140)
\bezier{400}(298,75)(317,43)(335,10)
\bezier{400}(372,75)(354,43)(335,10)
%
%
\put(291,30){$\ast$}
\put(333,48){$\ast$}
\put(325,98){\scriptsize $\widetilde{\hEuk{2}{3}\diff{2}{\bD}}$}
\put(352,27){\scriptsize $\widetilde{\Euk{2}{2}w\sb{2}\Fk{2}}$}
\end{picture}
\end{center}

\caption{Maps from $3$-simplices corresponding to \protect{\eqref{eqexttothree}}}
\label{fig4}
\end{figure}
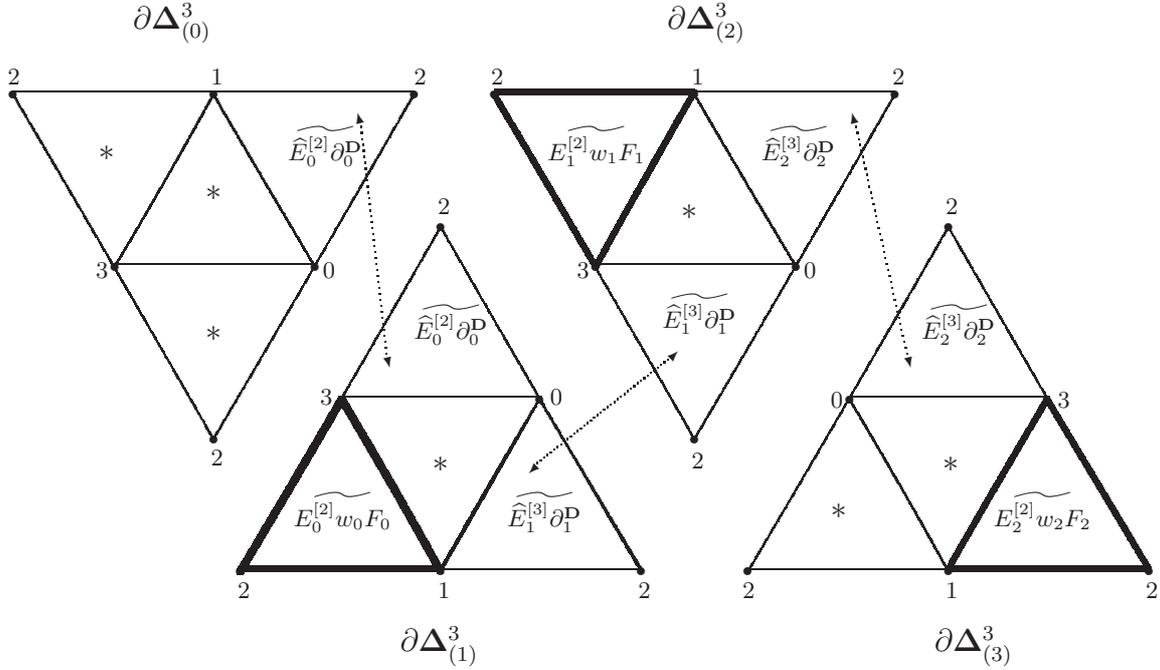

In \w[,]{\partial\Delnk{3}{0}} facet $3$ corresponds to the cone base \w[,]{\osW{2}{3}}
facets $1$ and $2$ correspond to the two suspension directions (and thus map to
$\ast$ in $\bX$). In general, facets \w{2,\dotsc, n-1} of
\w{\partial\Delnk{n}{k}} all map to $\ast$ as in \wref[,]{eqneweface} while the
$n$-th facet, corresponding to the cone base, agrees with the
$1$-facet \w{\partial\sb{1}\Delnk{n}{k+1}} of \w{\Delnk{n}{k+1}} if \w[.]{k<n}

The remaining maps \w{\tEuk{3}{1}} and \w{\tEuk{3}{2}} can be read off Figure \ref{fig4},
where only the boundaries of the four tetrahedra are shown. The edge \w{E\Pn{3}}
consists of the three $2$-simplices outlined in bold.
The full explanation of how each facet of every \w{\Delnk{n}{k}} maps to
\w{\bX\sp{\oW{3}}} is given in \cite[\S 5]{BSenH}, where each $n$-simplex is thought
of as a quotient of an $n$-cube.
\end{example}

\begin{prop}\label{pfoldpolymap}
Given an \wwb{n-1}map \w{\Eu{n-1}} for data \w{(\star\star)=\lra{\cW,\Eu{0}}}
as in \S \ref{dshho}, any choice of maps
\w{\hEuk{k}{n}:\oCsW{n-k-1}{n}\to\bX\sp{\Deln{k}}}
\wb{0\leq k\leq n} satisfying \wref{eqneweface}
determines a unique map \w{\tEuk{n}{}:\oW{n}\otimes\Pn{n}\to\bX} with the restriction
\w{\tEuk{n}{k}} to \w{\oW{n}\otimes\Delnk{n}{k}} adjoint as in \S \ref{sextnmap} to
\w{\hEuk{k}{n}\circ\tau\sp{n-k-1}:C\Sigma\sp{n-k-1}\oW{n}\to\bX\sp{\Deln{k}}}
for each \w[,]{1\leq k\leq n} and conversely.
\end{prop}

\begin{proof}
Given maps \w{\hEuk{k}{n}:\oCsW{n-k-1}{n}\to\bX\sp{\Deln{k}}}
\wb[,]{0\leq k\leq n} we obtain maps
\w{\tEuk{n}{k}:\oW{n}\otimes\Deln{n}\to\bX} as in \S \ref{sfoldp}.
The first condition in \wref{eqneweface} says that the restriction of
\w{\hEuk{k}{n}} to \w{\partial\sb{0}\Deln{k}} equals
\w[.]{\Euk{k-1}{n-1}\circ w\sb{k-1}\circ\Fk{k-1}} The second condition says
that its restriction to \w{\partial\sb{1}\Deln{k}} equals
\w[,]{\hEuk{k-1}{n-1}\circ\diff{n-k-1}{\bD}=\hEuk{k-1}{n-1}\circ\oi{k}\circ\oq{k-1}}
 corresponding to the restriction to the base of the cone and thus to the restriction
of \w{\tEuk{n}{k}} to \w[,]{\partial\sb{n}\Deln{n}} by the convention of
\S \ref{sfoldp}. Since the face maps \w{d\sb{i}} on \w{\oCsW{n-k-1}{n}} vanish for
\w[,]{i\geq 2} and \w{\hEuk{k-1}{n}} also vanishes at the other end of the cone
direction and at both ends of the suspension directions, we obtain a map
\w{\tEuk{n}{}:\oW{n}\otimes\Pn{n}\to\bX} as required.

Conversely, given such a map \w[,]{\tEuk{n}{}} its restrictions to the $n$-simplices
\w{\Delnk{n}{0},\dotsc,\Delnk{n}{n}} define the maps \w[,]{\tEuk{n}{k}} and thus
maps \w[.]{\ppp\hEuk{k}{n}:C\Sigma\sp{n-k-1}\oW{n}\to\bX\sp{\Deln{k}}}
We now show by induction on \w{k\geq 0} that these extend to maps
\w{\hEuk{k}{n}:\oCsW{n-k-1}{n}\to\bX\sp{\Deln{k}}} with
\begin{myeq}\label{eqnewold}
\ppp\hEuk{k}{n}~=~\hEuk{k}{n}\circ\tau\sp{n-k-1}
\end{myeq}
\noindent satisfying
\begin{myeq}\label{eqnewfaces}
\begin{cases}
~(\eta\sp{0})\sp{\ast}\circ\ppp\hEuk{k}{n} ~=&~
\Euk{k-1}{n-1}\circ w\sb{k-1}\circ\Fk{k-1}~,\\
~(\eta\sp{1})\sp{\ast}\circ\ppp\hEuk{k}{n}~=&~\ppp\hEuk{k-1}{n-1}\circ\diff{n-k-1}{\bD}~,\\
~(\eta\sp{i})\sp{\ast}\circ\ppp\hEuk{k}{n}~=&~0\hsm\text{for}\hsm i\geq 2~.
\end{cases}
\end{myeq}
\noindent To start the induction for \w[,]{k=0} where \wref{eqnewfaces} is vacuous,
use the fact that \w{\tau\sp{n-1}} is an acyclic cofibration and
\w{\bX\sp{\Deln{0}}=\bX} is fibrant to extend \w{\ppp\hEuk{0}{n}} to \w[.]{\hEuk{0}{n}}

In the induction step, the given map \w{\hEuk{k-1}{n}} induces a map
\w{L:\oCsW{n-k-1}{n}\to\bX\sp{\partial\Deln{k}}} fitting into
the following solid commutative diagram:
\myxdiag[\label{eqliftingh}]{
C\Sigma\sp{n-k-1}\oW{n} \ar@{^{(}->}[d]\sp{\simeq}\sb{\tau\sp{n-k-1}}
\ar[rrrr]\sp{\ppp\hEuk{k}{n}} &&&&
\bX\sp{\Deln{k}} \ar@{->>}[d]\sp{\inc\sp{\ast}} \\
\oCsW{n-k-1}{n} \ar@{.>}[rrrru]\sp{\hEuk{k}{n}} \ar[d]\sb{\Fk{k-1}}
\ar[rrd]\sb{\diff{n-k-2}{\bD}}  \ar[rrrd]\sp{0} \ar[rrrr]\sb{L} &&&&
\bX\sp{\partial\Deln{k}} \ar@{->>}[dd]\sp{(\eta\sp{0})\sp{\ast}}
\ar@{->>}[ldd]\sp{(\eta\sp{1})\sp{\ast}} \ar@{->>}[ld]\sb{(\eta\sp{i})\sp{\ast}}\\
\cM{k-1}\W{n-1} \ar[d]\sp{w\sb{k-1}} && \oCsW{n-k}{n} \ar[rd]\sb{\hEuk{k-1}{n}} &
\bX\sp{\Delnk{k-1}{i}} \\
\Wn{k-1}{n-1} \ar@/_{1.5pc}/[rrrr]\sb{\Euk{k-1}{n-1}}
&&& \bX\sp{\Delnk{k-1}{1}}  & \bX\sp{\Delnk{k-1}{0}} \\
}
\noindent where the bottom portion of the diagram fits together to define $L$
by \wref[,]{eqnewfaces} and the whole diagram commutes
by \wref{eqneweface} and \wref[.]{eqnewold}

Since the cofibration \w{\inc:\partial\Deln{k}\hra\Deln{k}} induces a
fibration \w{\inc\sp{\ast}} by \cite[II, \S 2, SM7]{QuiH}, and
\w{\tau\sp{n-k-1}} of \wref{eqmodconesusp} is an acyclic cofibration,
we have the lifting \w{\hEuk{k}{n}}
by the left lifting property. The fact that \wref{eqliftingh} commutes implies that
\wref{eqnewfaces} and \wref{eqnewold} hold for $k$ as well,
completing the induction step.
\end{proof}

\begin{assume}\label{spup}
Assume now that the pointed simplicial model category $\C$
has an underlying unpointed simplicial model category $\hC$ (see \cite[\S 1.1.8]{HovM}).
This is the case when \w{\C=\Sa} or \w[,]{\Topa} for example. Note that
the two simplicial tensorings are different: thus in \w{\hC=\Top}
we have the product \w{\bA\times K} for \w{\bA\in\hC} and
\w{K\in\cS} as the simplicial tensor \w[,]{\bA\otimes K} while for
\w{\C=\Topa} we must set \w[,]{\bA\otimes K=\bA\times
K/(\ast\times K)} where $\ast$ is the given basepoint in $\bA$
(because \w{\bA\otimes K\in\C} must itself be pointed, while
\w{\bA\times K} has no chosen basepoint, since $K$ is in $\cS$,
not \w[).]{\Sa} See \cite[\S 9.1.14]{PHirM}.

However, when $K$ has a basepoint $k$, we write \w{\bA\wedge K} for
\w[.]{\bA\otimes K/(\bA\otimes\{k\})}
We also write \w{K\times L} for the product in $\cS$.
\end{assume}

In this case we have an explicit description of the following classical
fact (see \cite{BJiblSL}):

\begin{lemma}\label{rhalfsmash}
  If $\bA$ is cofibrant in $\C$ as in \S \ref{spup}, and \w{B\in\Sa=\Seta\sp{\Dop}}
  is connected, with basepoint $b$, there is a canonical weak equivalence
\begin{myeq}\label{eqhalfsmash}
  \varphi:\Sigma\bA\otimes B~\xra{\simeq}~
  \Sigma\bA/(\Sigma\bA\otimes\{b\})\otimes B~\amalg~\Sigma\bA\otimes\{b\}~\simeq~
  (\Sigma\bA\wedge B)\vee\Sigma\bA
\end{myeq}
\noindent where $\varphi$ onto the first summand is the natural projection, and
the reverse weak equivalence on the second summand is induced by \w[.]{\{b\}\hra B}
\end{lemma}

\begin{proof}
We have the following diagram of pushout squares with vertical cofibrations:
\mytdiag[\label{eqhspo}]{
  \bA\otimes(B\times\{0,1\}) \ar@{^{(}->}[d]\sp{\inc\sb{1}}
  \ar@{^{(}->}[r]^-{\inc\sb{2}} &
  \bA\otimes(B\times I\sb{2}) \ar@{^{(}->}[d] \ar[r]^-{\proj\sb{\ast}} &
  \bA\otimes I\sb{2} \ar@{^{(}->}[d] \ar[r]\sp{\simeq} &
  \bA\otimes \ast \ar@{^{(}->}[d] \\
  \bA\otimes(B\times I\sb{1}) \ar@{^{(}->}[r] & \bA\otimes (B\times S\sp{1}) \ar[r] &
  \bA\otimes(\widetilde{B\ltimes S\sp{1}}) \ar[r]\sp{\simeq} &
  \bA\otimes(B\ltimes S\sp{1})
}
\noindent where \w{I\sb{1}} and \w{I\sb{2}} are two $1$-simplices and \w{\inc\sb{i}}
\wb{i=1,2} are the inclusions of the two endpoints \w{\{0,1\}} into each,
so that the resulting pushout is a two-cell model of the circle \w[,]{S\sp{1}}
and \w{\proj:B\times I\sb{1}\to I\sb{1}} is the projection. Note that
\w[.]{\bA\otimes (B\times S\sp{1})=(\bA\otimes B)\otimes S\sp{1}}

Here \w{\widetilde{B\ltimes S\sp{1}}} (the case where \w[)]{\bA=\ast} is a model of
the half-smash in $\cS$ consisting of the unreduced suspension \w{SB} with
an arc connecting the two suspension points. Under the quotient map
\w{SB\to\Sigma B} to the reduced suspension (which is a weak equivalence,
by \cite{GrayH}) this becomes a wedge \w[.]{\Sigma B\vee S\sp{1}}

Note that  \w[.]{\bA\otimes(B\ltimes S\sp{1})\cong(\bA\wedge S\sp{1})\otimes B}
Thus \w{\bA\otimes(\widetilde{B\ltimes S\sp{1}})} is a model for
\w[,]{\Sigma\bA\otimes B} which is weakly equivalent to
\w[.]{\Sigma\bA\wedge B\vee \Sigma\bA}
\end{proof}

\begin{defn}\label{dvalstrand}
We associate to any \wwb{n-1}map \w{\Eu{n-1}} as in \S \ref{dshho} a map
\w{\euk{n-1}:\oW{n}\otimes\partial\Pn{n}\to\bX} which sends
\w{\oW{n}\otimes \partial\sb{1}\Delnk{n}{k}} to $\bX$ by
\w{\Euk{k-1}{n-1}\circ\Fk{k-1}} for each \w[,]{1\leq k\leq n} and all other
\wwb{n-1}simplices of \w{\partial\Pn{n}} to the basepoint.
Here we use the convention of the beginning of the proof of Proposition \ref{pobst},
so \w[.]{\Fk{n-1}=\dz{n}}

Since at most two additional \wwb{n-1}facets of \w{\Delnk{n}{k}} are
identified with \wwb{n-1}facets of \w[,]{\Delnk{n}{k\pm 1}} we may think of
\w{C\Sigma\sp{n-k}\oW{n}} as a quotient of \w[,]{\oW{n}\otimes\Delnk{n}{k}}
so the map induced by \w{\Euk{k-1}{n}\circ\Fk{k-1}} is well-defined.
Moreover, these maps are compatible for adjacent values of $k$
(see Figure \ref{fig4} and \wref[).]{eqexttothree}

Assuming $\C$ satisfies the assumptions of \S \ref{spup},
by Lemma \ref{lfoldpoly} we can think of \w{\euk{n-1}} as a map
\w[,]{\oW{n}\otimes\bS{n-1}\to\bX} and because $\C$ is pointed, the source is
a half-smash \w{\oW{n}\ltimes\bS{n-1}:=(\oW{n}\times\bS{n-1})/(\ast\times\bS{n-1})}
in the corresponding unpointed simplicial model category.

We see from the previous description that if \w{v\sb{0}} is the $0$-vertex of
\w[,]{\Delnk{n}{0}} then \w{\Euk{}{n-1}} maps \w{\oW{n}\otimes\{v\sb{0}\}} to
$\ast$. Therefore, if we choose \w{v\sb{0}} as the basepoint of
\w[,]{\bS{n-1}\cong\partial\Pn{n}} we see that \w{\huk{n-1}} is trivial when
restricted to the second summand in \wref[,]{eqhalfsmash} and is thus uniquely
determined up to homotopy by the induced map
\w[.]{g\sb{n-1}:\Sigma\sp{n-1}\oW{n}\to\bX} We define
the \emph{value} of the \wwb{n-1}map \w{\Euk{}{n-1}} to be the class
\begin{myeq}\label{eqvalue}
\Val{\Euk{}{n-1}}~:=~[g\sb{n-1}]~\in~[\Sigma\sp{n-1}\oW{n},\,\bX]~\cong~
\Lambda\lin{\Sigma\sp{n-1}\oW{n}}~.
\end{myeq}
\noindent for \w[.]{\Lambda:=\piA\bX}
\end{defn}

\begin{prop}\label{pvanish}
Given data \w{(\star\star)=\lra{\cW,\Eu{0}}}  as in \S \ref{dshho},
the value for a corresponding \wwb{n-1}map \w{\Eu{n-1}:\W{n-1}\to\bX\sp{\Du}}
is zero if and only if it extends to an $n$-map.
\end{prop}

\begin{proof}
Evidently, \w{g\sb{n-1}} is nullhomotopic if and only if the original map
\w{\euk{n-1}} extends to \w[.]{\oW{n}\otimes\Pn{n}}

If \w{\Euk{}{n-1}} extends to an $n$-map \w[,]{\Euk{}{n}:\W{n}\to\bX\sp{\Du}}
as in the proof of Lemma \ref{lextendfmcone}, the acyclic fibration
\w{p:\W{n}\to{\W{n-1}}'[\Fn{n}]} in \wref{eqsquarb} has a section
\w[,]{s:{\W{n-1}}'[\Fn{n}]\to\W{n}}  and precomposition with the acyclic
  cofibration \w{p':\hWd{n-1}[\Fn{n}]\hra{\W{n-1}}'[\Fn{n}]} yields the restriction
  \w{\hEuk{}{n}:=\Euk{}{n}\circ s\circ p':\hWd{n-1}[\Fn{n}]\to\bX\sp{\Du}}
  (an $n$-map).

  Since $s$ and \w{p'} are weak equivalences, \w{\W{n-1}} is Reedy cofibrant, and
  \w{\bX\sp{\Du}} is Reedy fibrant, \w{\hEuk{}{n}\circ\wj:\W{n-1}\to\bX\sp{\Du}} is
  homotopic to \w[.]{\Euk{}{n-1}=\Euk{}{n}\circ\oi{n-1}}  By the description in
  Definition \ref{dvalstrand}, we see that the map
  \w{\euk{n-1}:\oW{n}\times\bS{n-1}\to\bX} induced by \w{\Euk{}{n-1}} is
  thus homotopic to the corresponding map \w{\hek{n-1}:\oW{n}\times\bS{n-1}\to\bX}
  induced by \w[.]{\hEuk{}{n}\circ\wj} The same is therefore true for the
  restrictions of the induced maps,
  \w[.]{g\sb{n-1}\sim\widetilde{g}\sb{n-1}:\Sigma\sp{n-1}\oW{n}\to\bX}
  Since \w{\widetilde{g}\sb{n-1}} is nullhomotopic (because \w{\hEuk{}{n}\circ\wj}
  extends to \w[),]{\hEuk{}{n}} so is \w[,]{g\sb{n-1}} and thus
  \w[.]{\Val{\Eu{n-1}}=0}

Conversely, if \w{g\sb{n-1}} is nullhomotopic, then \w{\euk{n-1}} extends
to \w[,]{\tEuk{n}{}:\oW{n}\times\Pn{n}\to\bX} and we let \w{\tEuk{n}{k}}
denote the restriction of \w{\tEuk{n}{}} to \w{\Delnk{n}{k}} for each
\w[.]{0\leq k\leq n} By Definition \ref{dvalstrand} for \w{\huk{n-1}} we see that
the maps \w{\tEuk{n}{k}} satisfy \wref[,]{eqneweface} so together with the original
\wwb{n-1}map \w{\Euk{}{n-1}} they define a map of $n$-truncated
restricted simplicial spaces \w[,]{\widetilde{E}:\Cone(\tFn{n})\to\bX\sp{\Du}}
and thus \w{\hEu{n}:\hWd{n-1}[\Fn{n}]\to\Ud} (see \S \ref {ridea}).
Lemma \ref{lextendfmcone} then yields the required extension.
\end{proof}

%
%
\sect{Comparing obstructions}
\label{ccompobst}
The value we have assigned to an \wwb{n-1}map serves as
the obstruction to extending it to an $n$-map, but only
with respect to a fixed sequential realization $\cW$. We now wish to
investigate to what extent the vanishing or otherwise depends on this choice of
$\cW$. For this purpose we require the following

\begin{defn}\label{dcorrespstr}
Let \w{(\star)=\lra{\bY,\bX,\vth}} be basic initial data as in \S \ref{dshho},
with \w{\cd{\bX}\stackrel{\simeq}\to\Ud} a fibrant replacement in
\w[.]{\C\sp{\Dop}} Assume given in addition \w{(\cW,\Eu{0})} and \w{(\ccWp,\Eup{0})}
as two choices of specific initial data \w[,]{(\star\star)} equipped with
extensions to \wwb{n-1}maps \w{\Eu{n-1}} and \w[,]{\Eup{n-1}} respectively.

If \w{\Phi:\cW\to\,\ccWp} is an $n$-stage comparison map, as in \wref[,]{eqncorresp}
write \w{\Eup{n-1}=\rs(\Eu{n-1})} if \w{\Eup{n-1}=\Eu{n-1}\circ\rn{n-1}}
and \w{\Eu{n-1}=\es(\Eup{n-1})} if \w[.]{\Eu{n-1}=\Eup{n-1}\circ\en{n-1}}

By \wref[,]{eqmapmodconesusp} \wref[,]{eqcompf} and \wref[,]{eqexttothree} we see that
\begin{myeq}\label{eqcorrvals}
\Val{\rs(\Eu{n})}=(\oor{n-1}{n})\sp{\ast}(\Val{\Eu{n}})\hsm \text{and}\hsm
\Val{\es(\Eup{n})}=(\oon{n-1}{n})\sp{\ast}(\Val{\Eup{n}})~,
\end{myeq}
\noindent in the notation of \wref[.]{eqmapmodconesusp}
\end{defn}

As a result we have:

\begin{lemma}\label{lindvanish}
Assume given an $n$-stage comparison map \w[,]{\Phi:\cW\to\cWp} an \wwb{n-1}map
  \w{\Eu{n-1}} for \w[,]{(\cW,\bX,\vth)} and an \wwb{n-1}map \w{\Eup{n-1}}
  for \w[.]{(\ccWp,\bX,\vth)} Then:
\begin{enumerate}[(a)]
\item \w{\Val{\Eu{n-1}}=0} if and only if \w[.]{\Val{\rs(\Eu{n-1})}=0}
\item If \w{\Val{\Eup{n-1}}=0} then \w[,]{\Val{\es(\Eup{n-1})}=0}
  but not necessarily conversely.
\end{enumerate}
\end{lemma}

This explains the need for the following:

\begin{defn}\label{dequivstr}
Assume given basic data \w[,]{(\star)=\lra{\bY,\bX,\vth}} with two choices
of specific data \w{(\star\star)} of the form
\w{\lra{\cW,\Eu{0}}} and \w[.]{\lra{\ccWp,\Eup{0}}}
If \w{\Eu{n}} and \w{\Eup{n}} are $n$-maps associated respectively to these choices,
we  write \w{\Eup{n}\sim\Eu{n}} if there is an $n$-stage comparison map
\w{\Phi:\cW\to\,\ccWp} such that \w[.]{\es(\Eup{n})=\Eu{n}} The equivalence relation
generated by \w{``\sim"} is called the {\it weak equivalence} relation on $n$-maps, and
equivalence classes are denoted by \w[.]{[\Eu{n}]}
\end{defn}

\begin{mysubsection}{The universal homotopy operations}
\label{snouho}
Given basic data \w{(\star)=\lra{\bY,\bX,\vth}} as above, we think of each
sequential realization $\cW$ for $\bY$ as a template for an infinite sequence
\w{\llrra{\star}=(\llrr{\star}{n})\sb{n=1}\sp{\infty}}
of higher homotopy operations, with \w{\Cone(\tFn{n})}
of its \wwb{n-1}st stage \w{\W{n-1}} serving as the template for the
\emph{universal $n$-th order homotopy operation} \w[,]{\llrr{\star}{n}}
for each \w[,]{n\geq 2} as in \wref[.]{eqcofseqr}

Formally, \w{\llrr{\star}{n}} is the function which assigns to
any choice of specific data \w{(\star\star)=(\cW,\Eu{0})} and \wwb{n-1}map
\w{\Eu{n-1}} for \w{(\star\star)} the value \w{\Val{\Eu{n-1}}} in
\w[,]{\Lambda\lin{\Sigma\sp{n-1}\oV{n}}} as in \wref[.]{eqvalue}
We write \w{\Vals{\Eu{n-1}}} for the set of all values at
all such \wwb{n-1}maps \w[.]{\Eup{n-1}\in[\Eu{n-1}]}

We say that \w{\llrra{\star}} \emph{vanishes coherently for}
\w{(\star\star)=(\cW,\Eu{0})} if for each \w[,]{n\geq 2} we are given an
\wwb{n-1}map \w{\Eu{n-1}} for \w{(\star\star)} such that \w{\Val{\Eu{n-1}}=0}
(and thus \w[),]{0\in\Vals{\Eu{n-1}}} so that
\w{\Eu{n-1}} extends by Proposition \ref{pvanish} to an $n$-map
\w[.]{\Eu{n}} Taken together, we thus obtain a strand \w{\Eu{\infty}}
for \w[.]{(\star\star)}

Finally, we say that \w{\llrr{\star}{n}} \emph{vanishes for $\bX$} if there is
\emph{some} \w{(\star\star)=(\cW,\Eu{0})} with an \wwb{n-1}map \w{\Eu{n-1}}
such that \w[:]{\Val{\Eu{n-1}}=0} that is, if \w[.]{0 \in \Vals{\Eu{n-1}}}
\end{mysubsection}

The following consequence of Theorem \ref{tzigzag} shows that
we can in fact disregard the notion of \w{\Vals{-}} defined for
equivalence classes of \wwb{n-1}maps \w[,]{[\Eu{n-1}]} and concentrate
instead on any one sequential realization $\cW$ of $\bY$ to determine
vanishing of \w[:]{\llrr{\star}{n}}

\begin{keylemma}\label{lkey}
Given \w{(\star)=\lra{\bY,\bX,\vth}} as in \S \ref{dshho},
\w{\llrr{\star}{n}} vanishes for $\bX$ if and only if for
\emph{every} \w{(\star\star)=(\cW,\Eu{0})} (in fact, for \emph{any} $n$-stage sequential
realization $\cW$ for $\bY$), there is an \wwb{n-1}map \w{\Eu{n-1}} with
\w[.]{\Val{\Eu{n-1}}=0}
\end{keylemma}

\begin{proof}
By definition, \w{\llrr{\vth}{n}} vanishes for $\bX$ if there is \emph{some}
$n$-stage sequential realization \w{\cWp} of $\bY$
and an \wwb{n-1}map \w{\Eup{n-1}} for \w{(\star\star)=(\cWp,\Eup{0})} such that
\w[.]{\Val{\Eup{n-1}}=0}  By Theorem \ref{tzigzag} (with \w[)]{\bY=\bYp}
we know that there is a finite
zigzag of cospans of comparison maps connecting \w{\cWp} to $\cW$, say
$$
\Phip{1}:\cuW{0}=\cWp\to\cuW{1},\hsp \Phip{2}:\cuW{2}\to\cuW{1},\hsp
\Phip{3}:\cuW{2}\to\cuW{3}~,
$$
\noindent and so on until \w[.]{\Phip{N}:\cuW{N-1}\to\cuW{N}=\cW} If
\w{\Phip{1}=\lra{\en{k},\,\rn{k},\dotsc\,}\sb{k=0}\sp{n}} as in \wref[,]{eqncorresp}
we set \w{\Eul{n-1}{1}:=\rs(\Eup{n-1})} (an \wwb{n-1}map for \w[),]{\cuW{1}}
and see from \wref{eqcorrvals} that \w[.]{\Val{\Eul{n-1}{1}}=0}
Similarly, if \w{\Phip{2}=\lra{\hen{k},\,\hrn{k},\dotsc\,}\sb{k=0}\sp{n}}
we set \w{\Eul{n-1}{2}:=\hes(\Eul{n-1}{1})} (an \wwb{n-1}map for
\w[),]{\cuW{2}} and again see from \wref{eqcorrvals} that
\w[.]{\Val{\Eul{n-1}{2}}=0} Continuing in this way we finally obtain
an \wwb{n-1}map \w{\Eu{n-1}=\Eul{n-1}{N}} for \w{\cuW{N}=\cW}
with \w[,]{\Val{\Eu{n-1}}=0} as required.
\end{proof}

%
%
\sect{Higher homotopy invariants for objects}
\label{chhio}

In this section we assume given a free simplicial \PAal resolution \w{\Vd} of
a realizable \PAal \w[,]{\Lambda=\piA\bY} for some \w{\bA\in\C} as
in \S \ref{cback}.A and any \w[.]{\bY\in\C}  Because each \w{V\sb{n}} is a free
\PAal[,] \w{\Vd\to\Lambda} can be realized by an augmented simplicial object
\w{\unW\to\bY} in the homotopy category \w[,]{\ho\C} unique up to weak
equivalence. Theorem \ref{tres} showed that this can always be rectified
to \w{\Wd\to\bY} in $\C$ through a sequential realization $\cW$ of \w[.]{\Vd}

The main question we address in this section is the following:  if we are given some
other (cofibrant) realization $\bZ$ of $\Lambda$ \wh or equivalently, an isomorphism
\w{\vth:\Lambda\to\piA\bZ} \wwh can we similarly rectify \w[?]{\unW\to\bZ}
More precisely, can we augment the given rectification \w{\Wd} of \w{\unW} to $\bZ$
instead, at least up to homotopy?

\begin{mysubsection}{The $0$-augmentation}
\label{szeroaug}
From the proof of Theorem \ref{tres} we expect a $0$-augmen\-ta\-tion,
the analog of a $0$-map in this context, to be completely determined by a choice of a
realization \w{e\sb{0}:\oW{0}\to\bZ} of \w[.]{\vth\circ\vare:V\sb{0}\to\piA\bZ}
Indeed, such a map always exists, is unique up to homotopy, and defines a
map \w[.]{\bv\sb{0}:\cd{\oW{0}}\to\cd{\bZ}}
Composing it with \w{p\sp{\ast}:\cd{\bZ}\to\bZ\sp{\Du}=\Ud} (see \S \ref{raug})
yields \w[.]{\widehat{\bv}\bp{0}:\cd{\oW{0}}\to\Ud}

As in \wref{eqsquara} we obtain \w[,]{{\bv'}\bp{0}:{\W{-1}}'[\Fn{0}] \to \Ud} and
the composite
$$
\W{0}=\W{-1}[\Fn{0}]~\xra{p}~{\W{-1}}'[\Fn{0}]~\xra{{\bv'}\bp{0}}~\Ud
$$
defines the $0$-augmentation \w[.]{\bve{0}:\W{0}\to\Ud} Thus, even though
  \w{\bve{0}} is formally part of the specific initial data \w[,]{(\star\star)}
  we shall omit mention of it henceforth.
\end{mysubsection}

\begin{defn}\label{dshhoo}
In this version of \S \ref{dshho}, the basic initial data consists
of \w[,]{(\star)=(\bY,\bZ,\vth)}
with $\bZ$ cofibrant and \w{\vth:\Lambda\to\piA\bZ} an \emph{isomorphism} of \PAal[s]
while the specific initial data \w{(\star\star)} consists of a sequential
realization $\cW$ of a CW-resolution \w{\vare:\Vd\to\Lambda:=\piA\bY} for $\bY$.
As before, we let \w{\Ud=\bZ\sp{\Du}} be our Reedy fibrant replacement for \w[.]{\cd{\bZ}}

The corresponding $n$-maps will then be called \emph{$n$-augmentations} \wh
that is, maps \w{\bve{n}:\W{n}\to\Ud} realizing
\w{\vth\circ\vare:\Vd\to\piA\bZ\cong\piA\|\Ud\|} though simplicial
dimension $n$.
\end{defn}

\begin{defn}\label{dextnaug}
As in \S \ref{dvalstrand}, given an \wwb{n-1}augmentation
\w{\bve{n-1}:\W{n-1}\to\Ud} for \w[,]{(\star\star)} we define
its \emph{value} \w{\Val{\bve{n-1}}} in \w{\Lambda\lin{\Sigma\sp{n-1}\oW{n}}} using
\wref[,]{eqvalue} and deduce from Proposition \ref{pvanish} that
this is zero if and only if \w{\bve{n-1}} extends to an
$n$-augmentation \w[.]{\bve{n}:\W{n}\to\Ud}

We denote by \w{\llrra{\bY}=(\llrr{\bY}{n})\sb{n=1}\sp{\infty}}
the universal homotopy operation
\w{\llrra{\star}=(\llrr{\star}{n})\sb{n=1}\sp{\infty}} as in \S \ref{snouho}
associated to \w[.]{(\star):=(\bY,\bZ,\vth:\piA\bY\xra{\cong}\piA\bZ)}
\end{defn}

\begin{example}\label{egcohervan}
For \w{(\star)=(\bY,\bZ,\vth)} with \w{\vth=f\sb{\#}} induced by an
$\bA$-equivalence \w[,]{f:\bY\to\bZ}  we see that  \w{\llrra{\bY}} vanishes
  coherently for $\bZ$ at any sequential realization \w{(\star\star)=\lra{\cW}} of $\bY$,
  since by assumption we have an actual augmentation \w[,]{\bv:\Wd\to\cd{\bY}}
  inducing a homotopy augmentation \w{p\sp{\ast}\circ\cd{f}\circ\bv:\Wd\to\Ud}
  (in the notation of \S \ref{raug}). Restricting this to each \w{\W{n}}
yields \w[.]{\bve{n}}
This shows that \w{\Val{\bve{n}}=0} for each \w[,]{n\geq 1} by Proposition \ref{pvanish}.
\end{example}

We may then formulate our next main result as follows:

\begin{thm}\label{tvanish}
For \w{\bA\in\C} as in \S \ref{setmc}, let \w{\vth:\piA\bY\to\piA\bZ} be an
isomorphism of \PAal[s.] Then the following are equivalent:
\begin{enumerate}[(i)]
\item The system of higher homotopy operations \w{\llrra{\bY}} vanishes coherently
for \emph{some} sequential realization $\cW$ for $\bY$;
\item The system \w{\llrra{\bY}} vanishes coherently for \emph{every}
sequential realization for $\bY$;
\item $\vth$ is realizable by a zigzag of $\bA$-equivalences between $\bY$ and
  $\bZ$ (that is, $\bY$ and $\bZ$ are $\bA$-equivalent).
\end{enumerate}
\end{thm}

\begin{proof}
The equivalence of the first two conditions follows from Key Lemma \ref{lkey}
and Proposition \ref{pvanish}.
As noted in \S \ref{raug}, the equivalence of the first and third conditions then
reduces to the existence of suitable homotopy augmentations:

If the system of higher operations vanishes coherently for some
sequential realization $\cW$ of $\bY$, there is a strand
\w{\bve{\infty}} for \w[,]{(\cW,\bZ,\vth)} and thus augmentations
\w{\bve{n}:\W{n}\to\Ud} for  all \w[.]{n\geq 0}
These fit together to define a homotopy augmentation \w{\bv:\Wd\to\Ud}
for \w[,]{\Wd:=\hocolim\W{n}} which induces an isomorphism
\w[.]{\pi\sb{0}\piA\Wd\to\piA\|\Ud\|}
By assumption \S \ref{setmc}(2), $\bY$ is $\bA$-equivalent to the realization
\w[,]{\|\Wd\|}  and thus the map \w{\bv\sb{\ast}:\|\Wd\|\to\|\Ud\|\simeq\bZ}
induced by the augmentation $\bv$ realizes $\vth$, so $\bY$ and $\bZ$ are related
by a cospan of $\bA$-equivalences.

Conversely, if $\bY$ and $\bZ$ are $\bA$-equivalent, they are related by
a span (or cospan) of $\bA$-equivalences, so it suffices to consider the
following two cases:

\begin{enumerate}[(a)]
\item Given an $\bA$-equivalence \w{f:\bY\to\bZ} and a sequential realization
$\cW$ for $\bY$, we may assume $f$ lifts to a map \w{\wf:\Td\xra{\simeq}\Ud}
  between the fibrant replacements for \w{\cd{\bY}} and \w[,]{\cd{\bZ}} respectively
  (using the functorial factorizations in $\C$, and thus in \w[,]{\C\sp{\Dop}}
  assumed in \S \ref{sass}). Postcomposing the $n$-augmentations
  \w{\bve{n}:\W{n}\to\Td} with \w[,]{\wf} we obtain $n$-augmentations
  \w{\wf\circ\bve{n}:\W{n}\to\Ud} still realizing \w[,]{\Vd\to\Lambda} since
  \w{f\sb{\#}:\Lambda\to\piA\bZ} is an isomorphism.
Thus, \w{\llrra{\bY}} vanishes coherently for \w{(\cW,\bZ,\vth)} by
Proposition \ref{pvanish}.
\item On the other hand, given an $\bA$-equivalence \w{g:\bZ\to\bY}
and a sequential realization \w{\cWp} for $\bZ$,
by postcomposing the $n$-augmentations \w{\bve{n}:\W{n}\to\Ud}
with \w{\wg:\Ud\xra{\simeq}\Td} as in (a), we obtain $n$-augmentations
\w{\wg\circ\bve{n}:\W{n}\to\Td}
realizing \w[,]{\Vd\to\Lambda} and thus making \w{\cWp} itself, with the
corresponding actual augmentations \w[,]{\alpha\sb{\bY}\circ\wg\circ\bve{n}} into a
sequential realization \w{\cWpp} for $\bY$.
By \S \ref{egcohervan} we thus have a strand \w{\bvp{\infty}} for
\w[,]{(\cWpp,\bY,\Id\sb{\Lambda})} and of course the actual augmentations
\w{\alpha\sb{\bZ}\circ\bve{n}:\W{n}\to\bZ} themselves form a corresponding strand
\w{\bve{\infty}} for \w[,]{\cWpp} showing that
the system of higher operations vanishes coherently for \w[.]{(\cWpp,\bZ,\vth)}
\end{enumerate}
\noindent This completes the proof.
\end{proof}

\begin{corollary}\label{csecoper}
If $\bY$ and \w{\bY'} are weakly equivalent $\bA$-cellular spaces, any sequential
realization $\cW$ for $\bY$ is also a sequential realization for \w[.]{\bY'}
In particular, \w{\Wd} has an augmentation to its realization \w{\|\Wd\|}
inducing an $\bA$-equivalence.
\end{corollary}

\begin{proof}
Since $\bY$ and \w{\bY'} are fibrant and cofibrant in the $\bA$-model category
structure of \S \ref{rrightbous}, there is a homotopy equivalence
\w{h:\bY\to\bY'} which we may compose with the augmentation \w{\Wd\to\bY}
to obtain a strict augmentation to \w[.]{\bY'}
\end{proof}

\begin{mysubsection}{The moduli space of weak homotopy types}
\label{smswht}
Theorem \ref{tvanish} provides a more geometric approach to the ``moduli
space'' \w{\M\sb{\A}} of all $\bA$-homotopy types in the model category $\C$,
described in \cite{BDGoeR} for \w{\bA=\bS{1}} in \w{\C=\Topa} (see also
\cite{PstrM}):

The primary decomposition of \w{\M\sb{\A}} is, of course, into connected
components corresponding to non-isomorphic realizable \PAal[s] $\Lambda$.
For a given $\Lambda$, we first choose some $\bY$ with
\w{\piA\bY\cong\Lambda} as our base point, with a sequential simplicial
realization $\cW$ for $\bY$.

We then filter the other realizations $\bZ$ of $\Lambda$ (not weakly equivalent
to $\bY$) by the greatest $n$ for which some $n$-augmentation
\w{\bve{n}\sb{(\cW,\bY,\Id\sb{\Lambda})}} exists for $\bZ$. Up to a shift in indexing,
this corresponds to the cohomological filtration of the component of $\bY$ in
\w{\M\sb{\A}}  given in \cite{BDGoeR} (see also \cite{BJTurnR}).
\end{mysubsection}

\begin{mysubsection}{An example in rational homotopy theory}
\label{srex}
As noted in \S \ref{egfree}, we may apply our theory to Quillen's model \w{\dgL} of
reduced differential graded Lie algebras over $\QQ$ (DGLs) for the homotopy theory
of simply-connected rational spaces (see \cite[\S 1]{QuiR}), with $\bA$ the standard
model for \w[,]{\bS{2}\sb{\QQ}} as in the following example:

Let \w{\As:=\LL{a,b,c,x,y,z,w,e}} denote the free DGL with generators in degrees
\w{|a|=|b|=|c|=m} (for $m$ even), \w[,]{|x|=|y|=|z|=2m+1} \w[,]{|e|=3m+1}
and \w[,]{|w|=3m+2} and with differentials \w[,]{d(x)=[b,c]} \w[,]{d(y)=[c,a]}
\w[,]{d(z)=[a,b]} and \w{d(w)=f} for the Lie-Massey product
\w{f:=[a,x]+[b,y]+[c,z]} of degree \w[.]{3m+1} All other differentials are zero.

Similarly, let \w{\Bs:=\LL{a,b,c,x,y,z}} with
\w[,]{|a|=|b|=|c|=m} \w[,]{|x|=|y|=|z|=2m+1} and non-zero differentials
\w[,]{d(x)=[b,c]} \w[,]{d(y)=[c,a]} and \w[.]{d(z)=[a,b]}
We truncate \w{\As} and \w{\Bs} in degree \w[,]{4m} so all Lie brackets vanish
in \w[.]{\His\As\cong\His\Bs}

Using the obvious free simplicial \PAal resolution (as graded Lie algebras), we
obtain the following augmented simplicial DGL \w{\Wd\to\As} in simplicial
dimensions \www{\leq 2} (with degrees indicated by subscripts):

\begin{enumerate}[(a)]
\item In dimension $0$ we have
\w[,]{\bW\sb{0}=\oW{0}\amalg C\oW{1}\amalg C\Sigma\oW{2}} where
\begin{enumerate}[(i)]
\item \w{\oW{0}=\LL{\ua{m},\ub{m},\uc{m},\ue{3m+1}}} with simplicial augmentation
\w{\vare:\bW\sb{0}\to\As} given by \w[,]{\ua{m}\mapsto a}
\w[,]{\ub{m}\mapsto b} \w[,]{\uc{m}\mapsto c} and \w[.]{\ue{3m+1}\mapsto e}
\item \w{C\oW{1}=\LL{\ox{2m},\oy{2m},\oz{2m},\oox{2m+1},\ooy{2m+1},\ooz{2m+1}}}
with differential \w[,]{d(\oox{2m+1})=\ox{2m}}
\w[,]{d(\ooy{2m+1})=\oy{2m}} and \w[.]{d(\ooz{2m+1})=\oz{2m}}

The simplicial augmentation is given by \w[,]{\ox{2m}\mapsto[b,c]}
\w[,]{\oy{2m}\mapsto[c,a]} and \w[,]{\oz{2m}\mapsto[a,b]} while
\w[,]{\oox{2m+1}\mapsto x} \w[,]{\ooy{2m+1}\mapsto y} and \w[.]{\ooz{2m+1}\mapsto z}
\item \w{C\Sigma \oW{2}=\LL{\hw{3m+1},\hhw{3m+2}}}
with differential \w{d(\hhw{3m+2})=\hw{3m+1}} and augmentation
\w[,]{\hw{3m+1}\mapsto f} and \w{\hhw{3m+2}\mapsto -w} (the sign is the usual
one for the suspension in chain complexes).

\end{enumerate}
\item In dimension $1$ we have
\w[,]{\bW\sb{1}=\oW{1}\amalg C\oW{2}\amalg s\sb{0}\bW\sb{0}} where
\w[,]{s\sb{0}\bW\sb{0}} as a coproduct summand, is the image of \w{\bW\sb{0}}
under the simplicial degeneracy \w[.]{s\sb{0}:\bW\sb{0}\to\bW\sb{1}} We have:
\begin{enumerate}[(i)]
\item \w{\oW{1}=\LL{\ux{2m},\uy{2m},\uz{2m}}} with simplicial face maps
\w[,]{d\sb{0}(\ux{2m})=[\ub{m},\uc{m}]}
\w[,]{d\sb{0}(\uy{2m})=[\uc{m},\ua{m}]} and \w[,]{d\sb{0}(\uz{2m})=[\ua{m},\ub{m}]}
while \w[,]{d\sb{1}(\ux{2m})=\ox{2m}} \w[,]{d\sb{1}(\uy{2m})=\oy{2m}}
and \w[.]{d\sb{1}(\uz{2m})=\oz{2m}}
\item \w{C\oW{2}=\LL{\ow{3m},\oow{3m+1}}}
with differential \w{d(\oow{3m+1})=\ow{3m}} and simplicial face maps
\w[,]{d\sb{0}(\ow{3m})=-[\ua{m},\ox{2m}]-[\ub{m},\oy{2m}]-[\uc{m},\oz{2m}]} and
\w[,]{d\sb{0}(\oow{3m+1})=
-[\ua{m},\oox{2m+1}]-[\ub{m},\ooy{2m+1}]-[\uc{m},\ooz{2m+1}]}
\w[,]{d\sb{1}(\oow{3m+1})=\hw{3m+1}} while
\w{d\sb{1}(\ow{3m})=d\sb{1}(d(\oow{3m+1}))=d(d\sb{1}(\oow{3m+1}))
 =d(\hw{3m+1})=dd(\hhw{3m+2})=0} (see above).
\end{enumerate}
\item Finally, in dimension $2$ \w[,]{\oW{2}=\LL{\uw{3m}}} with simplicial face maps
$$
d\sb{0}(\uw{3m})=[s\sb{0}\ua{m},\ux{2m}]+[s\sb{0}\ub{m},\uy{2m}]+
[s\sb{0}\uc{m},\uz{2m}]
-s\sb{0}[\ua{m},\ux{2m}]-s\sb{0}[\ub{m},\uy{2m}]-s\sb{0}[\uc{m},\uz{2m}]
$$
and \w[.]{d\sb{1}(\uw{3m})=\ow{3m}}
\end{enumerate}

If we try to augment this simplicial DGL to \w[,]{\Bs} rather than \w[,]{\As} we see
that necessarily \w[,]{\hw{3m+1}\mapsto f} but then we have nowhere to map
\w[,]{\hhw{3m+2}} precisely because the Massey product
\w{f:=[a,x]+[b,y]+[c,z]} survives in \w[.]{\Bs}

This shows us that \w{\As} and \w{\Bs} are not homotopy equivalent, as expected.
\end{mysubsection}

%
%
\sect{Higher homotopy invariants for maps}
\label{chhim}

The systems of higher homotopy operations described in Section \ref{chhio} for a
\PAal \w{\Lambda=\piA\bY} may be thought of as obstructions to realizing an algebraic
isomorphism \w{\vth:\Lambda\cong\piA\bZ} by a map \w{f:\bY\to\bZ} (necessarily
an $\bA$-equivalence) \wh as well as constituting a complete set of higher invariants for
the $\bA$-homotopy type of objects in $\C$ realizing the given $\Lambda$.
In this section we address the analogous problem for arbitrary morphisms of \PAal[s.]

\begin{mysubsection}{$\A$-invisible maps}
\label{sainvis}
We begin with the simple but important case of a map \w{f:\bY\to\bZ} which is
``$\A$-invisible'' \wh that is, induces the zero map
\w[.]{0=f\sb{\#}:\piA\bY\to\piA\bZ} We can think of the associated higher invariants
as obstructions to $f$ being nullhomotopic. (Note that we do not have an
analogous situation for objects \w[:]{\bY\in\C} if \w[,]{\piA\bY=0} the map
\w{\bY\to\ast} is an $\bA$-equivalence, so $\bY$ is $\bA$-weakly contractible.)

Consider a sequential realization $\cW$ for $\bY$, with (actual) augmentations
\w{\bve{n}:\W{n}\to\bY} \wb[,]{n\geq 0} starting with
\w[.]{\bve{0}\sb{0}:\oW{0}\to\bY}
Since we assumed \w[,]{f\sb{\#}=0} the composite
\w{f':=f\circ\bve{0}\sb{0}:\oW{0}\to\bZ} is nullhomotopic, and we may choose a
nullhomotopy \w[.]{H\sb{0}:f'\sim 0} On the other hand, we have a nullhomotopy
\w{\Fk{0}:C\oW{1}\to\bY} for \w{\bve{0}\sb{0}\circ\udz{1}} (as in the second step
in the proof of Proposition \ref{pobst}). Thus we have the (ordinary) Toda bracket
$$
\lra{f,\,\bve{0}\sb{0},\,\udz{1}}~\subseteq~[\Sigma\oW{1},\bZ]~,
$$
\noindent associated to the diagram
\mydiagram[\label{eqtodaone}]{
&&&&\\
\oW{1}~\ar[rr]\sp{\udz{1}} \ar@/^{3.0pc}/[rrrr]\sp{\ast} &&
\Wn{0}{0} \ar[rr]\sp{\bve{0}\sb{0}} \ar@{=>}[u]\sp{\Fk{0}}
\ar@/_{3.2pc}/[rrrr]\sb{\ast}&& \bY \ar[rr]\sp{f} \ar@{=>}[d]\sp{H\sb{0}} && \bZ \\
&&&&
}
\noindent (compare \wref[).]{eqhochain} This serves as the first obstruction
to extending \w{H\sb{0}} to compatible nullhomotopies of the augmentations
\w[,]{f\circ\bve{n}:\W{n}\to\bZ} which would induce a nullhomotopy of the
map \w{\|f\circ\bv\|:\|\Wd\|\simeq\bY\to\bZ} (homotopic to the original $f$).

This is a special case of a more general setup:
\end{mysubsection}

\begin{mysubsection}{$n$-homotopies}
\label{dshhom}
For our new version of \S \ref{dshho}, the basic initial data \w{(\star)}
consists of two maps \w{\fu{0},\fu{1}:\bY\to\bZ} in $\C$ which induce the same
homomorphism of \PAal[s] \w[.]{\psi:\piA\bY\to\piA\bZ}
The specific initial data \w{(\star\star)} consists of a sequential
realization $\cW$ of a CW-resolution \w{\vare:\Vd\to\Lambda:=\piA\bY} for $\bY$
(induced by the augmentation \w[)]{\bv:\Wd\to\bY}
together with a homotopy \w{\Huk{0}{0}:\oW{0}\to\ZI} between
\w{\fu{0}\circ\bv\sb{0}} and \w{\fu{1}\circ\bv\sb{0}} making the following
diagram commute:
\mydiagram[\label{eqhukzero}]{
\Wn{0}{0}=\oW{0}~\ar[rrr]\sp{\Huk{0}{0}} \ar[d]\sb{\bve{0}\sb{0}} &&&
\ZI \ar[d]\sp{e=\ev\sb{0}\top\ev\sb{1}} \\
\bY \ar[rrr]\sp{\fu{0}\top\fu{1}} &&& \bZ\times\bZ ~.
}
\noindent Here factoring the diagonal \w{\Delta:\bZ\to\bZ\times\bZ} as an acyclic
cofibration \w{\bZ\hra\ZI} followed by a fibration \w{e:\ZI\epic\bZ\times\bZ}
makes \w{\ZI} a path object for $\bZ$ in the sense of \cite[\S 7.3.1]{PHirM}.
In general, \w{\ev\sb{j}:\ZI\to\bZ} \wb{j=0,1} are given by the structure maps for the
product; if \w[,]{\C=\Top} we may choose \w[,]{\ZI:=\bZ\sp{[0,1]}} with the obvious
evaluation maps \w{\ev\sb{j}} induced by the inclusions \w{i\sb{j}:\{j\}\hra[0,1]}
\wb[.]{j=0,1}

Since $\bZ$ is fibrant, the projections \w{\proj\sb{j}:\bZ\times\bZ\to\bZ} are fibrations,
and thus the composite \w{\proj\sb{j}\circ e:\ZI\epic\bZ} is a trivial fibration.
Since $\bZ$ is also cofibrant, this  map has a splitting \w{\sigma\sb{j}:\bZ\to\ZI}
\wb[.]{j=0,1} When \w{\C=\Top} and \w[,]{\ZI:=\bZ\sp{[0,1]}} we may let
\w[,]{\sigma\sb{0}=\sigma\sb{1}} sending \w{z\in\bZ} to the constant path at $z$.

As in \S \ref{szeroaug}, the map \w{\Huk{0}{0}} in \wref{eqhukzero} defines
a $0$-map \w[,]{\Huk{}{0}:={H'}\bp{0}\circ p:\W{0}\to\ZI\sp{\Du}}
which we call a $0$-\emph{homotopy} for \w[.]{\lra{\cW,\fu{0},\fu{1}}}

For any \w[,]{n\geq 1} we then have a corresponding notion of an $n$-map
as in \S \ref{dshho}, with \w[,]{\bX:=\ZI} called an $n$-\emph{homotopy} for
\w[:]{\lra{\cW,\fu{0},\fu{1}}} namely, a map \w{\Huk{}{n}:\W{n}\to\ZI\sp{\Du}}
extending the given $0$-homotopy \w[.]{\Huk{}{0}}
More generally we say \w{\Huk{}{n}} \emph{extends}
an \wwb{n-1}homotopy \w{\Huk{}{n-1}} if \w{\Huk{}{n}\circ\prn{n}=\Huk{}{n-1}}
(see \wref[).]{eqtower} A \emph{difference strand}  is a sequence
\w{\Htl{\infty}:=(\Huk{}{n})\sb{n=0}\sp{\infty}} such that
\w{\Huk{}{n}} extends \w{\Huk{}{n-1}} for each \w[.]{n\geq 1}
\end{mysubsection}

\begin{mysubsection}{Extending $n$-homotopies}
\label{snhtpy}
by Proposition \ref{pfoldpolymap}, given an \wwb{n-1}homotopy
\w{\Huk{}{n-1}:\W{n-1}\to\ZI\sp{\Du}} any choice of maps
\w{\hHuk{k}{n}:\oCsW{n-k-1}{n}\to\ZI\sp{\Deln{k}}} \wb{0\leq k\leq
n} satisfying \wref{eqnewfaces} determines a unique extension to
\w[.]{\tHuk{n}{}:\oW{n}\otimes\Pn{n}\to\ZI} We associate to this
data a map \w{\huk{n-1}:\oW{n}\otimes\partial\Pn{n}\to\ZI} as in
\S \ref{dvalstrand}, and (if $\C$ satisfies the assumptions of \S
\ref{spup}), this map is uniquely determined up to homotopy by the
induced map \w[.]{g\sb{n-1}:\Sigma\sp{n-1}\oW{n}\to\ZI} The
\emph{value} of the \wwb{n-1}homotopy \w{\Huk{}{n-1}} is then the
class \w{\Val{\Huk{}{n-1}}:=[g\sb{n-1}]} in
\w[,]{\Lambda\lin{\Sigma\sp{n-1}\oW{n}}} for
\w[.]{\Lambda:=\piA\bZ}

Moreover, by Proposition \ref{pvanish} the value for \w{\Huk{}{n-1}} is zero if and
only if the \wwb{n-1}homotopy extends to an $n$-homotopy. However,
  unlike the values \w{\Val{\bve{n}}} of \S \ref{dextnaug}, this depends also on our
initial choice of a $0$-homotopy \w{\Huk{}{0}} for \w[,]{\lra{\cW,\fu{0},\fu{1}}}
(see \S \ref{dshhom}).
Thus the specific initial data \w{(\star\star)} consists here of \w[.]{\lra{\cW,\Huk{}{0}}}
\end{mysubsection}

\begin{defn}\label{dcorrespstrmap}
Given \w{(\star)=(\fu{0},\fu{1}:\bY\to\bZ)} as in \S \ref{dshhom},
an $n$-stage comparison map \w{\Phi:\cW\to\,\ccWp} between two
  sequential realizations for $\bY$ as in \wref[,]{eqncorresp} and
  two $n$-homotopies \w{\Huk{}{n}} and \w{\Hupk{}{n}} for $\cW$ and \w[,]{\ccWp}
respectively,  as in \S \ref{dcorrespstr} we write \w{\Hupk{}{n}=\rs(\Huk{}{n})} if
\w{\Hupk{k}{m}=\rnk{m}{k}\circ\Huk{k}{m}:\Wpn{k}{m}\to\ZI\sp{\Deln{k}}}
and \w{\Huk{}{n}=\es(\Hupk{}{n})} if
\w{\Huk{k}{m}=\enk{m}{k}\circ\Hupk{k}{m}:\Wn{k}{m}\to\ZI\sp{\Deln{k}}}
for each \w[.]{0\leq k\leq m\leq n}

By \wref{eqcorrvals} we have
\begin{myeq}\label{eqcorrvalsmap}
\Val{\rs(\Huk{}{n})}=(\oar{n})\sb{\ast}(\Val{\Huk{}{n}})\hs \text{and}\hsm
\Val{\es(\Hupk{}{n})}=(\oen{n})\sb{\ast}(\Val{\Hupk{}{n}})~,
\end{myeq}
\noindent so by Lemma \ref{lindvanish}:
\begin{myeq}\label{eqvancorrmap}
\begin{array}{l}
(a)\hs \Val{\Huk{}{n}}=0\hsm \text{if and only if}\hsn \Val{\rs(\Huk{}{n})}=0\\
(b)\hs \text{If}\hsn\Val{\Hupk{}{n}}=0\hsm \text{then}\hsn \Val{\es(\Hupk{}{n})}=0~.
\end{array}
\end{myeq}
\end{defn}

\begin{defn}\label{dvanish}
Given \w{(\star)=(\fu{0},\fu{1}:\bY\to\bZ)} as in \S \ref{dshhom},
the universal homotopy operation \w{\llrra{\star}} of \S \ref{snouho}
will be denoted by \w[.]{\llrra{\fu{0,1}}=(\llrr{\fu{0,1}}{n})\sb{n=2}\sp{\infty}}
\end{defn}

We then have the following analogue of Theorem \ref{tvanish}:

\begin{thm}\label{tvanishmap}
  Let \w{\fu{0},\fu{1}:\bY\to\bZ} be maps between fibrant and cofibrant objects in
$\C$, inducing the same map of \PAal[s] \w[.]{\psi:\piA\bY\to\piA\bZ}
  Then \w{\fu{0}} and \w{\fu{1}} are $\bA$-equivalent (see \S \ref{rrightbous})
  if and only if the associated system of higher operations \w{\llrra{\fu{0,1}}}
  vanishes for some (and thus for any) sequential realization $\cW$ of $\bY$.
\end{thm}

\begin{proof}
If \w{\fu{0}} and \w{\fu{1}} are $\bA$-equivalent, then by Remark \ref{rrightbous}
\w{\CWA{\fu{0}}} and \w{\CWA{\fu{1}}} are homotopic. By post-composing
the augmentation of a sequential realization $\cW$ of \w{\CWA{\bY}} with the
$\bA$-equivalence \w[,]{\CWA{\bY}\to\bY} we may think of $\cW$ as
a sequential realization of $\bY$ (see Corollary \ref{csecoper}). Similarly, we have
a natural levelwise $\bA$-equivalence \w[.]{h:\Path(\CWA{\bZ})\to\ZI}
Therefore, given \w{G: \CWA{\bY}\to \Path(\CWA{\bZ})} providing a
homotopy \w{\CWA{\fu{0}}\sim\CWA{\fu{1}}} in $\C$ (\cite[\S 7.3.1]{PHirM}), the map
\w[,]{h\circ G:\CWA{\bY}\to\ZI} composed with each augmentation
\w[,]{\bve{n}:\W{n}\to\CWA{\bY}} defines a map \w[,]{\overline{G}\bp{n}:\W{n}\to\ZI}
which lifts by the splitting \w{p\sp{\ast}:\ZI\hra\ZI\sp{\Du}} (see \S \ref{raug}) to
\w[.]{G\bp{n}:\W{n}\to\ZI\sp{\Du}} This defines compatible
$n$-homotopies (see \S \ref{dshhom}) for all \w[,]{n\geq 1} showing that
\w{\llrra{\fu{0,1}}\sb{n}} vanishes by Proposition \ref{pvanish}.

Conversely, compatible $n$-homotopies for \w{n\geq 1} define a map
\w{H:\Wd\to\ZI\sp{\Du}} fitting into a commutative diagram of simplicial
objects:
\mydiagram[\label{eqpathobject}]{
\Wd~\ar[rrr]\sp{H} \ar[d]\sp{\bv} &&&
\ZI\sp{\Du} \ar[ld]\sb{\bev\sb{0}\sp{\Du}}  \ar[rd]\sp{\bev\sb{1}\sp{\Du}} & \\
\cd{\bY} \ar@/_{2.0pc}/[rrrr]\sb{\vfu{1}} \ar[rr]\sp{\vfu{0}} &&
\bZ\sp{\Du} \ar@/_{1.0pc}/[ur]\sb{\sigma\sb{0}\sp{\Du}} &&
\bZ\sp{\Du}~, \ar@/^{1.0pc}/[ul]\sp{\sigma\sb{1}\sp{\Du}}
}
\noindent where \w[,]{\vfu{j}:=p\sp{\ast}\circ\cd{\fu{j}}}
and the maps $\bv$, \w[,]{\bev\sb{j}} and \w{\sigma\sb{j}} \wb{j=0,1} are induced by
the corresponding maps of \S \ref{dshhom}, and \w{p\sp{\ast}:\cd{\bZ}\to\bZ\sp{\Du}}
is the Reedy weak equivalence of \S \ref{raug} (this time for $\bZ$).

Applying geometric realization to \wref{eqpathobject} yields a path object
$$
\ev\sb{0}',\ev\sb{1}':\|\ZI\sp{\Du}\|~\to~\|\bZ\sp{\Du}\|~\simeq~\bZ
$$
\noindent (see \cite[I, \S 1]{QuiH}), and thus
\w{\| H\|:\|\Wd\|\to\|\ZI\sp{\Du}\|} is a homotopy between
\w{\|p\sp{\ast}\|\circ\fu{0}\circ\|\bv\|} and \w[.]{\|p\sp{\ast}\|\circ\fu{1}\circ\|\bv\|}
Since \w{\|p\sp{\ast}\|:\bZ\to\|\bZ\sp{\Du}\|} is a weak equivalence and
\w{\|\bv\|:\|\Wd\|\to\bY} is an $\bA$-equivalence, this implies that
\w{\fu{0}} and \w{\fu{1}} are $\bA$-equivalent.
\end{proof}

From Theorem \ref{tvanishmap} and Proposition \ref{pvanish} we deduce:

\begin{corollary}\label{ccimaps}
If \w{\fu{0},\fu{1}:\bY\to\bZ} have
\w[,]{\fu{0}\sb{\#}=\fu{1}\sb{\#}:\piA\bY\to\piA\bZ} the
system of higher operations \w{\llrra{\fu{0,1}}} is a complete set of
invariants for distinguishing between the $\bA$-equivalence classes
\w{[\fu{0}]} and \w{[\fu{1}]} in \w{[\bY,\bZ]\sb{\bA}} (see \S \ref{rrightbous}).
\end{corollary}

\begin{mysubsection}{An example of an $\A$-invisible map}
\label{sexaim}
Consider the pinch map
\w[,]{\nabla:\Sigma\sp{n-1}\RR P\sp{2}=\bS{n}\cup\sb{2}\be{n+1}\to\bS{n+1}}
which collapses \w{\bS{n}} to a point.  If we apply the \wwb{n+1}Postnikov
section functor \w{P\sp{n+1}} to it, we obtain a map \w{f:\bY\to\bZ=\KZ{n+1}}
which represents  \w[,]{\beta\sb{\ZZ}(\iota\sb{n})} where
\w{\beta\sb{\ZZ}:H\sp{n}(\bY;\ZZ/2)\to H\sp{n+1}(\bY;\ZZ)} is the Bockstein and
\w{0\neq\iota\sb{n}\in H\sp{n}(\bY;\ZZ/2)=\ZZ/2} (see \cite[Ch.\ 3]{MTanC}).
In particular, $f$ is trivial in \w[.]{\pis} Thus we can use the simplified approach
of \S \ref{sainvis}:

The cofibration sequence
\mydiagram[\label{eqchcxres}]{
\bS{n}~\ar[r]\sp{2} & \bS{n}~\ar[r]^-{i} \ar@/^2em/[rr]^{\ast}&
\Sigma\sp{n-1}\RR P\sp{2} \ar[r]^-{\nabla} & \bS{n+1}
}
\noindent is also a fibration sequence in the stable range, so for \w{n\geq 3} we
have a free chain complex resolution of \wwb{n+1}truncated \Pa[s]:
\myvdiag[\label{eqchcxr}]{
\oV{3}=\pis\bS{n+1} \ar[r]\sp{2} & \oV{2}=\pis\bS{n+1} \ar[r]\sp{\eta\sp{n}} &
\oV{1}=\pis\bS{n} \ar[r]\sp{2} & \oV{0}=\pis\bS{n} \ar[r] & \Lambda=\pis\bY.
}
\noindent Thus the simplicial resolution \w{\oW{1}\toto\bW\sb{0}\xra{\vare}\bY}
in dimensions \www[,]{\leq 1} together with the map \w[,]{f:\bY\to\bZ} is given by
$$
\xymatrix@R=7pt{
\bS{n}~\ar[rr]\sp{\odz{}=2} \ar@{^{(}->}[rrdd]\sb{d\sb{1}} &&
\bS{n}~\ar@{^{(}->}[rr]^-{\inc} &&
\bS{n}\cup\sb{2}\be{n+1} \ar[rr]\sp{\nabla}&&\bS{n+1}\\
&& \vee && \\
&& C\bS{n}\ar[rruu]\sb{H} &&
}
$$
\noindent where $H$ is a nullhomotopy for \w[.]{2\cdot\inc}
Since \w{\nabla\circ\inc} is zero, diagram \wref{eqtodaone} simplifies to
the solid portion of
\myzdiag[\label{eqtoda}]{
C\bS{n} \ar@{.>}[rr]^{\Id\sb{C\bS{n}}} && C\bS{n}  \ar[drrrr]\sp{0} \\
\bS{n}~\ar[rr]\sp{2} \ar@{^{(}->}[d]\sp{\iota} \ar@{^{(}.>}[u]\sp{\iota}&&
\bS{n}~\ar@{^{(}->}[rr]^-{\inc}  \ar@{^{(}->}[u]\sp{\iota}&&
\bS{n}\cup\sb{2}\be{n+1} \ar[rr]\sb{\nabla} && \bS{n+1} ~,\\
C\bS{n}\ar[rrrru]\sb{g} &&&&&&
}
\noindent Here $g$, the structure map for the pushout defining
  \w[,]{\bS{n}\cup\sb{2}\be{n+1}} is the identity on the interior
  of \w[.]{C\bS{n}=\be{n+1}} Since the left copy of \w{\bS{n}} maps by $0$ to
  \w[,]{\bS{n+1}} the associated Toda bracket is simply the map
\w[,]{\Sigma\bS{n}\cong C\bS{n}/\bS{n}\xra{\cong}\bS{n+1}} which has degree $1$.

Since the indeterminacy is \w{2\cdot[\Sigma\bS{n},\,\bS{n+1}]=2\ZZ} inside
\w[,]{\pi\sb{n+1}\bS{n+1}=\ZZ} the Toda bracket does not vanish, which shows
(as expected) that $f$ is non-trivial, despite inducing the zero map in
\w{\pis} in the relevant range.
\end{mysubsection}

%
%
\sect{A filtration index invariant}
\label{cfii}

As an application of our methods, we show how our constructions can be used to
describe explicitly a certain filtration index invariant for mod $p$
cohomology classes, dual to the Adams filtration for elements in
\w[.]{\pis\widetilde{\bY}\sb{p}}

\supsect{\protect{\ref{cfii}}.A}{A reverse Adams spectral sequence}

In order to define our index, we now set up an ad hoc cohomological reverse Adams
spectral sequence (see \cite{BlaO}), which is not particularly well behaved
or accessible to computation, but suffices to show that our index is indeed
homotopy invariant.

From now on, let \w[,]{\C=\Topa} \w[,]{\bA=\bS{1}} and \w[,]{\kappa=\omega}
as in \S \ref{egmapalg}, and let \w{\bY\in\Topa} be simply connected,
with the \Pa \w{\Lambda:=\pis\bY} of finite type (i.e., a finitely generated
abelian group in each degree).
We choose some CW-resolution \w{\Vd} of $\Lambda$, with CW basis
\w[,]{\{\oV{n}\}\sb{n=0}\sp{\infty}} and a sequential realization $\cW$.

Next, fix a prime $p$ and let $\bK$ be a strict topological Abelian group model
of \w[,]{\Kp{N}} for some \w{N>\!\!>0} to be determined later. Applying the functor
\w{\mapa(-,\bK)} dimensionwise to each simplicial space in \wref{eqtower} yields
\begin{myeq}\label{eqmtower}
\dotsc~\to~\X{n+1}~\xra{(\prn{n+1})\sp{\ast}}~\X{n}~\xra{(\prn{n})\sp{\ast}}~\X{n-1}~
\to~\dotsc~\to~\X{1}~\xra{(\prn{1})\sp{\ast}}~\X{0}~,
\end{myeq}
\noindent with \w[.]{\X{n}:=\mapa(\W{n},\,\bK)}
Since \w{M\sp{n}\mapa(\Wd,\bK)=\mapa(L\sb{n+1}\Wd,\bK)} for any simplicial space
\w{\Wd} (see \cite[VII, \S 1,4]{GJarS}), we see that each \w{\X{n}} is Reedy fibrant
(since \w{\W{n}} is Reedy cofibrant), and the maps in \wref{eqmtower} are Reedy
fibrations. Thus the (homotopy) limit of this tower is
\w[.]{\Xu:=\mapa(\Wd,\,\bK)}

Moreover, applying \w{\Tot:=\map\sb{c\cS}(\Du,-)} to \wref{eqmtower}
also yields a tower of fibrations
\begin{myeq}\label{eqttower}
\dotsc~\to~\Tot\X{n+1}~\xra{(\prn{n+1})\sp{\ast}}~\Tot\X{n}~\to~\dotsc~\to~
\Tot\X{1}~\xra{(\prn{1})\sp{\ast}}~\Tot\X{0}~,
\end{myeq}
\noindent by \cite[II, \S 2, SM7]{QuiH}, with
\w[.]{\Tot\X{n}\cong\mapa(\|\W{n}\|,\,\bK)} By \cite[XII, 4.3]{BKanH}, its
(homotopy) limit is thus:
\begin{myeq}\label{eqmtot}
\Tot\Xu~\cong~\mapa(\|\Wd\|,\,\bK)~\simeq~\mapa(\bY,\bK)
\end{myeq}

\begin{mysubsection}{Identifying the fibers}
\label{sifib}
Let \w{\Sigma\Ds\bp{n}} denote the chain complex in \w{\Topa} with
\w{\oCsW{n}{n-k-1}} in dimension $k$ (see \S \ref{smakesc}), and
\w{\sDd{n}:=\cL\E\Sigma\Ds\bp{n}} the
corresponding simplicial space (see \S \ref{dlmo}-\ref{schaincx}).

By \S \ref{dsrar}(ii), \w{\W{n-1}\hra\vWu{n}\to\sDd{n}} is a (homotopy) cofibration
sequence in \w[,]{\Topa\sp{\Dop}} so if we set \w{\sEd{n}:=\mapa(\sDd{n},\bK)} and
\w[,]{\vXu{n}:=\mapa(\vWu{n},\bK)} we have a (homotopy) fibration sequence
\begin{myeq}\label{eqhfib}
\sEd{n}~\hra~\vXu{n}~\epic~\X{n-1}
\end{myeq}
\noindent of cosimplicial spaces. Applying \w{\Tot} to \wref{eqhfib} yields
another fibration sequence. Since \w{\W{n}} and \w{\vWu{n}} are weakly equivalent
Reedy cofibrant simplicial spaces, by \wref{eqdoublereplace} (where \w{\vWu{n}}
is denoted by \w[),]{\hXd[F]} \w{\vXu{n}} and \w{\X{n}} are weakly equivalent
Reedy fibrant objects in \w[,]{\Sa\sp{\Delta}} so we have a homotopy fibration sequence
\begin{myeq}\label{eqhpb}
\Tot\sEd{n}~\to~\Tot\X{n}~\to~\Tot\X{n-1}
\end{myeq}
\noindent by \cite[XI, 4.3]{BKanH}.

Since the restricted simplicial space \w{\E\Sigma\Ds\bp{n}} is contractible in all
simplicial dimensions but $n$, and \w{\sEd{n}} is a strict cosimplicial simplicial Abelian
group, the homotopy spectral sequence for \w{\sEd{n}} (see \cite[X, 6]{BKanH})
collapses at the \ww{E\sb{2}}-term, and thus we have a weak equivalence of
\ww{\Fp}-GEMs:
\begin{myeq}\label{eqtote}
\Tot\sEd{n}~\simeq~\Omega\sp{n}\mapa(\oW{n},\,\bK)~.
\end{myeq}
\end{mysubsection}

\begin{mysubsection}{Identifying the $E\sb{2}$-terms}
\label{siett}
Now consider the homotopy spectral sequence of the tower of fibrations
\wref[,]{eqttower} with
$$
E\sb{1}\sp{n,i}~:=~\pi\sb{i}\Tot\sEd{n}~\Rightarrow~\pi\sb{i}\Tot\Xu~.
$$
\noindent From \wref[,]{eqtote} \wref[,]{eqmtot} and the fact that \w{\bK=\Kp{N}}
we see that this is:
\begin{myeq}\label{eqeto}
E\sb{1}\sp{n,i}=H\sp{N-i-n}(\oW{n};\Fp)~\Longrightarrow~H\sp{N-i}(\bY;\Fp)~.
\end{myeq}

In fact, from the description in \S \ref{crsar}.B we see that the $n$-th normalized
cochain object \w{N\sp{n}\Xu} of the cosimplicial space \w{\Xu\simeq\mapa(\Wd,\bK)}
is weakly equivalent to \w[,]{\mapa(\oW{n},\bK)} so in fact from the
\ww{E\sb{1}}-term on our spectral sequence is naturally isomorphic to the homotopy
spectral sequence for \w{\Xu} (see \cite[X, 6]{BKanH}).

Note that \w{\oW{n}} is, up to homotopy, a wedge of spheres realizing the free
\Pa \w[,]{\oV{n}} the $n$-th CW basis of the given free simplicial resolution
\w[.]{\Vd\to\Lambda:=\pis\bY} Moreover, since
\w{H\sp{k}(\oW{n};\Fp)\cong\Hom(H\sb{k}\oW{n},\,\Fp)} and we have a natural
identification of \w{H\sb{\ast}\oW{n}} with \w[,]{Q\oV{n}} where
\w{Q:\PAlg\to\gr\Abgp} is the indecomposables functor of \cite[\S 2.2.1]{BlaH},
we can write \w[,]{E\sb{1}\sp{n,i}\cong T\sp{N-i-n}\oV{n}} where the graded functor
\w{T:\PAlg\op\to\gr\VF} is defined for any \Pa $V$ by
\w[.]{T(V)=\Hom(QV,\Fp)}

Moreover, since \w{T\Vd} is a cosimplicial graded \ww{\Fp}-vector space,
we can calculate the cohomotopy groups \w{\pi\sp{n}T\Vd} using the
Moore cochain complex \w{C\sp{\ast}T\Vd} (see \S \ref{dmco} and compare
\cite[1.8]{BSenH}), and as in \cite[\S 2]{BJTurnHA} we have a natural isomorphism
\w[.]{C\sp{n}T\Vd\cong T\oV{n}} Therefore, as in \cite[\S 3]{BlaH}, we can identify
the \ww{E\sb{2}}-term of our spectral sequence as
\begin{myeq}\label{eqett}
E\sb{2}\sp{n,i}~\cong~[L\sb{n}T\sp{N-i-n}](\Lambda)~,
\end{myeq}
\noindent the $n$-th derived functor of $T$ (in degree \w[),]{N-i-n} applied to the
\Pa \w[.]{\Lambda:=\pis\bY}
\end{mysubsection}

\begin{remark}\label{rdepn}
Since any two sequential resolutions are connected by zigzags of
comparison maps (as we saw in Section \ref{ccsr}), and these
induce weak equivalences of simplicial resolutions (in the sense of Proposition
\ref{psimptal}), we see that the associated spectral sequences are all isomorphic from
the \ww{E\sb{2}}-term on.

Moreover, since we assume that \w{\Lambda:=\pis\bY} is of finite type, we can
\emph{choose} a CW resolution \w{\Vd\to\Lambda} with each \w{\oV{n}} (and thus
each \w[)]{V\sb{n}} of finite type \wh so each
\w[,]{E\sb{1}\sp{n,i}} and thus each \w[,]{E\sb{2}\sp{n,i}} will actually be a finite
dimensional \ww{\Fp}-vector space, and thus a finite set.  This guarantees that the
spectral sequence converges strongly for any choice of sequential
realization $\cW$ (see \cite[IX, 3]{BKanH}).

Finally, although the \ww{E\sb{1}}-term for this spectral sequence
as defined vanishes unless \w[,]{0\leq i\leq N-n} it is clear from the construction
that replacing \w{\bK=\Kp{N}} by \w{\Kp{N-1}} has the effect of applying loops
to every space in the tower \wref[,]{eqmtower} and thus in the tower \wref[,]{eqttower}
too \wh which results in simply re-indexing the spectral sequence by one in
the $i$-grading. Thus in order to calculate a differential on a particular
element $\alpha$ in \w[,]{E\sb{r}\sp{n,i}} we may simply choose $N$ large
enough so both the source and target are defined, and disregard the dependence
of our construction on $N$.
\end{remark}

We may summarize our results so far in the following:

\begin{prop}\label{prass}
For each sequential realization of a simply-connected finite type \w[,]{\bY\in\Topa}
\w{\cW=\lra{\W{n},\,\prn{n}, \Dsn{n}, \Fn{n}, \Tn{n}}\sb{n=0}\sp{\infty}}
and \w[,]{N\geq 1} there is a strongly convergent spectral sequence with
$$
E\sb{1}\sp{n,i}=H\sp{N-i-n}(\oW{n};\Fp)~\Longrightarrow~H\sp{N-i}(\bY;\Fp)~.
$$
\noindent The \ww{E\sb{2}}-term is independent of the choice of $\cW$, and if we
replace $N$ by \w[,]{N'=N+1} then \w{E\sb{2}\sp{n,i}} for the new spectral sequence
is isomorphic to \w{E\sb{2}\sp{n-1,i}} for the old whenever the latter is non-zero.
\end{prop}

\supsect{\protect{\ref{cfii}}.B}{The filtration index}

The \ww{E\sb{1}}-exact couple for our spectral sequence has the form:

\myvdiag[\label{eqeoneec}]{
\pi\sb{N-k+1}\Tot\X{n} \ar[d]\sp{(\prn{n})\sp{\ast}}
\ar[r]\sp(0.56){\partial\sb{n}} & H\sp{k-n-1}\oW{n+1} \ar[r]\sp(0.45){j\sp{n-1}} &
\pi\sb{N-k}\Tot\X{n+1} \ar[d]\sp{(\prn{n+1})\sp{\ast}}
\ar[r]\sp(0.55){\partial\sb{n+1}} & H\sp{k-n-1}\oW{n+2}\\
\pi\sb{N-k+1}\Tot\X{n-1} \ar[r]\sp(0.6){\partial\sb{n-1}} &
H\sp{k-n}\oW{n} \ar[r]\sp(0.43){j\sp{n-1}} &
\pi\sb{N-k}\Tot\X{n} \ar[r]\sp(0.55){\partial\sb{n}} & H\sp{k-n}\oW{n+1}.
}

\begin{defn}\label{dfilti}
Consider an element
\w[.]{\gamma\in H\sp{k}(\bY;\Fp)\cong\pi\sb{N-k}\mapa(\bY,\bK)\cong
\pi\sb{N-k}\Tot\Xu} Its \emph{filtration index} \w{I(\gamma)} is the least
\w{n\geq 0} such that the image of $\gamma$ in \w{\pi\sb{N-k}\Tot\X{n}} (under the
iterated fibrations \w{(\prn{n})\sp{\ast}} in \wref[)]{eqttower} is non-zero.
Convergence of the spectral sequence implies that \w{I(\gamma)=\infty} if and only
if \w[.]{\gamma=0}

From \wref{eqeoneec} we see that this image lifts (though not uniquely) to
\w[.]{\pi\sb{N-k}\Tot\sEd{n}\cong H\sp{k-n}(\oW{n};\Fp)} This means that
$\gamma$ is represented by an element in \w[,]{E\sb{\infty}\sp{n,N-k}} and thus in
\w[,]{E\sb{2}\sp{n,N-k}} which is independent of $\cW$.
\end{defn}

\begin{mysubsection}{Lifting nullhomotopies}
\label{slnullh}
We can represent \w{\gamma\in H\sp{k}(\bY;\Fp)} by a map \w[.]{g:\bY\to\bK:=\Kp{k}}
Precomposing with the augmentations \w{\bve{n}:\Wn{0}{n}\to\bY} yields
a particularly simple map \w[,]{\Gan{n}:\|\W{n}\|\to\bK} which we can think of
as a $0$-simplex in \w[.]{\Tot\X{n}=\mapa(\|\W{n}\|,\,\bK)}

Noting that \w[,]{\|\W{0}\|\simeq\oW{0}} we see that \w{\Gan{0}} is not
nullhomotopic \wh that is, \w{I(\gamma)=0} \wwh if and only if $g$ is
``visible to homotopy'' \wh that is, \w{g\sb{\#}:\pis\bY\to\pis\bK} is
non-trivial. Otherwise we can choose a nullhomotopy \w{\Gn{0}} for
\w[,]{\Gan{0}} and try to extend it inductively to a nullhomotopy \w{\Gn{n}}
for \w[,]{\Gan{n}} for the \emph{largest} possible \w[.]{n\geq 0}

We therefore assume that we have a nullhomotopy \w{\Gn{n-1}} for \w[.]{\Gan{n-1}}
To extend it, it is more convenient to work with the explicit description of
\w{\vWu{n}} in \S \ref{crsar}.B \wh in fact, in view of the coproduct decomposition
of \wref[,]{eqdopc} it suffices to extend \w{\Gn{n-1}} to \w[.]{\tWd{n}}
By Remark \ref{rdepn} (and Section \ref{ccsr}) we can use any
sequential realization we like, so we may assume for simplicity
that we use the standard sequential realization with
\w[.]{\oCsW{j}{n}=C\Sigma\sp{j}\oW{n}} for all \w{-1\leq j<n}

From the usual description of \w{\|\W{n}\|} in \cite[VII, 3]{GJarS}
(or of \w{\Tot\X{n}} in \cite[X, 3.2]{BKanH}), we think of \w{\Gan{n-1}}
as a map of simplicial spaces \w{\W{n-1}\to\bK\sp{\Del}} (which happens to
factor through the constant simplicial space \w[).]{\cd{\bY}} However,
\w{\Gn{n-1}:\W{n-1}\to\bK\sp{C\Del}} (viewed as a reduced path space)
  does not have this simple form (unless $g$ itself is nullhomotopic).
  An extension to \w{\tWd{n}} thus consists of a sequence of maps
  \w{H\sb{j}=H\sb{j}\bp{n}:C\Sigma\sp{n-j-1}\oW{n}\to\bK\sp{C\Deln{j}}}
fitting into a commutative diagram:

\mydiagram[\label{eqexnullh}]{
&&&&& \bK\sp{C\Deln{j}} \ar[lld]\sp{p}
\ar@<-3.5ex>[ddd]\sb{(C\eta\sp{0})\sp{\ast}}
\sp{(C\eta\sp{1})\sp{\ast}}\ar[ddd]\sp{\dotsc}
\ar@<2.5ex>[ddd]\sp{(C\eta\sp{j})\sp{\ast}} \\
\Wn{j}{n-1} \ar@/^{2.0pc}/[rrrrru]\sp(0.3){\Gnk{n}{j}}
\ar@<-2.5ex>[d]\sb{d\sb{0}}\sp{d\sb{1}} \ar[d]\sp{\dotsc} \ar@<2ex>[d]\sp{d\sb{j}} &
\amalg\hspace*{4mm} C\Sigma\sp{n-j-1}\oW{n} \ar@{.>}[rrrru]\sp{H\bp{n}\sb{j}}
\ar[ld]\sp(0.4){w\sb{j}\Fk{j}} \ar[d]^{\diff{j}{D}} \ar[rr]\sb{\tGank{n}{j}} &&
\bK\sp{\Deln{j}} \ar@<-2.5ex>[d]\sb{(\eta\sp{0})\sp{\ast}}
\sp{(\eta\sp{1})\sp{\ast}} \ar[d]\sp{\dotsc} \ar@<2.5ex>[d]\sp{(\eta\sp{j})\sp{\ast}} \\
\Wn{j-1}{n-1} \ar@/_{2.0pc}/[rrrrrd]\sb(0.3){\Gnk{n}{j-1}} &
\amalg\hspace*{4mm}C\Sigma\sp{n-j}\oW{n} \ar[rr]\sp{\tGank{n}{j-1}}
\ar@{.>}[rrrrd]\sb{H\bp{n}\sb{j-1}} && \bK\sp{\Deln{j-1}}\\
&&&&& \bK\sp{C\Deln{j-1}} \ar[llu]\sb{p}
}
\noindent for each \w[,]{1\leq j\leq n} where the path fibration
\w{p:\bK\sp{C\Deln{j}}\to\bK\sp{\Deln{j}}} is induced by the inclusion of
the cone base \w[.]{\Deln{j}\hra C\Deln{j}}

Note that the maps
\w{\ttGank{n}{j}:\tWn{n}{j}=\Wn{n}{j}\amalg C\Sigma\sp{n-j-1}\oW{n}\to\bK\sp{\Deln{j}}}
have the form \w[,]{\Gank{n-1}j\bot h\sb{j}} where
\w{h\sb{j}:C\Sigma\sp{n-j-1}\oW{n}\to\bK\sp{\Deln{j}}} factors through the
iterated face map \w[,]{D\sb{j}:C\Sigma\sp{n-j-1}\oW{n}\to\bY} and thus through
\w[,]{d\sb{j}} which vanishes for \w{j\geq 2} by the description in the proof of
Lemma \ref{lconef}.

Identifying \w{C\Deln{j}} with \w[,]{\Deln{j+1}} as in \S \ref{sfoldp}
the adjoint of each \w{H\sb{j}\bp{n}} defines a pointed map
\w[.]{\tHuk{n}{j}:\oW{n}\otimes\Deln{n+1}\to\bK}
If we denote the copy of \w{\Deln{n+1}} associated to \w{\tHuk{n}{j}} by
\w[,]{\Delnk{n+1}{j}} then the $0$-th facet of \w{\Delnk{n+1}{j}} corresponds
to the cone direction of \w{C\Deln{j}} in the adjunction to \w[,]{\bK\sp{C\Deln{j}}}
facets \w{1,\dotsc,j+1} correspond to facets \w{0,\dotsc,j} of \w{\Deln{j}}
 the next \w{n-j-1} facets correspond to the suspension directions of
 \w[,]{C\Sigma\sp{n-j-1}\oW{n}} and the \wwb{n+1}st facet corresponds to the
 cone direction of \w[.]{C\Sigma\sp{n-j-1}\oW{n}}

Commutativity of \wref{eqexnullh} translates into the requirement that
\begin{myeq}\label{eqadjfaces}
\tHuk{n}{j}\circ\eta\sp{i}~=~\begin{cases}
\widetilde{h\sb{j}} & \hsm\text{if}\ i=0 \\
\widetilde{\Gnk{n}{j-1}w\sb{j}\Fk{j}} & \hsm\text{if}\ i=1 \\
\widetilde{H\bp{n}\sb{j-1}\diff{j}{D}} & \hsm\text{if \ $i=2$ \ and} \ j\geq 1\\
\widetilde{H\bp{n}\sb{j}\diff{j+1}{D}} & \hsm\text{if \ $i=n+1$ \ and} \ j<n \\
0~&\hsm\text{otherwise}~.
\end{cases}
\end{myeq}
\noindent for each \w[.]{0\leq i\leq j\leq n}
\end{mysubsection}

\begin{defn}\label{dmfoldp}
For each \w[,]{n\geq 1} the $n$-th \emph{modified folding polytope} \w{\hP{n}}
is obtained from the disjoint union of \w{n} $n$-simplices
\w{\Delnk{n}{0},\dotsc,\Delnk{n}{n-1}} by identifying
\w{\partial\sb{n}\Delnk{n}{j-1}} with \w{\partial\sb{2}\Delnk{n}{j}}
for each \w{1\leq j\leq n} (see \cite[\S 4.11]{BBSenH}). Its \emph{boundary}
\w{\partial\hP{n}} is the image of all facets of the $n$-simplices
\w{\Delnk{n}{j}} \wb{1\leq j\leq n} not identified as above.

Note that a nullhomotopy  \w{\Gn{n-1}:\W{n-1}\to\bK\sp{C\Del}} for \w{\Gan{n-1}}
determines a pointed map \w{\Psi'\lo{F,G}:\oW{n}\otimes\partial\hP{n+1}\to\bK}
with \w[,]{\Psi'\lo{F,G}\rest{\partial_{1}\Delnk{n+1}{k}}=
\widetilde{\Gnk{n}{j-1}w\sb{j}\Fk{j}}}
and \w{\Psi'\lo{F,G}=\ast} on all other (non-identified) facets of
\w[.]{\Delnk{n+1}{k}} As in \S \ref{dvalstrand}, \w{\Psi'\lo{F,G}}
induces a unique map \w[.]{\Psi=\Psi\lo{F,G}:\oW{n}\wedge\partial\hP{n+1}\to\bK}
\end{defn}

We now have the following analog of Lemma \ref{lfoldpoly}:

\begin{lemma}\label{lmfoldp}
For each \w[,]{n\geq 1} the pair \w{(\hP{n},\partial\hP{n})} is
homeomorphic to \w[.]{(\bD\sp{n},\,\bS{n-1})}
\end{lemma}

Choosing \w{f=*} in
Propositions \ref{pfoldpolymap} and \ref{pvanish}, we have:

\begin{prop}\label{pmfoldpmap}
Given a sequential realization $\cW$ for $\bY$ as above, a
map \w{g:\bY\to\bK=\Kp{k}} extending by iterated face maps to
\w{\Gan{m}:\|\W{m}\|\to\bK\sp{\Del}} for each \w[,]{m\geq 0}
and a nullhomotopy \w{\Gn{n-1}:\W{n-1}\to\bK\sp{C\Del}} for \w[,]{\Gan{n-1}}
the map \w{\Psi\lo{F,G}:\oW{n}\wedge\partial\hP{n+1}\to\bK} of \S \ref{dmfoldp}
is null-homotopic if and only if \w{\Gn{n-1}} extends to a nullhomotopy
\w{\Gn{n}:\W{n}\to\bK\sp{C\Del}} for \w[.]{\Gan{n}}
\end{prop}

\begin{mysubsection}{Higher homotopy operations and the filtration index}
\label{sindhho}
Although the maps \w{\Psi\lo{F,G}} were formally defined only for \emph{standard}
sequential realizations (with \w{\oCsW{j}{n}=C\Sigma\sp{j}\oW{n}} for all $j$ and $n$),
one can show (as in the proof of Proposition \ref{pfoldpolymap}) that
Proposition \ref{pmfoldpmap} in fact holds for any sequential realization $\cW$.

This allows us to think of the cohomology class
\w{[\widetilde{\Psi\lo{F,G}}]\in H\sp{k-n}(\oW{n};\Fp)} as the \emph{value}
of a \emph{system of higher homotopy operations}
\w{\llrra{\gamma}=(\llrr{\gamma}{n})\sb{n=1}\sp{\infty}} associated to the class
\w[.]{\gamma\in H\sp{k}(\bY;\Fp)} This value is determined by the choice of
a nullhomotopy \w{\Gn{n-1}} in \w[,]{\Tot\X{n}=\mapa(\|\W{n}\|,\,\bK)} and serves
as the obstruction to lifting it to \w[.]{\Gn{n}}

Moreover, one can use Theorem \ref{tzigzag} to show, as in the proof of Lemma
\ref{lkey}, that if for \emph{some} sequential realization $\cW$,
\w{\Gan{n}:\|\W{n}\|\to\bK\sp{\Del}} has a nullhomotopy \w[,]{\Gn{n}} then this holds
for every sequential realization.

We may thus summarize the situation in the following
\end{mysubsection}

\begin{prop}\label{pindhho}
The filtration index \w{I(\gamma)} of a cohomology class
\w{\gamma\in H\sp{k}(\bY;\Fp)} is the largest $n$ for which \w{\llrr{\gamma}{n}}
does not vanish.
In particular, the system \w{\llrra{\gamma}} vanishes coherently if
and only if \w[.]{\gamma=0}
\end{prop}

\begin{example}\label{eqsecopn}
  Let \w{\bY=P\sp{n+1}(\bS{n}\cup\sb{2}\be{n+1})} (the \wwb{n+1}Postnikov section
of a Moore space), with \w{f:\bY\to\bZ=\KZ{n+1}}
induced by the pinch map, as in \S \ref{sexaim}, and let
\w{g:\bY\to\KP{\Ft}{n+1}} be the composite \w[,]{\rho\circ f} with \
\w{\rho:\KZ{n+1}\to\KP{\Ft}{n+1}} the reduction mod $2$ map. Thus
\w[,]{\gamma:=[g]=\Sq{1}\circ p} with \w{p:\bY\to\KP{\Ft}{n}} the Postnikov
fibration.

As in \wref[,]{eqchcxr} \w{g\sb{\#}:\pis\bY\to\pis\KP{\Ft}{n+1}} is trivial,
and since the Toda bracket \w{\lra{\nabla,\inc,2}} of \wref{eqtoda} is nontrivial,
the same is true for \w{\lra{g,\inc,2}} (see \cite[I]{TodC}). Thus
\w[,]{\llrr{\gamma}{1}\neq 0} so $\gamma$ has filtration index $1$.
\end{example}

\appendix
%
%
\sect{Comparison Maps}\label{Appa}

In this appendix we state and prove two facts about the comparison maps
of Section \ref{ccsr}  needed in the paper;  we deferred the proofs until now
because they are somewhat technical.

\begin{defn}\label{fibrational}
A sequential realization $\cW$ will be called \emph{fibrational}
if for each \w[,]{n\geq 1} the chain map \w{\Fn{n}:\Dsn{n}\epic\cMs\W{n-1}} is a
(levelwise) fibration in the (projective) model category of chain complexes
\w{\ChC\sp{\leq n-1}} of \S \ref{smodchaincx}.
\end{defn}

We now have a mild extension of Theorem \ref{tres}.

\begin{lemma}\label{tresext}
For \w{\bA\in\C} as in \S \ref{setmc},
any CW-resolution \w{\Vd} of a realizable \PAal \w{\Lambda=\piA\bY}
has a fibrational sequential realization
$$
\cW=\lra{\W{n},\,\prn{n}, \Dsn{n}, \Fn{n}, \Tn{n}}\sb{n=0}\sp{\infty}.
$$
\end{lemma}

\begin{proof}
  To make  the sequential realization $\cW$ of Theorem \ref{tres} fibrational,
  we factor $F$ as an acyclic cofibration \w{\Tn{n}:\Dsn{n}(\oW{n})\to\Dsn{n}}
  followed by a fibration \w{\Fn{n}:\Dsn{n}\epic\cMs\W{n-1}}
in the model category \w[.]{\ChC\sp{\leq n-1}}

\end{proof}

Note that the vertical acyclic cofibrations of \wref{eqmodconesusp} are obtained
from the map \w[.]{\Tn{n}}

\begin{prop}\label{pralgcomp}
For any algebraic comparison map \w{\Psi:\Vd\to\Vdp} for $\bY$ and sequential
realization $\cW$ of \w[,]{\Vd} there is a fibrational sequential
realization \w{\ccWp} of \w{\Vdp} with a comparison map \w{\Phi:\cW\to\,\ccWp}
over $\Psi$.
\end{prop}

\begin{proof}
We construct \w[,]{\ccWp} with the cofibrations \w{\en{n}:\W{n}\hra\Wp{n}}
and retractions \w{\rn{n}:\Wp{n}\to\W{n}} by induction on \w[:]{n\geq 0}

Since \w{\oV{n}} is a coproduct summand in \w[,]{\oVp{n}=\oU{n}\amalg\oV{n}}
say, if we realize \w{\oV{n}} by \w{\oW{n}} and \w{\oU{n}} by \w{\oX{n}}
then \w{\oVp{n}} is realized by \w[.]{\hWp{n}:=\oX{n}\amalg\oW{n}}
By Definition \ref{dsrar}, the $n$-th stage of $\cW$ is determined
by the choice of strongly cofibrant replacement \w{\bDs}
of \w[,]{\oW{n}\oS{n-1}} equipped with a levelwise fibration
\w{F:\bDs\to\cMs\W{n-1}} realizing the given attaching map
\w[.]{\odz{n}:\oV{n}\to C\sb{n-1}\Vd}  If \w{\bGs} is similarly a strongly cofibrant
replacement for \w[,]{\oX{n}\oS{n-1}} note that the attaching map
\w{\ppp\odz{n}:\oVp{n}\to C\sb{n-1}\Vdp} has the form \w[,]{\odz{n}\bot\tau}
and we may realize \w{\tau:\oU{n}\to C\sb{n-1}\Vdp} by
\w[,]{T:\bGs\to\cMs\Wp{n-1}} with \w{\ppp\oar{n}:\oX{n}\to\oW{n}} inducing
\w[.]{R:\bGs\to\bDs}  We then realize \w{\orh{n}:\oVp{} \rightarrow {\oV{n}}}
by \w[.]{\ppp\oar{n}\bot\Id\sb{\oW{n}}:\oX{n}\amalg\oW{n}\to\oW{n}}

Following \S \ref{crsar}.A-B, we now consider the following diagram in
the projective model category of $n$-truncated chain complexes over $\C$,
in which \w{\bPs} is the pullback of the lower right square.
\mydiagram[\label{eqpbatt}]{
\bGs \ar@{.>}[rr]\sb{S} \ar@/^2.5pc/[rrrr]\sp{R}
\ar@/_1.0pc/[rrd]\sb{T} &&
\ar @{} [drr]|(0.21){\framebox{\scriptsize{PB}}}
\bPs \ar[rr]\sb{p} \ar@{->>}[d]\sb{q} &&
\bDs \ar@{->>}[d]\sp{F} \ar@/_0.9pc/[ll]\sb{e} \\
&& \cMs\Wp{n-1} \ar[rr]\sp{\cMs\rn{n-1}} &&
\cMs\W{n-1} \ar@/^1.1pc/[ll]\sp{\cMs\en{n-1}}
}
\noindent and the section $e$ for $p$ is induced by the section
\w{\cMs\en{n-1}} for \w[.]{\cMs\rn{n-1}}

Since by Definition \ref{dacomp}
\w[,]{C\sb{n-1}\rho\circ\ppp\odz{n}=\odz{n}\circ\orh{n}} also
\w[,]{C\sb{n-1}\rho\circ\tau=\odz{n}\circ\orh{n}\rest{\oU{n}}}
so the outer square in \wref{eqpbatt} commutes up to homotopy.  Since
$F$ is a fibration, we may change $R$ up to homotopy to make it commute on the
nose by \cite[Lemma 5.11]{BJTurnR}. The maps $R$ and $T$ then induce $S$ as
indicated. This allows us to extend \wref{eqpbatt} to the solid commuting diagram
\mydiagram[\label{eqpbat}]{
\bGs\,\amalg\,\bDs \ar[rr]\sb{S\bot e} \ar@{^{(}.>}[d]_{j}^{\simeq} &&
\bPs \ar[rr]\sb{p} \ar@{->>}[d]\sb{q} &&
\bDs \ar@{->>}[d]\sp{F} \ar@/_0.9pc/[ll]\sb{e} \ar@/_2.5pc/[llll]\sb{\inc} \\
\bEs \ar@{.>>}[urr]_{Q} && \cMs\Wp{n-1} \ar[rr]\sp{\cMs\rn{n-1}} &&
\cMs\W{n-1} \ar@/^1.1pc/[ll]\sp{\cMs\en{n-1}}
}
\noindent with \w{\bGs\amalg\bDs} strongly cofibrant.

If we now factor \w{S\bot e} as an acyclic cofibration
\w{j:\bGs\,\amalg\,\bDs\hra\bEs} followed by a fibration
\w{Q:\bEs\epic\bPs} and set \w{G:\bEs\to\cMs\Wp{n-1}} equal to
\w[,]{q\circ Q} \w{\overline{e}:\bDs\to\bEs} equal to \w[,]{j\circ\inc}
and \w{\overline{r}:\bEs\to\bDs} equal to \w[,]{p\circ Q}
we see that \w{\bEs} is strongly cofibrant (since $j$ is a cofibration),
$F$ is a levelwise fibration, and they fit into a diagram
\mydiagram[\label{eqpba}]{
\bEs \ar[rr]\sp{\overline{r}} \ar@{->>}[d]\sp{G} &&
\bDs \ar@{->>}[d]\sp{F} \ar@/_1.5pc/[ll]\sb{\overline{e}} \\
\cMs\Wp{n-1} \ar[rr]\sp{\cMs\rn{n-1}} &&
\cMs\W{n-1} \ar@/^1.1pc/[ll]\sp{\cMs\en{n-1}}
}
\noindent in which the squares commute in both horizontal directions, and
\w[.]{\overline{r}\circ\overline{e}=\Id}

By Lemma \ref{lextnat} and the fact that the map induced by the identity clearly is
another identity, we obtain an $n$-stage comparison map
\w{\Phi:\cW\to\cWp} extending the given \wwb{n-1}stage comparison map.
\end{proof}

We now prove Theorem \ref{tzigzag}, which we re-state as follows:

\begin{thm}\label{tzigzagb}
Any two sequential realizations \w{\cuW{0}} and \w{\cuW{1}} of two CW resolutions
\w{\Vud{0}} and \w[,]{\Vud{1}} for two $\bA$-equivalent spaces \w{\Yi{0}} and \w[,]{\Yi{1}}
respectively, are weakly equivalent under a (locally finite) zigzag of comparison maps,
in the sense of Definition \ref{dwesr}.
\end{thm}

\begin{proof}
Assume \w{\cuW{0}} and \w{\cuW{1}} are associated respectively to the two
CW-resolutions \w{\Vud{0}} and  \w{\Vud{1}} of the \PAal
\w[,]{\Lambda=\piA\Yi{0}\cong\piA\Yi{1}} with CW bases
\w{(\ouV{n}{i})\sb{n\in\NN}} for \w[.]{i=0,1}

By Lemma \ref{lcylind}, there is a third CW resolution \w[,]{\ppp\vare:\Vdp\to\Lambda}
with CW basis \w[,]{(\oVp{n})\sb{n\in\NN}}
equipped with algebraic comparison maps \w{\Psi\up{i}:\Vud{i}\to\Vdp} \wb[.]{i=0,1}
By Proposition \ref{pralgcomp}, there are then two fibrational sequential
realizations \w{\cWpi{i}} of \w[,]{\Vdp\to\Lambda}  for \w[,]{i=0,1} each equipped
with a comparison map \w{\Phi\up{i}:\cuW{i}\to\cWpi{i}} over \w[.]{\Psi\up{i}}
Thus we are reduced to dealing with the case where the two fibrational sequential
realizations \w{\cuW{0}} and \w{\cuW{1}} (i.e., the \w{\cWpi{i}}
just constructed) are of the same CW resolution
\w{\Vd\to\Lambda} (i.e., the above
\w[),]{\Vdp} with CW basis \w[.]{(\oV{n})\sb{n\in\NN}}
We construct a zigzag of comparison maps between them, by induction on
\w{n\geq 0} (where the case \w{n=0} is trivial):

We assume by induction the existence of a cospan of \wwb{n-1}stage trivial
comparison maps \w{\Phip{i}:\cuW{i}\to\cW} \wb{i=0,1} over \w[.]{\Id\sb{\Vd}}
By Definition \ref{dsrar}, the $n$-th stage for \w{\cuW{i}} is determined
by the choice of strongly cofibrant replacements \w{\uDs{i}}
of \w{\oW{n}\oS{n-1}} (where \w{\oW{n}} is some realization of the $n$-th
algebraic CW basis object \w[),]{\oV{n}} together with levelwise fibrations
\w{F\up{i}:\uDs{i}\epic\cMs\Wi{n-1}{i}} \wb{i=0,1} realizing the given
attaching map \w[.]{\odz{n}:\oV{n}\to C\sb{n-1}\Vd}

Again following \S \ref{crsar}.A-B, we consider the following diagram in
the projective model category of $n$-truncated chain complexes over $\C$,
in which \w{\uPs{i}} is the pullback of the lower square
\mydiagram[\label{eqpbattach}]{
\uEs{i} \ar@{.>>}[d]\sb{S\up{i}}\sp{\simeq} &&& \\
\ar @{} [drrr]|(0.23){\framebox{\scriptsize{PB}}}
{\uPs{i}} \ar[rrr]\sb{p\up{i}}\sp{\simeq} \ar@{->>}[d]\sb{q\up{i}}
&&& {\uDs{i}} \ar@{->>}[d]\sp{F\up{i}} \ar@/_2.0pc/[lll]\sp{e\up{i}}
\ar@/_1.7pc/@{_{(}.>}[ulll]\sb{\xi\up{i}}\\
\cMs\W{n-1} \ar[rrr]\sp{\cMs\rni{n-1}{i}}\sb{\simeq} &&&
\cMs\Wi{n-1}{i}~, \ar@/^1.5pc/[lll]\sp{\cMs\eni{n-1}{i}}
}
\noindent The section \w{e\up{i}} for \w{p\up{i}} is induced by the section
\w{\cMs\eni{n-1}{i}} for \w[.]{\cMs\rni{n-1}{i}} We then factor \w{e\up{i}}
as a cofibration \w{\xi\up{i}} followed by the acyclic fibration
\w{S\up{i}:\uEs{i}\epic\uPs{i}} (so \w{\uEs{i}} is a cofibrant replacement for
\w[).]{\uPs{i}}

Applying Lemma \ref{lextnat} to the following diagram:
\mydiagram[\label{eqmaparrow}]{
\uEs{i} \ar[rr]\sb{p\up{i}\circ S\up{i}}\sp{\simeq}
\ar@{->>}[d]\sb{q\up{i}\circ s\up{i}=G\up{i}} &&
\uDs{i} \ar@{->>}[d]\sp{F\up{i}} \ar@/_1.5pc/[ll]\sb{\xi\up{i}} \\
\cMs\W{n-1} \ar[rr]\sp{\cMs\rni{n-1}{i}}\sb{\simeq} &&
\cMs\Wi{n-1}{i}~, \ar@/^1.5pc/[ll]\sp{\cMs\eni{n-1}{i}}
}
\noindent we obtain $n$-stage trivial comparison maps
\w{\Phi\up{i}:\cuW{i}\to\cWpi{i}} \wb{i=0,1} extending the given ones to
\w[.]{\cW}

Note, however, that \w{G\up{0}} and \w{G\up{1}} are weakly equivalent
fibrant and cofibrant objects in the slice category
\w[,]{\ChC\sp{\leq n}\,/\,\cMs\W{n-1}} with its standard model category
structure (see \cite[Theorem 7.6.5(a)]{PHirM}). We can therefore apply
Lemma \ref{lcylin} to obtain an intermediate object $G$ in the slice category fitting
into the following diagram:
\mydiagram[\label{eqmaparrows}]{
\uEs{0} \ar@{^{(}->}[rr]\sb{f\up{0}}\sp{\simeq} \ar@{->>}[rrd]\sb{G\up{0}} &&
\bEs \ar@/^1.5pc/[rr]\sp{s\up{1}} \ar@{->>}[d]\sp{G}
\ar@/_1.5pc/[ll]\sb{s\up{0}} &&
\uEs{1} \ar@{_{(}->}[ll]\sp{f\up{1}}\sb{\simeq} \ar@{->>}[lld]\sp{G\up{1}} \\
&& \cMs\W{n-1} &&
}
\noindent in which all four triangles commute, and \w{s\up{i}\circ f\up{i}=\Id}
\wb[.]{i=0,1}

Applying Lemma \ref{lextnat} yields a new $n$-stage sequential realization \w{\cWp}
(corresponding to \w[),]{G:\bEs\epic\cMs\W{n-1}} with two new $n$-stage trivial
comparison maps \w{\ppp\Phi\up{i}:\cWpi{i}\to\cWp} \wb[.]{i=0,1}

The two composites:
\mydiagram[\label{eqnewcospan}]{
\cuW{0} \ar[r]\sp{\Phi\up{0}} & \cWpi{0} \ar[r]\sp{\ppp\Phi\up{0}} & \cWp &
\cWpi{1} \ar[l]\sb{\ppp\Phi\up{1}} & \cuW{1} \ar[l]\sb{\Phi\up{1}}
}
\noindent then yield the required cospan of $n$-stage comparison maps.
\end{proof}

\end{document}